\documentclass[final,reqno]{amsart}

\usepackage[bookmarks=false,draft=false,breaklinks,colorlinks]{hyperref}
\usepackage{mathtools,amssymb,mathrsfs}
\usepackage{fullpage,enumitem}
\usepackage[only llbracket,rrbracket,longmapsfrom]{stmaryrd}
\usepackage{braket}
\usepackage{tikz-cd}
\usepackage{showkeys}
\usepackage[l2tabu,orthodox]{nag}
\usepackage[all, warning]{onlyamsmath}

\numberwithin{equation}{subsection}
\allowdisplaybreaks

\usepackage{enumitem}
\setenumerate{label=(\arabic*),nosep}
\setitemize{nosep}
\newlist{clist}{enumerate}{1}
\setlist*[clist]{label=(\roman*), nosep}

\usepackage{cleveref}
\crefname{thm}{Theorem}{Theorems}
\crefname{asm}{Assumption}{Assumptions}
\crefname{cor}{Corollary}{Corollaries}
\crefname{dfn}{Definition}{Definitions}
\crefname{fct}{Fact}{Facts}
\crefname{lem}{Lemma}{Lemmas}
\crefname{mth}{Theorem}{Theorem}
\crefname{prp}{Proposition}{Propositions}
\crefname{rmk}{Remark}{Remarks}
\crefname{eg}{Example}{Examples}
\crefname{figure}{Figure}{Figures}
\crefname{table}{Table}{Tables}
\crefname{section}{\S\!}{\S\S\!}
\crefname{subsection}{\S\!}{\S\S\!}
\crefname{subsubsection}{\S\!}{\S\S\!}
\crefname{equation}{}{}

\usepackage{amsthm}
\theoremstyle{definition}
\newtheorem{thm}{Theorem}[subsection]
\newtheorem{dfn}[thm]{Definition}
\newtheorem{lem}[thm]{Lemma}
\newtheorem{prp}[thm]{Proposition}

\newtheorem{rmk}[thm]{Remark}
\newtheorem{eg}[thm]{Example}

\newtheorem*{rmk*}{Remark}

\newcommand{\ol}{\overline}

\newcommand{\wt}{\widetilde}
\newcommand{\bl}{\bullet}

\newcommand{\ve}{\varepsilon}
\newcommand{\pdd}{\partial}
\newcommand{\ceq}{\coloneqq} 
\newcommand{\bs}{\mathbin{\setminus}}

\newcommand{\xr}{\xrightarrow}
\newcommand{\xrr}[1]{\xrightarrow{\, #1 \, }}
\newcommand{\inj}{\hookrightarrow}

\newcommand{\lto}{\longrightarrow}

\newcommand{\sto}{\xr{\sim}}
\newcommand{\lsto}{\xrr{\sim}}
\newcommand{\linj}{\lhook\joinrel\longrightarrow}

\newcommand{\mto}{\mapsto}
\newcommand{\lmto}{\longmapsto}
\newcommand{\siff}{\mathbin{\Leftrightarrow}}

\newcommand{\ev}{\ol{0}}
\newcommand{\od}{\ol{1}}
\newcommand{\Aff}{\mathrm{Aff}}
\newcommand{\vac}{\ket{0}} 

\newcommand{\bbA}{\mathbb{A}}
\newcommand{\bbG}{\mathbb{G}}
\newcommand{\bbN}{\mathbb{N}}
\newcommand{\bbZ}{\mathbb{Z}}
\newcommand{\bbK}{\mathbb{K}}
\newcommand{\clD}{\mathcal{D}}
\newcommand{\clH}{\mathcal{H}}
\newcommand{\clO}{\mathcal{O}}
\newcommand{\frg}{\mathfrak{g}}
\newcommand{\frS}{\mathfrak{S}}
\newcommand{\shA}{\mathscr{A}}
\newcommand{\shO}{\mathscr{O}}
\newcommand{\shD}{\mathscr{D}}
\newcommand{\rmT}{\mathrm{T}}

\newcommand{\ch}{\textup{ch}}
\newcommand{\tred}{\textup{red}}
\newcommand{\ash}{\mathord{\text{\textexclamdown}}} 

\newcommand{\cC}{\mathsf{C}}
\newcommand{\cOp}{\mathsf{Op}}
\newcommand{\coOp}{\mathsf{Op}_{\ev}}
\newcommand{\cSMod}{\mathop{\mathsf{Mod}}{\frS}}
\newcommand{\cLin}{\mathop{\mathsf{Mod}}{\bbK}}
\newcommand{\cgLin}{\mathop{\mathsf{grMod}}{\bbK}}

\DeclareMathOperator{\ad}{ad}
\DeclareMathOperator{\id}{id}

\DeclareMathOperator{\MC}{MC}

\DeclareMathOperator{\oC}{\mathcal{C}}
\DeclareMathOperator{\oP}{\mathcal{P}}
\DeclareMathOperator{\oQ}{\mathcal{Q}}
\DeclareMathOperator{\oT}{\mathcal{T}}
\DeclareMathOperator{\oAss}{\mathcal{A}\mathit{ssoc}}
\DeclareMathOperator{\oCom}{\mathcal{C}{\kern-0.1em}\mathit{om}}
\DeclareMathOperator{\oHom}{\mathcal{H}{\kern-0.1em}\mathit{om}}
\DeclareMathOperator{\oLie}{\mathcal{L}{\kern-0.5pt}\mathit{ie}}
\DeclareMathOperator{\sgn}{sgn}
\DeclareMathOperator{\End}{End}
\DeclareMathOperator{\Hom}{Hom}
\DeclareMathOperator{\Img}{Im}
\DeclareMathOperator{\Ker}{Ker}
\DeclareMathOperator{\Res}{Res}
\DeclareMathOperator{\Spec}{Spec}
\DeclareMathOperator{\crMod}{\mathsf{Mod}}
\DeclareMathOperator{\Cas}{Cas}
\DeclareMathOperator{\Der}{Der}
\DeclareMathOperator{\Ind}{Inder}

\newcommand{\abs}[1]{\left| #1 \right|}
\newcommand{\rst}[2]{\left. #1 \right|_{#2}}

\newcommand{\oPch}[1]{\mathcal{P}^{\textup{ch}}_{#1}}
\newcommand{\oPchB}[1]{\mathop{\mathcal{P}^{\textup{ch}N_{\bullet}=N}_{#1}}}
\newcommand{\oPchW}[1]{\mathop{\mathcal{P}^{\textup{ch}N_W=N}_{#1}}}
\newcommand{\oPchK}[1]{\mathop{\mathcal{P}^{\textup{ch}N_K=N}_{#1}}}
\newcommand{\oPchWT}[1]{\mathop{\mathcal{P}^{\textup{ch}N_W=N, \mathrm{T}}_{#1}}}
\newcommand{\oPchKT}[1]{\mathop{\mathcal{P}^{\textup{ch}N_K=N, \mathrm{T}}_{#1}}}
\newcommand{\oCh}[1]{\mathop{\mathcal{C}{\kern-0.6pt}hom_{#1}}}
\newcommand{\oChW}[1]{\mathop{\mathcal{C}{\kern-0.6pt}hom^{N_W=N}_{#1}}}
\newcommand{\oChK}[1]{\mathop{\mathcal{C}{\kern-0.6pt}hom^{N_K=N}_{#1}}}

\begin{document}

\title{Algebraic operad of SUSY vertex algebra}
\author{Yusuke Nishinaka, Shintarou Yanagida}
\date{29 September 2022; revised 30 April 2023} 
\address{Graduate School of Mathematics, Nagoya University.
 Furocho, Chikusaku, Nagoya, Japan, 464-8602.}
\email{m21035a@math.nagoya-u.ac.jp, yanagida@math.nagoya-u.ac.jp}
\thanks{Y.N. is supported by JSPS Research Fellowship for Young Scientists (No.\ 23KJ1120). 
 S.Y.\ is supported by JSPS KAKENHI Grant Number 19K03399.}

\begin{abstract}
We introduce algebraic operads $\mathcal{P}^{\textup{ch}N_W=N}$ and $\mathcal{P}^{\textup{ch}N_K=N}$ encoding the structures of $N_W=N$ and $N_K=N$ SUSY vertex algebras, and study the corresponding cohomology theory. Our operad is a SUSY analogue of the operad $P^{\textup{ch}}$ introduced by Bakalov, De Sole, Heluani and Kac (2019) as a purely algebraic translation of the chiral operad of Beilinson and Drinfeld.
\end{abstract}

\maketitle
{\small \tableofcontents}

\setcounter{section}{-1}
\section{Introduction}\label{s:0}

\subsubsection*{Vertex algebras, chiral algebras and chiral operad}

Nowadays vertex algebras appear in various branches of mathematics, mathematical physics and theoretical physics. They were introduced by Borcherds \cite{Bo} as an algebraic framework of the operator algebra structure appearing in the chiral part of two-dimensional conformal field theory. This framework fits very well with the representation theory of infinite-dimensional algebras, and has made enormous progress since its appearance until today. 

On the other hand, at least since the mid-1980s, it was expected that there should be a geometric framework for chiral conformal field theory. Along this expectation, together with the viewpoint of the geometric Langlands program, Beilinson and Drinfeld started in the early 1990s to develop the algebro-geometric framework, which resulted in their book \cite{BD} of chiral algebras. At present, both vertex algebras and chiral algebras are used in the study of algebraic and geometric representation theory, and in the theoretical physics of conformal field theories and their cousins. Let us cite \cite{FBZ} for the reference that explains both the framework of vertex algebras and chiral algebras.

The theory of chiral algebras is a flexible framework of chiral conformal field theory, allowing geometrical studies. But it has one drawback: the definition is built on various complicated structures using the theory of $\clD$-modules and that of operads, so that it is non-trivial for non-experts to check the relation \cite[0.15]{BD} between chiral algebras and Borcherds' axiom of vertex algebras. A feature of the theory is that a chiral algebra is defined to be a Lie algebra in a non-standard monoidal structure. More precisely, in the category of $\clD$-modules on a smooth algebraic curve, Beilinson and Drinfeld introduced the chiral tensor product and the \emph{chiral operad}. The latter object might best be called a sheaf operad, which is neither an algebraic operad nor a topological operad appearing commonly. Then a (non-unital) chiral algebra is defined to be an operad morphism from the Lie operad to the chiral operad. Starting from such a definition, it is not an easy task to recover Borcherds' axiom of vertex algebras.  

In the late 2010s, Bakalov, De Sole, Heluani and Kac \cite{BDHK} started their investigation of the cohomology theory of vertex algebras, which continues up to present \cite{BDHK2,BDK20,BDK21,BDKV21}. As far as we understand, their study sits in the sequence of the investigation of Lie conformal algebras, vertex Poisson algebras and their cohomology theories which started in the late 1990s (see \cite{BKV,DK13} for example). Among several results in \cite{BDHK}, they constructed an algebraic operad $P^{\ch}_V$ for each linear superspace $V$, and showed that an operad morphism from the Lie operad $\oLie$ to $P^{\ch}_V$ corresponds bijectively to a structure of the vertex superalgebra on $V$:
\begin{align}\label{eq:0:Pch}
 \Hom_{\coOp}(\oLie,P^{\ch}_V) \lsto \{\text{vertex algebra structures on $V$}\}.
\end{align}
Here $\coOp$ denotes the category of operads (see \cref{ss:1:op}, \eqref{eq:1:coOp-cOp} for the convention).
Thus, $P^{\ch}$ (suppressing the linear space $V$) is a purely algebraic counterpart of the chiral operad of Beilinson and Drinfeld. The standard operadic cohomology theory then gives a natural cohomology theory of a vertex algebra, and one can study the relation between the obtained cohomology and that of vertex Poisson algebras. 

\subsubsection*{SUSY vertex algebras and SUSY chiral operad}

This note is written in a simple motivation to give an analogue of the algebraic operad $P^{\ch}$ which will encode the structure of a \emph{SUSY vertex algebra}. The latter object is a superfield analogue of a vertex algebra, introduced by Heluani and Kac \cite{HK}. The framework of SUSY vertex algebras looks very natural from the point of view of superconformal field theory, and there are many examples, especially of a geometric nature. However, as far as we know, mathematical investigations of SUSY vertex algebras have not been pursued as much as those of ordinary vertex (super)algebras.

There are two classes of SUSY vertex algebras: $N_W=N$ and $N_K=N$ SUSY vertex algebras, which originate from the classification of superconformal Lie algebras \cite{KL}. Accordingly, our argument is divided into two parts. We will introduce the algebraic operad $\oPchW{}$ in \cref{s:W} and $\oPchK{}$ in \cref{s:K}, which encode the structure of an $N_W=N$ and $N_K=N$ SUSY vertex algebra, respectively. These \cref{s:W} and \cref{s:K} are the main body of this note.

Our argument basically follows the non-SUSY case given in \cite[\S5-\S7]{BDHK}. We first construct the operads $\oChW{}$ and $\oChK{}$ of SUSY Lie conformal algebras, and then construct the operads $\oPchW{}$ and $\oPchK{}$ of SUSY vertex algebras. Note that although the outline of our construction is quite the same as \cite{BDHK}, we need to make some technical modifications to treat the SUSY case appropriately. For example, in \cref{ss:W:VA} of the construction of the operad $\oPchW{}$, we have to treat carefully the integral formulation of the SUSY vertex algebras, and give detailed calculations in \cref{dfn:W:indef}--\cref{prp:W:Jqas} and \cref{lem:W:circ}--\cref{lem:W:ResFn}, which we could not find in the literature.

The operads $\oPchW{}$ and $\oPchK{}$ give the cohomology complexes of SUSY vertex algebras with coefficients in their modules, whose definition and analysis is given in \cref{s:coh}.

\subsubsection*{Summary of results}

The main objects of this note are the $N_W=N$ SUSY chiral operad $\oPchW{}$ introduced in \cref{dfn:W:Pch} and the $N_K=N$ SUSY chiral operad $\oPchK{}$ introduced in \cref{dfn:K:Pch}. The main statements are \cref{thm:W:VA} and \cref{thm:K:VA} which establish the following bijection \eqref{eq:0:WK-Pch}\footnote{Taking the odd part $\Hom_{\cOp}(\cdot,\cdot)_{\od}$, not the even part, in the left hand side of \eqref{eq:0:WK-Pch} is due to making our convention compatible with that in \cite{BDHK} when we set $N=0$.} between the odd Maurer-Cartan solutions of the operads $\oPchW{},\oPchK{}$ and the structures of the non-unital $N_W=N$ and $N_K=N$ SUSY vertex algebras:
\begin{align}\label{eq:0:WK-Pch}
 \Hom_{\cOp}(\oLie,\oPchB{\Pi^{N+1}V})_{\od} \lsto 
 \{\text{$N=N_{\bl}$ SUSY vertex algebra structures on $(V,\nabla)$}\}, \quad
 \bl = W \text{ or } K. 
\end{align}
Here $\cOp$ denotes the supercategory of superoperads, $\Hom_{\cOp}(\cdot,\cdot)_{\od}$ denotes the linear space of odd morphisms of superoperads (see \cref{ss:1:op}, \eqref{eq:1:coOp-cOp} for these notations), $(V,\nabla)$ denotes an $\clH_W$-supermodule with underlying linear superspace $V$ and $\nabla=(T,S_1,\dotsc,S_N)$ (see \cref{dfn:W:clHW}) and $\Pi$ denotes the parity change functor. Setting $N=0$ in \eqref{eq:0:WK-Pch}, we have $\oPchB{\Pi^{N+1}V}=P^{\ch}_{\Pi V}$ and $\Hom_{\cOp}(\oLie,\oPchB{\Pi^{N+1}V})_{\od}=\Hom_{\cOp}(\oLie,P^{\ch}_{\Pi V})_{\od}=\Hom_{\coOp}(\oLie,P^{\ch}_V)$, and thus we recover the bijection \eqref{eq:0:Pch} in the non-SUSY case. Therefore, \eqref{eq:0:WK-Pch} is a SUSY analogue of \eqref{eq:0:Pch}.

Another block of main results is the cohomology theory given in \cref{s:coh}. The cohomology complex for an $N_W=N$ SUSY vertex algebra is given in \cref{dfn:W:VAmod}, and the interpretation of the low degree cohomology is given in \cref{thm:W:coh}. We have a simple analogue for the $N_K=N$ case, which is briefly explained in \cref{ss:coh:K}.

\subsubsection*{Organization}

Let us explain the organization. 
\begin{itemize}
\item
\cref{s:op} is a preliminary section on superobjects, superoperads and operadic cohomology theory. In particular, in \cref{ss:1:Kos} we will explain the convolution Lie superalgebra $\frg(\oP,\oQ)$ for superoperads $\oP$ and $\oQ$ which encodes superoperad morphisms $\oP \to \oQ$, and explain in \cref{ss:1:LQ} that the Lie superalgebra $\frg(\oLie,\oQ)$ associated with a superoperad $\oQ$ introduced in \cite[\S 3]{BDHK} is nothing but the convolution Lie superalgebra in our sense (\cref{prp:1:LP}).

\item
\cref{s:W} is devoted to the study of the algebraic operad encoding the structure of $N_W=N$ SUSY vertex algebras. We start with the preliminary study of polynomial superalgebras in \cref{ss:W:poly}, which will be ``the base algebra'' of our operad.
Then following the line of \cite[\S5, \S6]{BDHK}, we introduce an operad $\oChW{}$, which encodes the structure of $N_W=N$ SUSY Lie conformal algebras in \cref{ss:W:LCA}, and an operad $\oPchW{}$ encoding that of $N_W=N$ SUSY vertex algebras, which is one of the main objects of this note, in \cref{ss:W:VA}. In \cref{ss:W:CA}, we relate the operad $\oPchW{}$ to the $N_W=N$ SUSY analogue of the Beilinson-Drinfeld chiral operad, following the line of the non-SUSY case \cite[Appendix A]{BDHK}.

\item
\cref{s:K} treats the $N_K=N$ SUSY vertex algebras, and the materials are delivered in the same order as in the previous \cref{s:W}.

\item
The operads $\oPchW{}$ and $\oPchK{}$ naturally give the cohomology complexes of SUSY vertex algebras with coefficients in their modules, which will be studied in \cref{s:coh}. We start with the review of the operadic cohomology theory in \cref{ss:coh:CE}. Then we discuss the cohomology theory for $N_W=N$ SUSY vertex algebras in \cref{ss:coh:W}, and the theory of $N_K=N$ SUSY vertex algebras in \cref{ss:coh:K}.
We conclude this note with some remarks in \cref{ss:cnc}.
\end{itemize}

\subsubsection*{Global notation}

\begin{itemize}
\item
The symbol $\bbN$ denotes the set $\{0,1,2,\dotsc\}$ of non-negative integers. 

\item 
For a positive integer $m$, the symbol $[m]$ denotes the set $\{1,2,\dotsc,m\} \subset \bbZ$. 

\item
For $m,n \in \bbN$, the symbol $\bbN^m_n$ is defined to be 
\begin{align}\label{eq:0:Nkn}
 \bbN^m_n \ceq \{(n_1,\ldots,n_m) \in \bbN^m \mid n_1+\dotsb+n_m=n\}.
\end{align}

\item
For a finite set $S$, we denote by $\# S$ the number of elements in $S$.

\item 
The word `ring' or `algebra' means a unital associative one unless otherwise specified. 

\item
The symbol $\pdd_x$ denotes the partial differential $\frac{\pdd}{\pdd x}$ with respect to $x$.

\item
The symbol $\delta_{m,n}$ denotes the Kronecker delta: $\delta_{m,n}=1$ if $m=n$ and $\delta_{m,n}=0$ otherwise. We will use it in the case $m,n$ are numbers and in the case $m,n$ are sets.
\end{itemize}

\section{Lie algebra structures on operads}\label{s:op}

This section gives an overview of the \emph{universal Lie superalgebra associated to an operad} introduced in \cite[\S3]{BDHK}. Although the definitions and arguments therein is clear, we feel that the reader would require extra explanation on the motivation and backgrounds, and this section is designed to give backgrounds on the theory of algebraic operads and to give some complementary comments on loc.\ cit. We refer to \cite{LV} for the comprehensive reference of algebraic operads, and to \cite{DK13} for the preceding study of the universal Lie superalgebra.

In the starting \cref{ss:1:op}, we introduce notation and terminology on operads, and explain some fundamental notion. We will work in the super setting, following the presentation of \cite[\S2, \S3]{BDHK}. 

In the next \cref{ss:1:Kos}, we give a brief review of the theory of operadic Koszul duality and operadic deformation theory. In particular, for a binary quadratic operad $\oP$ and an operad $\oQ$, we introduce in \eqref{eq:1:gPQ} the convolution Lie superalgebra $\frg(\oP,\oQ)$, which concerns the $\oP$-algebra structure on $\oQ$ as \eqref{eq:1:MCQ}.

In the last \cref{ss:1:LQ}, applying the arguments so far to the case $\oP=\oLie$ (the Lie operad), we recover the universal Lie superalgebra associated to $\oQ$ in the sense of \cite{BDHK}. See \cref{dfn:1:LP} and \cref{prp:1:LP} for the precise statement.

In this section, all the objects (linear spaces, algebras and so on) are defined over the base field $\bbK$ of characteristic $0$ unless otherwise stated. The unit of the field $\bbK$ is denoted by $1_\bbK$.

\subsection{Operads and algebras over operads in super setting}\label{ss:1:op}

Here we collect from \cite{LV} and \cite{BDHK} some basic notions and symbols on algebraic operads in the super setting.

Here is the list of our super terminology, which is more or less the standard one and can be found in \cite{DM} for example.
\begin{itemize}
\item 
We write $\bbZ_2 \ceq \bbZ/2\bbZ$, and denote by $\ol{i} \in \bbZ_2$ the parity of $i \in \bbZ$ . We also define $(-1)^{\ol{i}} \ceq (-1)^i$. 

\item 
A $\bbZ_2$-graded linear space $V$ is called a \emph{linear superspace}. The grading is expressed by $V=V_{\ev}\oplus V_{\od}$. An element of $V_{\ev} \bs \{0\}$ or $V_{\od} \bs \{0\}$ is called \emph{even} or \emph{odd}, respectively. An element of $(V_{\ev} \cup V_{\od}) \bs \{0\}$ is called \text{pure}. The parity $p(x) \in \bbZ_2$ of a pure element $x$ is defined by 
\[
 p(x) \ceq 
 \begin{cases} \ev & (x \in V_{\ev}), \\ \od & (x \in V_{\od}). \end{cases}
\]
Whenever we write $p(x)$ in the text, we assume $x$ is pure. 

\item
For linear superspaces $V$ and $W$, we denote by $\Hom(V,W)$ the linear superspace of linear maps from $V$ to $W$. The $\bbZ_2$-grading is denoted by 
\[
 \Hom(V,W)=\Hom(V,W)_{\ev} \oplus \Hom(V,W)_{\od},
\]
and an element of $\Hom(V,W)_{\ev}$ (resp.\ $\Hom(V,W)_{\od}$) is called an even (resp.\ odd) linear map. Sometimes we denote $\Hom_{\bbK}(V,W) \ceq \Hom(V,W)$ to stress the base field $\bbK$. We also denote by $\End(V)=\End(V)_{\ev} \oplus \End(V)_{\od}$ the linear superspace of linear endomorphisms of $V$. 

\item
A $\bbZ_2$-graded $\bbK$-linear category $\cC$ will be called a \emph{supercategory}.
Thus, for any objects $M$ and $N$ of $\cC$, the morphism set $\Hom_{\cC}(M,N)$ is a $\bbZ_2$-graded $\bbK$-linear space, and the $\Hom_{\cC}(M,N)_{\ev}$ and $\Hom_{\cC}(M,N)_{\od}$ denote the $\bbK$-linear spaces of even and odd morphisms, respectively.

\item
The supercategory of linear superspaces and linear maps is denoted by $\cLin$.
It has the standard monoidal structure, and the tensor product is written by $\otimes \ceq \otimes_{\bbK}$.
See \cite[(1.1)]{DM} for the sign rule (often called the \emph{Koszul sign rule}) of this tensor product.

\item 
The \emph{parity change functor} is denoted by $\Pi$. Thus, for a linear superspace $V$, we have the linear superspace $\Pi V$ with $(\Pi V)_{\ev} = V_{\od}$ and $(\Pi V)_{\od} = V_{\ev}$. 

\item 
A $\bbZ_2$-graded algebra $A$ is called a \emph{superalgebra}, and a linear superspace equipped with a left (resp.\ right) $\bbZ_2$-graded $A$-module structure is called a \emph{left} (resp.\ \emph{right}) \emph{$A$-supermodule}. A \emph{morphism $M \to N$ of left $A$-supermodules $M$ and $N$} means a linear map consistent with the supermodule structures of $M$ and $N$. The linear superspace of all morphisms of left $A$-supermodules is denote by
\[
 \Hom_A(M,N) = \Hom_A(M,N)_{\ev} \oplus \Hom_A(M,N)_{\od}.
\]
Thus, left $A$-supermodules form a supercategory whose morphism superspace is given by $\Hom_A(\cdot,\cdot)$.

\item
A superalgebra $A$ is called \emph{commutative} if the multiplication satisfies $x y=(-1)^{p(x) p(y)} y x$ for any pure $x,y \in A$.
\end{itemize}

Next we give some symbols for symmetric groups and their modules.
\begin{itemize}
\item
For $n \in \bbN$, we denote by $\frS_n$ the $n$-th symmetric group with the convention $\frS_0 \ceq \{e\} = \frS_1$. We consider the group algebra $\bbK[\frS_n]$ as a purely even superalgebra. 

\item 
For $m,n \in \bbN$ and $\nu=(n_1,\ldots,n_m) \in \bbN^m_n$ (see \eqref{eq:0:Nkn}), we denote 
\[
 \frS_\nu \ceq \frS_{n_1}\times \cdots \times \frS_{n_m},
\]
which is regarded as a subgroup of $\frS_n$. 

\item
Let $V$ be a linear superspace. For $n \in \bbN$, the superspace $V^{\otimes n}$ carries a structure of a left $\frS_n$-supermodule by letting $\sigma \in \frS_n$ act on $v_1 \otimes \dotsb \otimes v_n \in V^{\otimes n}$ by
\begin{align}\label{eq:1:Sn-Vn}
 \sigma(v_1\otimes \cdots\otimes v_n) \ceq 
 \prod_{\substack{i<j \\ \sigma(i)>\sigma(j)}} (-1)^{p(v_i)p(v_j)} 
 \cdot v_{\sigma^{-1}(1)}\otimes \cdots\otimes v_{\sigma^{-1}(n)}.
\end{align}
In what follows, we always consider $V^{\otimes n}$ with this left $\frS_n$-supermodule structure unless otherwise specified. 
\end{itemize}

Now we introduce the notion of $\frS$-modules in the super setting. See \cite[\S5.1]{LV} for the detail in the non-super case.
\begin{itemize}
\item
An \emph{$\frS$-supermodule} $M=\bigl(M(n)\bigr)_{n \in \bbN}$ is a collection of right $\frS_n$-supermodules $M(n)$. The index $n$ is called the \emph{arity}. A \emph{morphism $f\colon M \to N$ of $\frS$-supermodules} is a family of morphisms $f_n\colon M(n) \to N(n)$ of right $\frS_n$-supermodules. The resulting supercategory of $\frS$-supermodules is denoted by $\cSMod$.

\item
For $m,n \in \bbN$, $\nu=(n_1,\ldots,n_m)\in \bbN^m_n$, and an $\frS$-supermodule $M$, we define a right $\frS_\nu$-supermodule
\begin{align*}
 M(\nu)\ceq M(n_1)\otimes \cdots \otimes M(n_m). 
\end{align*}

\item
The supercategory $\cSMod$ has a monoidal structure whose tensor product $M \circ N = \bigl((M \circ N)(n)\bigr)_{n \in \bbN}$ is given by 
\begin{align}\label{eq:pre:op:circ}
 (M\circ N)(n) \ceq \bigoplus_{m \in \bbN} \Bigl(M(m) \otimes_{\bbK[\frS_m]}
\Bigl(\bigoplus_{\nu\in\bbN^m_n}N(\nu)\otimes_{\bbK[\frS_\nu]}\bbK[\frS_n]\Bigr)\Bigr)
\end{align}
and whose unit object is $I \ceq (\delta_{n,1}\bbK)_{n \in \bbN}=(0,\bbK,0,0,\dotsc)$.
\end{itemize}

In the non-super setting \cite{LV}, an operad is defined to be a monoid object in the monoidal category $(\cSMod,\circ,I)$, i.e., it consists of an $\frS$-module $\oP=\bigl(\oP(n)\bigr)_{n \in \bbN}$ and two morphisms $\gamma\colon \oP \circ \oP \to \oP$ and $\eta\colon I \to \oP$ of $\frS$-modules, which should satisfy the standard axioms of monoids. Unraveling the axioms in the super setting, we obtain the following definition. Hereafter we use the symbol $[k] \ceq \{1,2,\dotsc,k\}$ for $k \in \bbN$ with the convention $[0] \ceq \emptyset$.

\begin{dfn}[{\cite[\S3.1]{BDHK}}]\label{dfn:1:op}
Consider a triple $(\oP,\gamma,1)$ consisting of:
\begin{itemize}
\item 
An $\frS$-supermodule $\oP$, i.e, a map $n \mto \oP(n)$ associating a right $\frS_n$-supermodule to each $n\in\bbN$. 

\item 
A family $\gamma$ of even linear maps 
\begin{align*}
 \gamma_\nu\colon \oP(m) \otimes \oP(\nu) \lto \oP(n) \quad (m,n \in \bbN, \, \nu \in \bbN^m_n)
\end{align*}
satisfying
\begin{align}\label{eq:1:gamma_nu}
 \gamma_\nu(f^\sigma\otimes g_1^{\tau_1}\otimes \cdots \otimes g_m^{\tau_m})
=\gamma_{\sigma\nu}(f\otimes \sigma(g_1\otimes \cdots \otimes g_m))^{\sigma_\nu(\tau_1, \ldots, \tau_m)}
\end{align} 
for $\sigma\in\frS_m$, $(\tau_1, \ldots, \tau_m)\in\frS_\nu$ and $f\in\oP(m)$, $g_1\otimes \cdots \otimes g_m\in\oP(\nu)$. 
Here $f^\sigma$ denotes the right action of $\sigma\in\frS_m$ on $f \in \oP(m)$. 
Also, for $\nu=(n_1, \ldots, n_m)\in\bbN^m_n$, we set $\sigma\nu\ceq (n_{\sigma^{-1}(1)}, \ldots, n_{\sigma^{-1}(m)})$, and define $\sigma_\nu\in\frS_n$ by 
\begin{align*}
 \sigma_\nu(k) \ceq n_{\sigma^{-1}(1)}+\dotsb+n_{\sigma^{-1}(\sigma(i)-1)}+j
\end{align*}
for each $k\in[n]$ written uniquely as $k=n_1+\cdots n_{i-1}+j$ with $i \in [m]$ and $j \in [n_i]$. 
For simplicity, we denote 
\[
 f \circ (g_1 \odot \cdots \odot g_m) \ceq \gamma_{\nu}(f \otimes g_1 \otimes \dotsb \otimes g_m)
\]
for $f \in \oP(m)$ and $g_1 \otimes \cdots \otimes g_m \in \oP(\nu)$.

\item 
An even element $1 \in \oP(1)_{\ol{0}}$.  
\end{itemize}

Such a triple $(\oP,\gamma,1)$ is called a \emph{superoperad} if the following conditions are satisfied.
\begin{clist}
\item
Let us given any $l, m, n\in \bbN$, $\mu=(m_1,\dotsc,m_l) \in \bbN^l_m$ and $\nu\in\bbN^m_n$.
For any $f\in\oP(l)$, $g_1\otimes \cdots \otimes g_l\in\oP(\mu)$ and $h_1\otimes \cdots \otimes h_m\in\oP(\nu)$, we have 
\begin{align*}
&\bigl(f\circ (g_1\odot \cdots \odot g_l)\bigr)\circ (h_1\odot \cdots \odot h_m)\\
&=\pm f\circ \bigl((g_1\circ (h_1\odot\cdots \odot h_{M_1}))\odot \cdots \odot (g_l\circ (h_{M_{l-1}+1}\odot \cdots \odot h_{M_l}))\bigr). 
\end{align*}
Here we set $M_i\ceq m_1+\cdots +m_i$ for each $i\in[l]$ and 
\begin{align*}
\pm\ceq \prod_{1\le i<j\le l}\prod_{k=M_{i-1}+1}^{M_i}(-1)^{p(g_j)p(h_k)}. 
\end{align*}

\item 
For $n\in\bbN$ and $f\in \oP(n)$, we have $f\circ (1\odot \cdots \odot 1)=1\circ f=f$. 
\end{clist}
\end{dfn}

Some additional terminology is in order.
\begin{itemize}
\item 
For a superoperad $(\oP,\gamma=\{\gamma_\nu\}_{\nu},1)$, the linear maps $\gamma_\nu$ are called the \emph{composition maps} and $1$ is called the \emph{unit}. We will often say that $\oP$ is a superoperad, without mentioning $\gamma$ and $1$. 

\item
A superoperad $\oP$ is called \emph{even} if the linear superspace $\oP(n)$ is even for each $n \in \bbN$ and all the composition maps $\gamma_\nu$ are linear. An even superoperad is also called an \emph{operad} or a \emph{non-super operad}.

\item
Let $\oP$ be a superoperad. For each $m,n \in \bbN$ and $i \in [m]$, the linear map 
\begin{align}\label{eq:1:circ_i}
 \circ_i\colon \oP(m)\otimes \oP(n) \lto \oP(m+n-1), \quad 
 f\circ_ig \ceq f\circ (1\odot \cdots \odot 1 \odot \overset{i}{\check{g}}\odot 1\odot \cdots \odot 1) 
\end{align}
is called the \emph{infinitesimal composition}.
Note that the infinitesimal compositions are even.
\end{itemize}

Next we recall the notion of algebras over an operad in the super setting. 
\begin{itemize}
\item 
A \emph{morphism $\alpha\colon \oP \to \oQ$ of superoperads} is a morphism of $\frS$-supermodules which is compatible with the monoid object structures. See \cite[\S5.2.1]{LV} for the detail in the non-super setting. Superoperads and their morphisms form a sub-supercategory $\cOp$ of $\cSMod$, which is called the \emph{supercategory of superoperads}. It has already appeared at \eqref{eq:0:WK-Pch} in the introduction. 
Also we denote by $\coOp$ the category of non-super operads and even morphisms between them, which has already appeared at \eqref{eq:0:Pch}. Thus, for non-super operads $\oP$ and $\oQ$, we have 
\begin{align}\label{eq:1:coOp-cOp}
 \Hom_{\coOp}(\oP,\oQ) = \Hom_{\cOp}(\oP,\oQ)_{\ev}.
\end{align}

\item 
For a linear superspace $V$, we denote by $\oHom_V$ the \emph{endomorphism superoperad}, which is a straightforward analogue of the endomorphism operad in the non-super setting \cite[\S5.2.11]{LV}. Explicitly, as an $\frS$-supermodule, it is given by 
\begin{align}\label{eq:1:HomV}
 \oHom_V(n) \ceq \Hom(V^{\otimes n},V)
\end{align}
with right $\frS_n$-module structure induced by the left $\frS_n$-action on $V^{\otimes n}$ by permuting tensor factors \eqref{eq:1:Sn-Vn}. The composition map $\gamma=\{\gamma_\nu\}_\nu$ is given naturally by the composition of linear maps. See \cite[\S5.2.11]{LV} for the detail.

\item
For a superoperad $\oP$ and a linear superspace $V$, a \emph{$\oP$-algebra structure on $V$} 
is a morphism $\alpha\colon \oP \to \oHom_V$ of superoperads\footnote{Here we can restrict $\alpha$ to be even, as explained in the classical examples $\oCom$, $\oAss$ and $\oLie$. We include odd morphisms to establish the main bijection \eqref{eq:0:WK-Pch} in such a way that the convention of our SUSY chiral operad $\oPchB{}$ is compatible with the chiral operad $\oPch{}$ in \cite{BDHK}, as explained in the footnote to \eqref{eq:0:WK-Pch}.}.
See \cite[\S5.2.3, Proposition 5.2.2]{LV} for another equivalent definition.
Such a pair $(V,\alpha)$ will be called a \emph{$\oP$-algebra}.
\end{itemize}

Let us recall the formulation of classical algebraic structures in terms of the operad theory.
\begin{itemize}
\item
Let $\oCom$ be the \emph{commutative operad} \cite[\S 5.2.10, \S 13.1]{LV}. We have $\oCom(n)=\bbK$, the trivial representation of $\frS_n$, for each $n \in \bbN$. For a linear superspace $V$, a $\oCom$-algebra $(V,\alpha)$ with even $\alpha\colon \oCom \to \oHom_V$ is nothing but a commutative $\bbK$-superalgebra structure on $V$ whose multiplication $\mu\colon V^{\otimes 2}\to V$ is given by the image of $1_\bbK \in \bbK=\oCom(2)$ under $\alpha_2\colon \oCom(2) \to \oHom_V(2)=\Hom(V^{\otimes 2},V)$, i.e., $\mu=\alpha_2(1_\bbK)$. Note that the evenness of $\alpha$ guarantees the parity-preserving property of the multiplication $\mu$.

\item 
Let $\oAss$ be the \emph{associative operad} \cite[\S 5.2.10, Chap.\ 9]{LV}. We have $\oAss(n)=\bbK[\frS_n]$, the regular representation of $\frS_n$. An $\oAss$-algebra structure $\alpha$ on a linear superspace $V$ is determined by the image of $e \in \oAss(2)=\bbK[\frS_2]=\bbK e+\bbK(1,2)$. If $\alpha$ is even, then the element $\mu=\alpha_2(e) \in \oHom_V(2)=\Hom(V^{\otimes 2},V)$ is nothing but the multiplication in the corresponding associative $\bbK$-superalgebra structure on $V$. Conversely, the $\bbK$-superalgebra structure on $V$ corresponds to an even morphism $\alpha \in \Hom_{\cOp}(\oAss,\oHom_V)_{\ev}$.

\item
The Lie algebra structure is encoded by the \emph{Lie operad} $\oLie$ \cite[\S13.2]{LV}. To describe its structure, it is convenient to introduce the notion of \emph{quadratic operad}, and we leave the explanation to the next \cref{ss:1:Kos}.
\end{itemize}

\subsection{Koszul operad theory and convolution Lie superalgebra}\label{ss:1:Kos}

The operads $\oCom$ and $\oAss$ are \emph{quadratic operads}, for which one can develop the cohomology theory and the deformation theory using the \emph{operadic Koszul duality}.
Let us give a brief recollection.  See \cite{GK} and \cite[Chap.\ 7]{LV} for the detail. 
\begin{itemize}
\item
For an $\frS$-supermodule $M$, we have the \emph{free operad} $\oT(M)$ which is a superoperad equipped with an $\frS$-supermodule morphism $\eta(M)\colon M \to \oT(M)$ having a universal property \cite[\S5.5.1]{LV}. The free superoperad $\oT(M)$ is endowed with the \emph{weight grading}, denoted as $\oT(M)=\bigoplus_{r \in \bbN} \oT(M)^{(r)}$ \cite[\S5.5.3]{LV}. 

\item
An operadic quadratic data is a pair $(E,R)$ of an $\frS$-supermodule $E$ and 
a sub-$\frS$-supermodule $R \subset \oT(E)^{(2)}$.
The quadratic superoperad associated to an operadic quadratic data $(E,R)$ is a superoperad defined as $\oP(E,R) \ceq \oT(E)/(R)$, where $(R) \subset \oT(E)$ is the operadic ideal generated by $R$. 
It inherits the weight grading $\oP(E,R)=\bigoplus_{r \in \bbN}\oP(E,R)^{(r)}$.
See \cite[\S7.1]{LV} for the detail. Below we consider only the \emph{quadratic operad} $\oP(E,R)$, i.e., the above $E$ is even, so that the superoperads $\oT(E)$ and $\oP(E,R)$ are even.

\item
A \emph{binary quadratic operad} is a quadratic operad $\oP(E,R)$ whose $\frS$-supermodule $E$ concentrates in the 2-arity. 
\end{itemize}
The operads $\oCom$ and $\oAss$ are binary quadratic operads, 
and so is the operad $\oLie$ recalled below.
\begin{itemize}
\item
Let $\oLie$ be the \emph{Lie operad} \cite[\S 13.2]{LV}.
It is a binary quadratic operad $\oLie=\oT(\bbK c)/(R)$ with $R \ceq (c \circ_1 c)+(c \circ_1 c)^{(123)}+(c \circ_1 c)^{(132)} \subset \oT(\bbK c)^{(2)}$, where we used the infinitesimal composition $\circ_1$ in \eqref{eq:1:circ_i} and denoted the right action of the elements $(123),(132) \in \frS_3$ by superscript (see the explanation after \eqref{eq:1:gamma_nu}). 
We have $\oLie(2)=\bbK c$, and the generator $c$ corresponds to the structure of a Lie algebra. More precisely, for a linear superspace $V$, a $\oLie$-algebra $(V,\alpha)$ with even $\alpha$ is nothing but a Lie superalgebra structure on $V$ whose Lie bracket is given by $\alpha_2(c) \in \oHom_V(2)=\Hom(V^{\otimes 2},V)$.
\end{itemize}

Next, we give an overview of the deformation theory of algebraic structures associated with quadratic operads.
We refer to \cite[Chap.\ 12]{LV} for the detail. We begin with the notation for graded categories.
\begin{itemize}
\item 
In the following, we need $\bbZ$-graded objects \cite[\S6.2, \S6.3]{LV}, and everything in $\cLin$ which appeared so far is replaced by one in $\cgLin$, the ($\bbZ_2$,$\bbZ$)-bigraded category of $\bbZ$-graded linear superspaces $V=\bigoplus_{i \in \bbZ}V_i$ where each $V_i$ is a linear superspace. We call the additional $\bbZ$-grading $i$ the \emph{degree}. The morphism set $\Hom(V,W)$ for $V,W \in \cgLin$ is given by $\Hom(V,W)=\bigoplus_{i \in \bbZ}\Hom(V,W)_i$, where each $\Hom(V,W)_i$ is the linear superspace consisting of linear maps $f\colon V \to W$ satisfying $f(V_j) \subset W_{i+j}$ for each $j \in \bbZ$. In the following, we will sometimes omit the word ``$\bbZ$-'' and simply refer to $\bbZ$-graded objects as graded objects.

\item
We denote by $s$ the degree shift in $\cgLin$. So, for a graded linear superspace $V = \bigoplus_{i \in \bbZ} V_i$, we have $(s V)_i=V_{i-1}$. An element $f \in \Hom(V,W)_i$ is called a linear map of degree $i$.

\item
For example, a graded $\frS$-supermodule is a collection $M=(M(n))_{n \in \bbN}$ of $\bbZ$-graded right $\frS_n$-modules $M(n)=\bigoplus_{i \in \bbZ} M_i(n)$, and the index $i$ denotes the degree. We also have $(s M)_i(n)=M_{i-1}(n)$. 
\end{itemize}

We now recall the Koszul dual cooperads and the convolution Lie superalgebras.
\begin{itemize}
\item
We use the notion of \emph{cooperad} $\oC=(\oC,\Delta,\ve)$ \cite[\S5.8]{LV}. In particular, for a ($\bbZ$-graded) $\frS$-(super)module $E$, we have the cofree cooperad $\oT^c(E)$ \cite[\S5.8.6]{LV}.

\item
Let $\oP=\oP(E,R)$ be a quadratic operad ($E$ is even). Then we have the \emph{Koszul dual cooperad} 
\begin{align}\label{eq:1:Kdcop}
 \oP^{\ash} \ceq \oC(s E,s^2R).
\end{align} 
It is a sub-cooperad of the cofree cooperad $\oT^c(s E)$ which is universal among the sub-cooperad $\oC$ such that the composite $\oC \to \oT^c(s E) \to \oT^c(s E)^{(2)}/(s^2 R)$ is zero, and also has the induced weight grading $\oP^{\ash}=\bigoplus_{r \in \bbN} (\oP^{\ash})^{(r)}$ from the weight grading of $\oP$. See \cite[\S7.1]{LV} for the detail of quadratic operads and Koszul dual cooperads.

\item
For a $\bbZ$-graded cooperad $\oC=(\oC,\Delta,\ve)$ and a $\bbZ$-graded superoperad $\oP=(\oP,\gamma,\eta)$, we set
\[
 \Hom_{\frS}(\oC,\oP) \ceq \prod_{n \ge 0} \Hom_{\frS_n}\bigl(\oC(n),\oP(n)\bigr),
\]
where $\Hom_{\frS_n}\bigl(\oC(n),\oP(n)\bigr) \subset \Hom\bigl(\oC(n),\oP(n)\bigr)$ denotes the $\bbZ$-graded  linear superspace of graded morphisms of graded $\frS_n$-supermodules, i.e., $\frS_n$-equivariant $(\bbZ_2,\bbZ)$-graded linear maps. Using the infinitesimal composition $\circ_{(1)}$ given in \eqref{eq:1:circ_i} and the infinitesimal operations $\gamma_{(1)}\colon \oP \circ_{(1)} \oP \to \oP$ and $\Delta_{(1)}\colon \oC \to \oC \circ_{(1)} \oC$ (see \cite[\S6.1]{LV} for the detail), we define the composite 
\begin{align}\label{eq:1:sq}
 f \square g \ceq \gamma_{(1)} \circ (f \circ_{(1)} g) \circ \Delta_{(1)}
\end{align}
for $f,g \in \Hom_{\frS_n}\bigl(\oC(n),\oP(n)\bigr)$, which is again in $\Hom_{\frS_n}\bigl(\oC(n),\oP(n)\bigr)$. Here we used the symbol $\square$ in \cite{BDHK}, instead of $\star$ in \cite{LV}.
This operation $\square$ is a graded pre-Lie product \cite[Propositions 6.4.3, 6.4.5]{LV}, and hence we have a graded Lie superalgebra 
\begin{align}\label{eq:1:cvLie}
 \bigl(\Hom_{\frS}(\oC,\oP),[\cdot,\cdot]\bigr)
\end{align}
with $[f,g] \ceq f \square g-(-1)^{\abs{f}\abs{g}}g \square f$, where $\abs{f}$ denotes the $\bbZ$-grading of the morphism $f$. This graded Lie superalgebra is called the \emph{convolution Lie superalgebra}.
\end{itemize}

We have several applications of the construction of convolution Lie superalgebra \eqref{eq:1:cvLie}.
\begin{itemize}
\item 
Let $\oP=\oP(E,R)$ be a quadratic operad which is \emph{homogeneous}, i.e., $R \subset \oT(E)^{(2)}$, and $\oP^{\ash}$ be the corresponding Koszul dual cooperad \eqref{eq:1:Kdcop}. Let $V$ be a linear superspace and $\oHom_V$ be the endomorphism superoperad. Applying the previous argument to the case $\oC=\oP^{\ash}$ and $\oP=\oHom_V$, we have the convolution Lie superalgebra
\begin{align}\label{eq:1:gPV}
 \frg_{\oP,V} \ceq \bigl(\Hom_{\frS}(\oP^{\ash},\oHom_V),[\cdot,\cdot]\bigr).
\end{align}
It has an $\bbN$-grading $\frg_{\oP,V} = \prod_{r \in \bbN} \frg^r_{\oP,V}$
induced by the weight grading on $\oP^{\ash}$. Explicitly, we have 
\[
 \frg_{\oP,V}^0 = \Hom_{\frS}(I,\oHom_V) \cong \End(V), \quad 
 \frg_{\oP,V}^r = \Hom_{\frS}\bigl((\oP^{\ash})^{(r)},\oHom_V) \quad (r \ge 1),
\]
where $I$ denotes the unit object of the monoidal supercategory $\cSMod$ (see the line after \eqref{eq:pre:op:circ}). Then, by \cite[Proposition 10.1.4]{LV}, we have a bijection
\begin{align}\label{eq:1:MCV}
\begin{split}
 \{\text{$\oP$-algebra structures on $V$}\} &= \Hom_{\cOp}(\oP,\oHom_V) \\
 &\lsto \MC\bigl(\frg_{\oP,V}\bigr) \ceq 
 \{X \in \frg_{\oP,V}^1 \mid \tfrac{1}{2}[X,X]=0\}.
\end{split}
\end{align}
The set $\MC(\frg_{\oP,V})$ is called the \emph{solutions of the Maurer-Cartan equation} (Maurer-Cartan solutions for brevity) of the (differential) graded Lie superalgebra $\frg_{\oP,V}$ (with trivial differential).

\item
Let us replace the endomorphism superoperad $\oHom_V$ in the previous item by an arbitrary superoperad $\oQ$.
Then we obtain the convolution Lie superalgebra 
\begin{align}\label{eq:1:gPQ}
 \frg(\oP,\oQ) \ceq \bigl(\Hom_{\frS}(\oP^{\ash},\oQ),[\cdot,\cdot]\bigr)
\end{align}
whose $\bbN$-graded structure is given by 
\[
 \frg(\oP,\oQ)^0 = \Hom_{\frS}(I,\oQ), \quad 
 \frg(\oP,\oQ)^r = \Hom_{\frS}\bigl((\oP^{\ash})^{(r)},\oQ) \quad (r \ge 1).
\]
In particular, the previous algebra \eqref{eq:1:gPV} is recovered as $\frg_{\oP,V}=\frg(\oP,\oHom_V)$. The argument in \cite[Proposition 10.1.4]{LV} works for general $\oQ$ as well, and we have a bijection
\begin{align}\label{eq:1:MCPQ}
 \Hom_{\cOp}(\oP,\oQ) \lsto \MC\big(\frg(\oP,\oQ)\bigr) \ceq 
 \{X \in \frg(\oP,\oQ)^1 \mid \tfrac{1}{2}[X,X]=0\}.
\end{align}
Let us call a superoperad morphism $\oP \to \oQ$ a \emph{$\oP$-algebra structure on $\oQ$}. Then the above bijection says that a $\oP$-algebra structure on $\oQ$ is in one-to-one correspondence with a Maurer-Cartan solution $X \in \frg(\oP,\oQ)^1$, $[X,X]=0$.
\end{itemize}

\subsection{Lie superalgebra associated to a superoperad}\label{ss:1:LQ}

We are interested in the convolution Lie superalgebra \eqref{eq:1:gPQ} with $\oP=\oLie$, i.e., 
\begin{align}\label{eq:1:frg}
 \frg(\oLie,\oQ) = \bigl(\Hom_{\frS}(\oLie^!,\oQ),[\cdot,\cdot]\bigr), \quad 
 [f,g] \ceq f \square g - (-1)^{\abs{f} \abs{g}}g \square f
\end{align}
for a superoperad $\oQ$. By \eqref{eq:1:MCPQ}, 
this graded Lie superalgebra concerns $\oLie$-algebra structures in $\oQ$: 
\begin{align}\label{eq:1:MCQ}
 \{\text{$\oLie$-algebra structures on $\oQ$}\} \lsto \MC\big(\frg(\oLie,\oQ)\bigr) = 
 \{X \in \frg(\oLie,\oQ)^1 \mid \tfrac{1}{2}[X,X]=0\}.
\end{align}

Let us write down the graded component $\frg(\oLie,\oQ)^n$ for each $n \in \bbN$. Since the homogeneous quadratic operad $\oP=\oLie$ is binary, the argument in \cite[Proposition 12.1.1]{LV} works, and we have
\[
 \frg\bigl(\oLie,\oQ\bigr)^n \cong \Hom_{\frS_n}\bigl(\oLie^{\ash}(n+1),\oQ(n+1)\bigr)
\]
as linear superspaces.
Next, by \cite[\S7,6]{LV}, for any homogeneous quadratic binary operad $\oP$, each arity-graded component $\oP^{\ash}(n)$ of the Koszul dual cooperad $\oP^{\ash}$ satisfies
\[
 \oP^{\ash}(n) \cong \oP^!(n) \otimes_{\frS_n} \sgn_n
\]
as an $\frS_n$-module, where $\sgn_n$ denotes the signature representation of $\frS_n$, and $\oP^!$ is the \emph{Koszul dual operad} of the quadratic operad $\oP$.
We also have $\oLie^! = \oCom$ by the Koszul duality \cite[\S 7.6]{LV}. Since $\oCom(n)$ is the trivial $\frS_n$-representation, we have 
\begin{align}\label{eq:1:frgn}
 \frg(\oLie,\oQ)^n \cong \bigl(\Pi \oQ(n+1)\bigr)^{\frS_{n+1}} \ceq 
 \{f \in \oQ(n+1) \mid \forall \sigma \in \frS_{n+1}, \ f^{\sigma} = \sgn(\sigma) f\}
\end{align}
as linear superspaces, where $\Pi$ is the parity shift functor, and $f^\sigma$ denotes the right action of $\sigma \in \frS_{n+1}$ on $f \in \oQ(n+1)$. 

Next, we give an explicit expression of the pre-Lie product $\square$ in $\frg(\oLie,\oQ)$. Unraveling the definition \eqref{eq:1:sq}, we have for $f \in \frg(\oLie,\oQ)^n$ and $g \in \frg(\oLie,\oQ)^m$ that 
\begin{align}\label{eq:1:square}
 f \square g \ceq \sum_{\sigma \in \frS_{m+1,n}} (f \circ_1 g)^{\sigma^{-1}},
\end{align}
where for $k,l \in \bbN$, the symbol $\frS_{k,l} \subset \frS_{k+l}$ denotes the subset of $(k,l)$-shuffles: 
\begin{align*}
 \frS_{k,l} \ceq \{\sigma \in \frS_{k+l} \mid  
 \sigma(1)<\dotsb<\sigma(k), \, \sigma(k+l)<\dotsb<\sigma(k+l)\}.
\end{align*}
The graded Lie superalgebra $\frg(\oLie,\oQ)$ is essentially the same as the \emph{universal Lie superalgebra associated to $\oQ$} in \cite[\S3.2]{BDHK}. The definition therein shifts the grading and compensates for the parity shift. It is given as follows:


\begin{dfn}\label{dfn:1:LP}
For a superoperad $\oQ$, we define a $\bbZ_{\ge-1}$-graded linear superspace
\begin{align*}
 L(\oQ) \ceq \bigoplus_{n \ge -1}L^n(\oQ), \quad 
 L^n(\oQ) \ceq \oQ(n+1)^{\frS_{n+1}} = 
 \{f\in\oQ(n+1) \mid \forall \sigma\in\frS_{n+1}, \, f^\sigma=f\}
\end{align*} 
and a linear map $\square\colon L(\oQ)\otimes L(\oQ)\to L(\oQ)$ by
\begin{align*}
 f\square g \ceq \sum_{\sigma \in \frS_{m+1, n}} (f\circ_1 g)^{\sigma^{-1}}
 \quad (f\in L^n(\oQ), \ g\in L^m(\oQ)).
\end{align*}
\end{dfn}

Let us summarize the arguments so far in the following form.

\begin{dfn}\label{dfn:1:Lie}
For a superoperad $\oQ$, a \emph{Lie algebra structure on $\oQ$} 
means a superoperad morphism $\oLie \to \oQ$.
\end{dfn}

\begin{prp}[{\cite[Prop.\ 10.1.4]{LV}, \cite{DK13}, \cite[Theorem 3.4]{BDHK}}]\label{prp:1:LP}
For a superoperad $\oQ$, the pair $(L(\oQ),\square)$ in \cref{dfn:1:LP} is a pre-Lie superalgebra (see \cite[\S1.4]{LV} for example). Hence we obtain a graded Lie superalgebra $(L(\oQ),[\cdot,\cdot])$ with 
\begin{align*}
 [f, g]\ceq f\square g-(-1)^{p(f)p(g)}g\square f. 
\end{align*}
Moreover, we have a bijection between Maurer-Cartan solutions and Lie algebra structures:
\[
 \MC\bigl(L(\oQ)\bigr) \ceq \{X \in L^1(\oQ) \mid X \square X = 0\} \lsto 
 \{\text{Lie algebra structures on $\oQ$}\}.
\]
\end{prp}

\section{\texorpdfstring{$N_W=N$}{NW=N} SUSY chiral operad}\label{s:W}

We continue to work over the base field $\bbK$ of characteristic $0$ and to use the symbol $[k] \ceq \{1,2,\dotsc,k\}$ for $k \in \bbN$. We also fix a positive integer $N$. 

\subsection{Polynomial superalgebra}\label{ss:W:poly}

Here we give a summary of the polynomial ring of supervariables. 
Most of the material is standard, and we refer to \cite[\S3.1]{HK} for a detailed description.

\begin{dfn}\label{dfn:W:poly}
Let $A$ be a set and $\Lambda_\alpha=(\lambda_\alpha, \theta_\alpha^1, \ldots, \theta_\alpha^N)$ be a sequence of letters for each $\alpha\in A$. We denote by 
$\bbK[\Lambda_\alpha]_{\alpha \in A}$ the free commutative $\bbK$-superalgebra generated by even $\lambda_\alpha$ $(\alpha\in A)$ and odd $\theta_\alpha^i$ $(\alpha\in A, i\in [N])$, i.e, the $\bbK$-superalgebra generated by these elements with relations 
\begin{align}\label{eq:W:poly}
 \lambda_\alpha\lambda_\beta-\lambda_\beta\lambda_\alpha=0, \quad 
 \lambda_\alpha\theta_\beta^i-\theta_\beta^i\lambda_\alpha=0, \quad
 \theta_\alpha^i\theta_\beta^j+\theta_\beta^j\theta_\alpha^i=0 \quad
 (\alpha, \beta\in A, \ i, j\in[N]). 
\end{align}
Each $\Lambda_\alpha$ for $\alpha\in A$ is called a \emph{$(1|N)_W$-supervariable}, and the $\bbK$-superalgebra $\bbK[\Lambda_\alpha]_{\alpha \in A}$ is called the \emph{$N_W=N$ polynomial superalgebra} of the supervariables $(\Lambda_\alpha)_{\alpha\in A}$. 
\end{dfn}

In the case $A=[n]=\{1,\dotsc,n\}$ for $n \in \bbZ_{>0}$, we often denote the polynomial superalgebra by $\bbK[\Lambda_k]_{k=1}^n$ instead of $\bbK[\Lambda_k]_{k \in [n]}$. If $A$ consists of one element, then we suppress the subscript $\alpha \in A$ and denote the polynomial superalgebra by $\bbK[\Lambda]$.

For a $(1|N)_W$-supervariable $\Lambda=(\lambda, \theta^1, \ldots, \theta^N)$ and a subset $I=\{i_1<\cdots< i_r\}\subset [N]$, we denote $\theta^I\ceq \theta^{i_1}\cdots\theta^{i_r} \in \bbK[\Lambda]$. Also, we define $\sigma(I, J)\in\{0, \pm 1\}$ for $I, J\subset [N]$ by the relation $\theta^I\theta^J=\sigma(I, J)\theta^{I\cup J}$, and set $\sigma(I)\ceq \sigma(I, [N] \bs I)$. For $\alpha \in A$, $m\in\bbN$ and $I\subset [N]$, we denote $\Lambda_\alpha^{m|I}\ceq \lambda_\alpha^m\theta_\alpha^I \in \bbK[\Lambda_\alpha]_{\alpha \in A}$. 

For a linear superspace $V$ and $(1|N)_W$-supervariables $\Lambda_\alpha$ ($\alpha \in A$), 
we denote
\begin{align}\label{eq:W:VLam}
 V[\Lambda_\alpha]_{\alpha \in A} \ceq \bbK[\Lambda_\alpha]_{\alpha\in A} \otimes_{\bbK} V,
\end{align}
which is a (left) $\bbK[\Lambda_\alpha]_{\alpha\in A}$-supermodule. 

As a preliminary of the following subsections, let us introduce:

\begin{dfn}\label{dfn:W:clHW}
Let $\clH_W$ be the free commutative $\bbK$-superalgebra generated by even $T$ and odd $S^i$ ($i\in[N]$), i.e., the $\bbK$-superalgebra generated by these elements with relations
\begin{align}\label{eq:W:TSi}
 TS^i-S^iT=0, \quad S^iS^j+S^jS^i=0\quad (i, j\in[N]).
\end{align}
For simplicity, we set
\begin{align}\label{eq:W:nabla}
 \nabla \ceq (T, S^1, \ldots, S^N).
\end{align}
We also denote a linear superspace $V$ equipped with a left $\clH_W$-supermodule structure as 
\[
 (V,\nabla)=(V,T, S^1, \ldots, S^N),
\]
where $T$ is regarded as an even linear transformation on $V$ and $S^i$ as an odd linear transformation, satisfying the relations \eqref{eq:W:TSi}.
\end{dfn}

Note that $\clH_W$ is isomorphic to $\bbK[\Lambda]$ as a superalgebra by the homomorphism defined by 
\begin{align}\label{eq:W:TSLambda}
T \mto -\lambda,\quad S^i \mto -\theta^i\quad (i\in[N]). 
\end{align}
Since $\clH_W$ is a commutative superalgebra, we suppress the word `left' of an $\clH_W$-supermodule hereafter.

\subsection{$N_W=N$ SUSY Lie conformal operad}\label{ss:W:LCA}

In this subsection, we introduce a natural $N_W=N$ SUSY analogue of the superoperad $\oCh{}$ of Lie conformal algebra in \cite[\S5]{BDHK}. The main contents are \cref{dfn:W:Chom} and \cref{thm:W:LCA}.

Let us fix a $(1|N)_W$-supervariable $\Lambda=(\lambda, \theta^1, \ldots, \theta^N)$. 
Recall the polynomial superalgebra $\bbK[\Lambda]$ in \cref{dfn:W:poly} and the $\bbK[\Lambda]$-supermodule $V[\Lambda] = \bbK[\Lambda] \otimes V$ in \eqref{eq:W:VLam}.

\begin{dfn}[{\cite[Definition 3.2.2]{HK}}]\label{dfn:W:LCA}
Let $(V, \nabla)=(V, T, S^1, \ldots, S^N)$ be an $\clH_W$-supermodule and $[\cdot_\Lambda\cdot]\colon V\otimes V \to V[\Lambda]$ be a linear map of parity $\ol{N}$. A triple $(V,\nabla,[\cdot_\Lambda\cdot])$ is called an \emph{$N_W=N$ SUSY Lie conformal algebra} if it satisfies the following conditions: 
\begin{clist}
\item (\emph{sesquilinearity}) For any $a, b\in V$, 
\begin{align}\label{eq:LCA:sesq}
\begin{split}
&[Ta_\Lambda b]=-\lambda[a_\Lambda b], \hspace{43pt}
[a_\Lambda Tb]=(\lambda+T)[a_\Lambda b], \\
&[S^ia_\Lambda b]=-(-1)^{N}\theta^i[a_\Lambda b], \quad 
 [a_\Lambda S^ib]=(-1)^{p(a)+\ol{N}}(\theta^i+S^i)[a_\Lambda b]\quad (i\in[N]). 
\end{split}
\end{align} 

\item (\emph{skew-symmetry}) For any $a, b\in V$, 
\begin{align}\label{eq:LCA:ssym}
 [b_\Lambda a]=-(-1)^{p(a)p(b)+\ol{N}}[a_{-\Lambda-\nabla}b],
\end{align}
where we used $\nabla \ceq (T, S^1, \ldots, S^N)$ in \eqref{eq:W:nabla}. 

\item (\emph{Jacobi identity}) For any $a, b, c\in V$, 
\begin{align}\label{eq:LCA:Jac}
 [a_{\Lambda_1}[b_{\Lambda_2}c]]
=(-1)^{(p(a)+\ol{N})\ol{N}}[[a_{\Lambda_1}b]_{\Lambda_1+\Lambda_2}c]
+(-1)^{(p(a)+\ol{N})(p(b)+\ol{N})}[b_{\Lambda_2}[a_{\Lambda_1}c]], 
\end{align}
where $\Lambda_1$ and $\Lambda_2$ are $(1|N)_W$-supervariables. 
\end{clist}
The linear map $[\cdot_\Lambda\cdot]$ is called the \emph{$\Lambda$-bracket} of the $\clH_W$-supermodule $(V, \nabla)$. For simplicity, we say $(V, \nabla)$, or more simply $V$, is an $N_W=N$ SUSY Lie conformal algebra. 
\end{dfn}

\begin{rmk}
A few remarks on notation are in order.
\begin{enumerate}
\item 
The linear map $[\cdot_\Lambda \cdot]$ can be expressed as 
\[
 [a_\Lambda b] = \sum_{m \in \bbN, \, I \subset [N]} \Lambda^{m|I} c_{m,I}, \quad c_{m,I} \in V.
\]
Then the term $[a_{-\Lambda-\nabla}b]$ in \eqref{eq:LCA:ssym} is defined to be 
\[
 [a_{-\Lambda-\nabla}b] \ceq 
 \sum_{m \in \bbN, \, I \subset [N]} (-\Lambda-\nabla)^{m|I} c_{m,I}
\]
with $(-\Lambda-\nabla)^{m|I} \ceq (-\lambda-T)^m(-\theta^{i_1}-S^{i_1})\dotsm(-\theta^{i_r}-S^{i_r})$ for $I=\{i_1<\dotsc< i_r\} \subset [N]$.

\item
In \cref{eq:LCA:sesq} and \cref{eq:LCA:ssym}, the symbols $T$ and $S^i$ on the left of the $(1|N)_W$-supervariable $\Lambda$ are identified with the elements of $\bbK[\Lambda]$ by the isomorphism \cref{eq:W:TSLambda}. 
So in $V[\Lambda]$ we have the identities 
\begin{align*}
 T\Lambda^{m|I}v = \Lambda^{m|I}Tv, \quad S^i\Lambda^{m|I}v = (-1)^{\#I}\Lambda^{m|I}S^iv
\end{align*}
for $m\in\bbN$, $I\subset [N]$ and $v\in V$. Also, \cref{eq:LCA:Jac} is regarded as an identity in $V[\Lambda_k]_{k=1, 2}$. 
\end{enumerate}
\end{rmk}



In the remaining of this \cref{ss:W:LCA}, let $(V, \nabla)$ be an $\clH_W$-supermodule and $\Lambda_k=(\lambda_k, \theta_k^1, \ldots, \theta_k^N)$ be a $(1|N)_W$-supervariable for each $k\in\bbZ_{>0}$. 

Next, we turn to the definition of the superoperad $\oChW{V}$. 
The polynomial superalgebra $\bbK[\Lambda_k]_{k=1}^n$ of $(1|N)_W$-supervariables $\Lambda_1,\dotsc,\Lambda_n$ carries a structure of a $(\clH_W^{\otimes n}, \clH_W)$-superbimodule as follows: 
First, since any element of $\clH_W$ is expressed uniquely as $\varphi(\nabla)$ for some $\varphi(\Lambda) \in \bbK[\Lambda]$ using \eqref{eq:W:TSLambda}, the space $\bbK[\Lambda_k]_{k=1}^n$ has a left $\clH_W^{\otimes n}$-supermodule structure by letting $\varphi_1(\nabla) \otimes \cdots \otimes \varphi_n(\nabla) \in \clH_W^{\otimes n}$ act on $a(\Lambda_1, \ldots, \Lambda_n) \in \bbK[\Lambda_k]_{k=1}^n$ as 
\[
 \bigl(\varphi_1(\nabla) \otimes \cdots \otimes \varphi_n(\nabla)\bigr)a(\Lambda_1 \ldots, \Lambda_n)
 \ceq \varphi_1(-\Lambda_1) \cdots \varphi_n(-\Lambda_n)a(\Lambda_1, \ldots, \Lambda_n). 
\]
Also, the space $\bbK[\Lambda_k]_{k=1}^n$ is a right $\clH_W$-supermodule by letting 
\begin{align*}
 a(\Lambda_1,\ldots,\Lambda_n) \cdot T \ceq 
 a(\Lambda_1,\ldots,\Lambda_n)\Bigl(-\sum_{k=1}^n\lambda_k\Bigr), \quad
 a(\Lambda_1,\ldots,\Lambda_n) \cdot S^i \ceq 
 a(\Lambda_1,\ldots,\Lambda_n)\Bigl(-\sum_{k=1}^n\theta_k^i\Bigr)
\end{align*}
for $a(\Lambda_1,\ldots,\Lambda_n) \in \bbK[\Lambda_k]_{k=1}^n$. 
These left and right supermodule structures are consistent, and we have the desired superbimodule structure.

Thus, for an $\clH_W$-supermodule $V=(V,\nabla)$ and $n \in \bbZ_{>0}$, 
we obtain a left $\clH_W^{\otimes n}$-supermodule
\begin{align}\label{eq:LCA:HWn-mod}
 V_{\nabla}[\Lambda_k]_{k=1}^n \ceq \bbK[\Lambda_k]_{k=1}^n \otimes _{\clH_W} V. 
\end{align}
For $n=0$, we denote
\begin{align*}
 V_\nabla[\Lambda_k]_{k=1}^n \ceq V/\nabla V, 
\end{align*}
where $\nabla V\ceq TV+\sum_{i=1}^NS^iV\subset V$. 

By the definition of the right action of $\clH_W$ on $\bbK[\Lambda_k]_{k=1}^n$, we have:

\begin{lem}\label{lem:W:VnLn=VLn-1}
The linear superspace $V_{\nabla}[\Lambda_k]_{k=1}^n$ is isomorphic to $V[\Lambda_k]_{k=1}^{n-1}$ by the linear map 
\begin{align*}
 a(\Lambda_1,\dotsc,\Lambda_n) \otimes v \lmto 
 a(\Lambda_1,\dotsc,\Lambda_{n-1},-\Lambda_1-\cdots-\Lambda_{n-1}-\nabla)v
 \quad (a \in \bbK[\Lambda_k]_{k=1}^n, \, v \in V).
\end{align*}
In the case $n=1$, the statement is understand as $V_\nabla[\Lambda] \sto V$, $a(\Lambda) \otimes v \mto a(-\nabla)v$.
\end{lem}

The tensor product $V^{\otimes n}$ is naturally a left $\clH_W^{\otimes n}$-supermodule. Thus, the following definition makes sense, which is a SUSY analogue of the superoperad $\oCh{}$ in \cite[\S5.2]{BDHK}.

\begin{dfn}
For an $\clH_W$-supermodule $V=(V,\nabla)$ and $n \in \bbN$, we denote
\[
 \oChW{V}(n) \ceq \Hom_{\clH_W^{\otimes n}}(V^{\otimes n}, V_{\nabla}[\Lambda_k]_{k=1}^n).
\]
In other words, $\oChW{V}(n)$ is a linear sub-superspace of 
$\Hom_\bbK(V^{\otimes n}, V_{\nabla}[\Lambda_k]_{k=1}^n)$ spanned by elements $X$ such that
\begin{align*}
 X(\varphi v) = (-1)^{p(\varphi)p(X)}\varphi X(v) \quad 
 (\varphi \in \clH_{W}^{\otimes n}, \ v \in V^{\otimes n}). 
\end{align*}
To stress the supervariables $\Lambda_1,\ldots,\Lambda_n$, we express an element $X \in \oChW{V}(n)$ as $X_{\Lambda_1,\ldots,\Lambda_n}$. Note that we have 
\begin{align*}
&\oChW{V}(0) = \Hom_{\bbK}(\bbK,V/\nabla V) \cong V/\nabla V, \\
&\oChW{V}(1)=\Hom_{\clH_W}(V, V_\nabla[\Lambda])\cong \End_{\clH_W}V
\end{align*}
\end{dfn}

For $\sigma \in \frS_n$ and  $X \in \oChW{V}(n)$, we define a linear map $X^\sigma\colon V^{\otimes n} \to V_{\nabla}[\Lambda_k]_{k=1}^n$ by
\begin{align*}
 X^\sigma(v_1 \otimes \cdots \otimes v_n) \ceq 
 X_{\sigma(\Lambda_1,\ldots,\Lambda_n)}\bigl(\sigma(v_1 \otimes \cdots \otimes v_n)\bigr) 
 \quad (v_1,\ldots,v_n \in V),
\end{align*}
where $\sigma(\Lambda_1,\ldots,\Lambda_n) \ceq (\Lambda_{\sigma^{-1}(1)},\ldots,\Lambda_{\sigma^{-1}(n)})$ and $\sigma(v_1 \otimes \dotsb \otimes v_n) \ceq \pm v_{\sigma^{-1}(1)} \otimes \dotsb \otimes v_{\sigma^{-1}(n)}$ given in \eqref{eq:1:Sn-Vn}. 
Then one can check $\sigma(\varphi v) = (\sigma\varphi)(\sigma v)$ for $\varphi \in \clH_W^{\otimes n}$ and $v \in V^{\otimes n}$, by which one has $X^\sigma \in \oChW{V}(n)$. 
The linear superspace $\oChW{V}(n)$ carries a structure of a right $\frS_n$-supermodule by this action, and we have an $\frS$-supermodule $\oChW{V} \ceq \bigl(\oChW{V}(n)\bigr)_{n \in \bbN}$. 

\begin{dfn}
Let $m,n \in \bbN$, $(n_1,\ldots,n_m) \in \bbN^m_n$, and set $N_j \ceq n_1+\cdots+n_j$ for $j \in [m]$ and $N_0\ceq 0$. 
\begin{enumerate}
\item
For $Y_j\in\oChW{V}(n_i)$ with $j \in [m]$, we define a linear map
\begin{align*}
 Y_1 \odot \cdots \odot Y_m\colon V^{\otimes n} \lto 
 \bigotimes_{j=1}^m V_\nabla[\Lambda_k]_{k=N_{j-1}+1}^{N_j}
 \cong \bigotimes_{j=1}^m V[\Lambda_k]_{k=N_{j-1}+1}^{N_ij1}
\end{align*}
by  
\begin{align*}
 (Y_1 \odot \cdots \odot Y_m)(v_1 \otimes \ldots \otimes v_n) \ceq 
 \pm (Y_1)_{\Lambda_1,\ldots,\Lambda_{N_1}}(w_1) \otimes \cdots \otimes 
     (Y_m)_{\Lambda_{N_{m-1}+1},\ldots,\Lambda_{N_m}}(w_m)
\end{align*}
for $v_1,\ldots,v_n\in V$, where we used
\begin{align*}
\pm \ceq \prod_{1\le i< j\le m} (-1)^{p(w_i)p(Y_j)}, \quad 
 w_j \ceq v_{N_{j-1}+1} \otimes \cdots \otimes v_{N_j} \quad (j\in[m]). 
\end{align*}

\item
For $X\in\oChW{V}(m)$ and $Y_j\in\oChW{V}(n_j)$ with $j\in[m]$, we denote by 
\begin{align*}
 X \circ (Y_1 \odot \cdots \odot Y_m)\colon V^{\otimes n} \lto V_\nabla[\Lambda_k]_{k=1}^n
\end{align*}
the composition of linear maps
\[
 V^{\otimes n} \xrr{Y_1 \odot \cdots \odot Y_m} 
 \bigotimes_{j=1}^m V[\Lambda_k]_{k=N_{j-1}+1}^{N_j-1} 
 \xrr{X_{\Lambda_1', \ldots, \Lambda_m'}}  V_\nabla[\Lambda_k]_{k=1}^n. 
\]
Here we set $\Lambda_j' \ceq \Lambda_{N_{j-1}+1}+\cdots+\Lambda_{N_j}$. 
Also, the symbol $X_{\Lambda'_1, \ldots, \Lambda'_m}$ is the linear map defined by 
\begin{align*}
 a_1v_1 \otimes \cdots \otimes a_m v_m \lmto
 \pm (a_1 \cdots a_m)X_{\Lambda'_1,\ldots,\Lambda'_m}(v_1 \otimes \cdots \otimes v_m)
\end{align*}
for each $a_j\in\bbK[\Lambda_k]_{k=N_{j-1}+1}^{N_j-1}$ and $v_j\in V$ ($j\in[m]$) with the sign
\begin{align*}
 \pm \ceq \prod_{1\le i<j\le m} (-1)^{p(v_i)p(a_j)} \cdot \prod_{j=1}^m (-1)^{p(a_j)p(X)}. 
\end{align*}
\end{enumerate}
\end{dfn}

Let $m,n \in \bbN$ and $\nu \in \bbN^m_n$. For $X \in \oChW{V}(m)$ and $Y_1 \otimes \cdots \otimes Y_m \in \oChW{V}(\nu)$, a direct calculation shows that 
\begin{align*}
 X \circ (Y_1 \odot \cdots \odot Y_m) \in \oChW{V}(n). 
\end{align*}

Now recall \cref{dfn:1:op} of superoperads. The $\frS$-supermodules $\oChW{V}$ give rise to:

\begin{lem}
For an $\clH_W$-supermodule $V=(V,\nabla)$, the $\frS$-supermodule
\[
 \oChW{V} = \bigl(\oChW{V}(n)\bigr)_{n\in\bbN}
\]
is a superoperad by letting $X \otimes Y_1 \otimes \cdots \otimes Y_m \mto X \circ (Y_1 \odot \cdots \odot Y_m)$ be the composition map and $\id_V \in \oChW{V}(1)$ be the unit. 
\end{lem}

\begin{proof}
The proof in the non-SUSY case \cite[\S5.2]{BDHK} works with minor modifications. We omit the detail. 
\end{proof}

\begin{dfn}\label{dfn:W:Chom}
The superoperad $\oChW{V}$ is called the \emph{$N_W=N$ SUSY Lie conformal operad} of the $\clH_W$-supermodule $V=(V,\nabla)$.
\end{dfn}

For an $\clH_W$-supermodule $V=(V,\nabla)$ and $k \in \bbN$, we denote by $\Pi^k V=(\Pi^k V,\nabla)$ the $k$-times parity-shifted linear superspace $\Pi^k V$ equipped with the supermodule structure $\nabla=(T,S^1,\dotsc,S^N)$ over the commutative superalgebra $\clH_W$. Recall the set $\MC\bigl(L(\oQ)\bigr)$ of Maurer-Cartan solutions $X \square X = 0$, $X \in L^1(\oQ)$ in the graded Lie superalgebra $L(\oQ)$ associated to a superoperad $\oQ$ (see \cref{prp:1:LP}).  Now we can state the main claim of this \cref{ss:W:LCA}, which is a natural $N_W=N$ SUSY analogue of \cite[Proposition 5.1]{BDHK}.

\begin{thm}\label{thm:W:LCA}
Let $V=(V,\nabla)$ be an $\clH_W$-supermodule, i.e., a linear superspace $V$ endowed with $T \in \End_{\bbK}(V)_{\ev}$ and $S^i \in \End_{\bbK}(V)_{\od}$ ($i\in[N]$) satisfying \eqref{eq:W:TSi}. 
\begin{enumerate}
\item 
For an odd Maurer-Cartan solution $X\in \MC\bigl(L\bigl(\oChW{\Pi^{N+1}V}\bigr)\bigr)_{\od}$, define a linear map $[\cdot_\Lambda\cdot]_X\colon V\otimes V\to V[\Lambda]$ by 
\begin{align}\label{eq:LCAstr}
 [a_\Lambda b]_X \ceq 
 (-1)^{p(a)(\ol{N}+\od)}X_{\Lambda, -\Lambda-\nabla}(a \otimes b) \quad (a,b \in V).
\end{align}
Then $(V, \nabla, [\cdot_\Lambda\cdot]_X)$ is an $N_W=N$ SUSY Lie conformal algebra. 
\item
The map $X \mto [\cdot_\Lambda\cdot]_X$ gives a bijection 
\[
 \MC\bigl(L\bigl(\oChW{\Pi^{N+1}V}\bigr)\bigr)_{\od} \lsto 
 \{\text{$N_W=N$ SUSY Lie conformal algebra structures on $(V,\nabla)$}\}.
\]
\end{enumerate}
\end{thm}

\begin{proof}
We denote $\wt{V} \ceq \Pi^{N+1}V=(\Pi^{N+1}V,\nabla)$ and $\oP \ceq \oChW{\Pi^{N+1}V}$ for simplicity, and let $p$ and $\wt{p}$ be the parity of $V$ and $\wt{V}$, respectively. Consider the linear map $[\cdot_\Lambda\cdot]_X\colon V \otimes V \to V[\Lambda]$ defined by \cref{eq:LCAstr} for an odd element $X \in \oP(2)_{\ol{1}}$.
We check the following two properties of $[\cdot_\Lambda\cdot]_X$.
\begin{itemize}
\item
The linear map $[\cdot_\Lambda\cdot]_X$ has parity $\ol{N}$: Since $X$ is an odd linear map $X\colon \wt{V}\otimes \wt{V} \to \wt{V}_{\nabla}[\Lambda_k]_{k=1,2}$, 
\begin{align*}
 p(X(a\otimes b)) = \wt{p}(X(a\otimes b))+\ol{N}+\od=\wt{p}(a)+\wt{p}(b)+\ol{N}=p(a)+p(b)+\ol{N}
\end{align*}
holds for $a,b \in V$. 

\item
The linear map $[\cdot_\Lambda\cdot]_X$ satisfies the sesquilinearity (i) in \cref{dfn:W:LCA}.
This can be checked by direct calculation using
\begin{align*}
 X(\varphi v) = (-1)^{\wt{p}(\varphi)}\varphi X(v) \quad 
 (\varphi \in \clH_{W}^{\otimes 2}, \, v \in \wt{V}^{\otimes 2}). 
\end{align*}
For instance, we can calculate 
\begin{align*}
 [a^i_\Lambda S^i b]
&=(-1)^{p(a)(\ol{N}+\od)}X_{\Lambda, -\Lambda-\nabla}(a\otimes S^ib)
 =(-1)^{p(a)(\ol{N}+\od)+\wt{p}(a)}
  X_{\Lambda, -\Lambda-\nabla}\bigl((\id_V \otimes S^i)(a\otimes b)\bigr) \\
&=(-1)^{p(a)(\ol{N}+\od)+\wt{p}(a)+1}(\theta^i+S^i)X_{\Lambda,-\Lambda-\nabla}(a\otimes b) \\
&=(-1)^{p(a)+\ol{N}}(\theta^i+S^i)[a_\Lambda b]_X \quad (a, b\in V).
\end{align*}
\end{itemize}

At this point, we see that the map $X \mto [\cdot_\Lambda\cdot]_X$ gives a bijective correspondence between the set $\oP(2)_{\od}$ and the set of all linear maps $[\cdot_\Lambda\cdot]: V\otimes V \to V[\Lambda]$ of parity $\ol{N}$ satisfying (i) in \cref{dfn:W:LCA}. 

Next, we study the conditions (ii) and (iii) in \cref{dfn:W:LCA} and show that they can be restated as the defining conditions of Maurer-Cartan solutions. 
\begin{itemize}
\item 
The linear map $[\cdot_\Lambda\cdot]_X$ satisfies the skew-symmetry \eqref{eq:LCA:ssym} 
((ii) in \cref{dfn:W:LCA})
$\siff$ $X^\sigma=X$ for any $\sigma\in\frS_2$:
The element $X$ satisfies $X^\sigma=X\ (\sigma\in\frS_2)$ if and only if
\begin{align*}
 (-1)^{\wt{p}(a)\wt{p}(b)}X_{\Lambda_2, \Lambda_1}(b\otimes a)
=X_{\Lambda_1, \Lambda_2}(a\otimes b)\quad(a, b\in V). 
\end{align*}
Since  
\begin{align*}
[b_\Lambda a]
&=(-1)^{p(b)(\ol{N}+\od)}X_{\Lambda, -\Lambda-\nabla}(b\otimes a), \\
[a_{-\Lambda-\nabla} b]
 &=(-1)^{p(a)(\ol{N}+\od)}X_{-\Lambda-\nabla, \Lambda}(a\otimes b) \quad
 (a, b\in V),
\end{align*}
the skew-symmetry is equivalent to $X^\sigma=X$ $(\sigma\in\frS_2)$. 
\end{itemize}

In what follows, let $X\in L^1(\oP)_{\od}$. Thus $[\cdot_\Lambda\cdot]_X\colon V\otimes V \to V_{\nabla}[\Lambda]$ is a linear map of parity $\ol{N}$ satisfying (i), (ii) in \cref{dfn:W:LCA}. 
We abbreviate $X$ in $[\cdot_\Lambda\cdot]_X$
\begin{itemize}
\item
The linear map $[\cdot_\Lambda\cdot]_X$ satisfies the Jacobi identity \eqref{eq:LCA:Jac}
((iii) in \cref{dfn:W:LCA}) $\siff$ $X\square X=0$: For $a, b, c\in V$, a direct calculation shows that
\begin{align*}
(X\circ_1X)(a\otimes b\otimes c)
&=X_{\Lambda_1+\Lambda_2, \Lambda_3}(X_{\Lambda_1, -\Lambda_1-\nabla}(a\otimes b)\otimes c)\\
&=(-1)^{p(b)(\ol{N}+\od)}[[a_{\Lambda_1}b]_{\Lambda_1+\Lambda_2}c], \\
(X\circ_2X)(a\otimes b\otimes c)
&=(-1)^{\wt{p}(a)}X_{\Lambda_1, \Lambda_2+\Lambda_3}(a\otimes X_{\Lambda_2, -\Lambda_2-\nabla}(b\otimes c))\\
&=(-1)^{\wt{p}(a)}(-1)^{p(a)(\ol{N}+\od)}(-1)^{p(b)(\ol{N}+\od)}[a_{\Lambda_1}[b_{\Lambda_2}c]]\\
(X\circ_2X)^{(1, 2)}(a\otimes b\otimes c)
&=(-1)^{\wt{p}(a)\wt{p}(b)}(X\circ_2X)_{\Lambda_2, \Lambda_1, \Lambda_3}(b\otimes a\otimes c)\\
&=(-1)^{\wt{p}(a)\wt{p}(b)}(-1)^{\wt{p}(b)}(-1)^{p(a)(\ol{N}+\od)}(-1)^{p(b)(\ol{N}+\od)}[b_{\Lambda_2}[a_{\Lambda_1}c]]. 
\end{align*}
Thus, we have
\begin{align*}
 &(X\square X)(a\otimes b\otimes c)\\
 &=(X\circ_1X)(a\otimes b\otimes c)+(X\circ_2X)(a\otimes b\otimes c)+(X\circ_2X)^{(1, 2)}(a\otimes b\otimes c)\\
 &=\pm\bigl([a_{\Lambda_1}[b_{\Lambda_2}c]]
 -(-1)^{(p(a)+\ol{N})\ol{N}}[[a_{\Lambda_1}b]_{\Lambda_1+\Lambda_2}c]
 -(-1)^{(p(a)+\ol{N})(p(b)+\ol{N})}[b_{\Lambda_2}[a_{\Lambda_1}c]]\bigr), 
\end{align*}
where $\pm\ceq (-1)^{\wt{p}(a)}(-1)^{p(a)(\ol{N}+\od)}(-1)^{p(b)(\ol{N}+\od)}$. 
Hence the Jacobi identity is equivalent to $X\square X=0$. 
\end{itemize}
Therefore, we have the statement.
\end{proof}

\begin{rmk}
Combined with \cref{prp:1:LP}, the theorem says that a Lie conformal algebra structure on $(V,\nabla)$ is an odd Lie algebra structure on the operad $\oChW{\Pi^{N+1}V}$. 
\end{rmk}

\subsection{\texorpdfstring{$N_W=N$}{NW=N} SUSY chiral operad}\label{ss:W:VA}

Here is the main part of this \cref{s:W}, where we will introduce a natural $N_W=N$ SUSY analogue of the operad of vertex algebras in \cite[\S6]{BDHK}. 

Let us fix a $(1|N)_W$-supervariable $\Lambda=(\lambda, \theta^1, \ldots, \theta^N)$.
For even linear transformations $F$ and $G$ on a linear superspace $V$, we define a linear map $\int_F^G d\Lambda\colon V[\Lambda] \to V$ by 
\begin{align}\label{eq:W:intLv}
 \int_F^G d\Lambda\ \Lambda^{m|I}v\ceq\frac{\delta_{I, [N]}}{m+1}(G^{m+1}v-F^{m+1}v)
 \quad (m\in\bbN, \, I\subset [N], \, v\in V).
\end{align}
The linear map $\int_F^G d\Lambda$ has the parity $\ol{N}$. 
Also, if $V$ is a superalgebra (not necessarily unital nor associative), we define a linear map 
$\int_F^G d\Lambda \, a\colon V[\Lambda] \to V$ for $a\in V$ by
\begin{align*}
 \Bigl(\int_F^G d\Lambda \, a\Bigr)\Lambda^{m|I}v \ceq 
 \Bigl(\int_F^G d\Lambda\ \Lambda^{m|I}a\Bigr)v,
\end{align*}
where the term $\int_F^G d\Lambda\ \Lambda^{m|I}a$ in the right hand side is given by \eqref{eq:W:intLv}. Then the linear map $\int_F^G d\Lambda \,a$ has the parity $p(a)$. 
Using this integral, we introduce:

\begin{dfn}[{\cite[Definition 3.3.15]{HK}}]\label{dfn:NWVA}
Let $(V,\nabla,[\cdot_\Lambda\cdot])=(V,T,S^1,\dotsc,S^N,[\cdot_\Lambda\cdot])$ be an $N_W=N$ SUSY Lie conformal algebra (\cref{dfn:W:LCA}) and $\mu\colon V \otimes V \to V$ be an even linear map. We denote $a b\ceq\mu(a\otimes b)$ for $a, b\in V$. A tuple $(V,\nabla,[\cdot_\Lambda\cdot],\mu)$ is called a \emph{non-unital $N_W=N$ SUSY vertex algebra} if it satisfies the following conditions: 
\begin{clist}
\item For any $a, b\in V$, 
\begin{align}\label{eq:VA:der}
 T(ab)=(Ta)b+a(Tb), \quad S^i(ab)=(S^ia)b+(-1)^{p(a)}a(S^ib)\quad(i\in[N]). 
\end{align}

\item (\emph{quasi-commutativity}) For any $a, b\in V$, 
\begin{align}\label{eq:VA:qcom}
 ab-(-1)^{p(a)p(b)}ba=\int_{-T}^0d\Lambda[a_\Lambda b]. 
\end{align}

\item (\emph{quasi-associativity}) For any $a, b, c\in V$, 
\begin{align}\label{eq:VA:qass}
  (ab)c-a(bc) = \Bigl(\int_0^Td\Lambda a\Bigr)[b_\Lambda c]
+(-1)^{p(a)p(b)}\Bigl(\int_0^Td\Lambda b\Bigr)[a_\Lambda c]. 
\end{align}

\item (\emph{Wick formula}) For any $a, b, c\in V$, 
\begin{align}\label{eq:VA:Wick}
 [a_\Lambda bc]=[a_\Lambda b]c+(-1)^{(p(a)+\ol{N})p(b)}b[a_\Lambda c]+\int_0^\lambda d\Gamma [[a_\Lambda b]_\Gamma c], 
\end{align}
where $\Gamma$ is an additional $(1|N)_W$-supervariable. 
\end{clist}

For simplicity, we say $(V, \nabla)$, or more simply $V$, is a non-unital $N_W=N$ SUSY vertex algebra. The map $\mu$ is called the \emph{multiplication} of $V$.
\end{dfn}

\begin{dfn}
A non-unital $N_W=N$ SUSY vertex algebra $V$ is called an \emph{$N_W=N$ SUSY vertex algebra} if there exists an even element $\vac \in V$ such that $a\vac=\vac a=a$ for all $a\in V$. 
\end{dfn}

\begin{eg}[{\cite[Example 5.4]{HK}}]
Let $F_N$ be the $N_W=N$ SUSY vertex algebra generated by even $\alpha$ and odd $\varphi$ whose non-trivial OPE is
\begin{align*}
 [\alpha_\Lambda \varphi] = \vac. 
\end{align*}
When $N=1$, expand the corresponding superfields
\begin{align*}
 \alpha(Z) = a(z)+\theta\psi(z), \quad \varphi(Z) = \phi(z)+\theta b(z), 
\end{align*}
then we have
\begin{align*}
 [b_\lambda a] = [\psi_\lambda \phi] = \vac.
\end{align*}
These are the OPEs of well-known $bc$-$\beta\gamma$ system. 
\end{eg}

In the remaining of this \cref{ss:W:VA}, let $(V,\nabla)$ be an $\clH_W$-supermodule and $\Lambda_k=(\lambda_k, \theta_k^1, \ldots, \theta_k^N)$ be a $(1|N)_W$-supervariable for each $k\in\bbZ_{>0}$. 

Now we give several preliminaries to introduce the $N_W=N$ SUSY chiral operad (see \cref{dfn:W:Pch}).  
\cref{lem:rWick} and \cref{lem:lsym} below are SUSY analogues of \cite[(1.38), (1.41)]{DK}.

\begin{lem}\label{lem:rWick}
 Let $[\cdot_\Lambda\cdot]: V\otimes V \to V[\Lambda]$ be a linear map of parity $\ol{N}$ satisfying the sesquilinearity \eqref{eq:LCA:sesq} and $\mu: V\otimes V\to V$ be an even linear map satisfying \eqref{eq:VA:der}. In addition, we assume that $[\cdot_\Lambda\cdot]$ and $\mu$ satisfy the skew symmetry \eqref{eq:LCA:ssym} and the quasi-commutativity \eqref{eq:VA:qcom}. Then the Wick formula \eqref{eq:VA:Wick} is equivalent to the following identity which is called the \emph{right Wick formula}: 
\begin{align*}
 [ab_\Lambda c]
 =(-1)^{p(a)\ol{N}}(e^{\nabla\cdot\pdd_{\Lambda}}a)[b_\Lambda c]
 +(-1)^{(p(a)+\ol{N})p(b)}(e^{\nabla\cdot\pdd_{\Lambda}}b)[a_\Lambda c]
 +(-1)^{(p(a)+\ol{N})p(b)}\int_0^\lambda d\Gamma[b_\Gamma[a_{\Lambda-\Gamma}c]].
\end{align*}
Here we denoted $\pdd_\Lambda\ceq(\pdd_\lambda, \pdd_{\theta^1}, \ldots, \pdd_{\theta^N})$, $\nabla\cdot\pdd_\Lambda\ceq T\pdd_\lambda+\sum_{i=1}^NS^i\pdd_{\theta^i}$, and for $a\in V$, we defined a linear map $e^{\nabla\cdot\pdd_{\Lambda}}a\colon V[\Lambda] \to V[\Lambda]$ by
\begin{align*}
 \bigl(e^{\nabla\cdot\pdd_{\Lambda}}a\bigr)\Lambda^{m|I}v \ceq 
 (-1)^{p(a) \cdot \ol{\#I}}\bigl(e^{\nabla\cdot\pdd_{\Lambda}}\Lambda^{m|I}a\bigr)v 
 \quad  (m \in \bbN, \, I \subset [N], \, v \in V).
\end{align*}
\end{lem}
\begin{proof}
If the Wick formula holds, then
\begin{align*}
[ab_\Lambda c]
 &=\pm[c_{-\Lambda-\nabla}ab]
 =\pm e^{\nabla\cdot\pdd_{\Lambda}}[c_{-\Lambda}ab]\\
 &=\pm e^{\nabla\cdot\pdd_{\Lambda}}\Bigl([c_{-\Lambda}a]b+(-1)^{(p(c)+\ol{N})p(a)}a[c_{-\Lambda}b]+\int_0^{-\lambda}d\Gamma\,[[c_{-\Lambda}a]_{\Gamma}b]\Bigr)\\
 &=\pm e^{\nabla\cdot\pdd_{\Lambda}}\Bigl((-1)^{(p(a)+p(c)+\ol{N})p(b)}b[c_{-\Lambda}a]+(-1)^{(p(c)+\ol{N})p(a)}a[c_{-\Lambda}b]+\int_{-T}^{-\lambda}d\Gamma\,[[c_{-\Lambda}a]_{\Gamma}b]\Bigr), 
\end{align*}
Here we set $\pm\ceq (-1)^{(p(a)+p(b))p(c)+\ol{N}+\od}$, and
used the skew-symmetry in the first equality, the Wick formula in the third equality, and the quasi-commutativity in the last equality. A direct calculation shows that
\begin{align*}
&\pm e^{\nabla\cdot\pdd_{\Lambda}}(b[c_{-\Lambda}a])
=\pm\left(e^{\nabla\cdot\pdd_\Lambda}b\right)[c_{-\Lambda-\nabla} a]
=(-1)^{p(b)p(c)}\left(e^{\nabla\cdot\pdd_\Lambda}b\right)[a_{\Lambda}c]\\
&\pm e^{\nabla\cdot\pdd_{\Lambda}}(a[c_{-\Lambda}b])
=(-1)^{p(a)p(c)}\left(e^{\nabla\cdot\pdd_\Lambda}a\right)[b_\Lambda c]
\end{align*}
and 
\begin{align*}
\pm e^{\nabla\cdot\pdd_{\Lambda}}\int_{-T}^{-\lambda}d\Gamma\,[[c_{-\Lambda}a]_{\Gamma}b]
=(-1)^{(p(a)+\ol{N})p(b)}\int_0^\lambda d\Gamma\, [b_\Gamma[a_{\Lambda-\Gamma}c]]
\end{align*}
by which the right Wick formula holds. We can similarly prove the converse. 
\end{proof}

\begin{lem}\label{lem:lsym}
Let $[\cdot_\Lambda\cdot]\colon V\otimes V \to V[\Lambda]$ and $\mu\colon V\otimes V\to V$ be the same as in \cref{lem:rWick}. In addition, we assume that $[\cdot_\Lambda\cdot]$ and $\mu$ satisfy the skew-symmetry \eqref{eq:LCA:ssym}, the quasi-commutativity \eqref{eq:VA:qcom} and the Wick formula \eqref{eq:VA:Wick}. Then the quasi-associativity \eqref{eq:VA:qass} is equivalent to the following identity: 
\begin{align}\label{eq:lsym}
 a(bc)-(-1)^{p(a)p(b)}b(ac)=(ab-(-1)^{p(a)p(b)}ba)c
\end{align}
\end{lem}
\begin{proof}
Since the right-hand side of the quasi-associativity is symmetric with respect to $a$ and $b$, it is clear that the quasi-associativity implies the identity \eqref{eq:lsym}. Conversely, we assume the identity \eqref{eq:lsym}. By the quasi-commutativity and the identity \eqref{eq:lsym}, we have
\begin{align*}
 (ab)c-a(bc)
 &=(-1)^{(p(a)+p(b))p(c)}c(ab)-(-1)^{p(b)p(c)}a(cb)
 +\int_{-T}^0d\Lambda[ab_\Lambda c]-a\int_{-T}^0d\Lambda [b_\Lambda c]\\
 &=-(-1)^{p(b)p(c)}\Bigl(\int_{-T}^0d\Lambda[a_\Lambda c]\Bigr)b
 +\int_{-T}^0d\Lambda[ab_\Lambda c]-a\int_{-T}^0d\Lambda [b_\Lambda c]\\
 &=(-1)^{p(a)p(b)}\int_{-T}^0d\Gamma\Bigl[b_{\Gamma}\Bigl(\int_{-T}^0d\Lambda [a_\Lambda c]\Bigr)\Bigr]\\
 &\quad-(-1)^{p(a)p(b)}b\int_{-T}^0d\Lambda [a_\Lambda c]
 +\int_{-T}^0d\Lambda[ab_\Lambda c]-a\int_{-T}^0d\Lambda [b_\Lambda c].
\end{align*}
We can calculate the double integral as
\begin{align*}
 \int_{-T}^0d\Lambda_2\Bigl[b_{\Lambda_2}\Bigl(\int_{-T}^0d\Lambda_1 [a_{\Lambda_1} c]\Bigr)\Bigr]
 &=(-1)^{(p(b)+\ol{N})\ol{N}}\int_{-T}^0d\Lambda_2\int_{-\lambda_2-T}^0d\Lambda_1[b_{\Lambda_2}[a_{\Lambda_1} c]]\\
 &=(-1)^{(p(b)+\ol{N})\ol{N}}\int_{-T}^0d\Lambda_2\int_{-T}^{\lambda_2}d\Lambda_1[b_{\Lambda_2}[a_{\Lambda_1-\Lambda_2} c]]\\
 &=(-1)^{p(b)N}\int_{-T}^0d\Lambda_1\int_{\lambda_1}^0d\Lambda_2[b_{\Lambda_2}[a_{\Lambda_1-\Lambda_2}c]],
\end{align*}
so by \cref{lem:rWick}, we obtain
\begin{align*}
 &(-1)^{p(a)p(b)}\int_{-T}^0d\Gamma\Bigl[b_{\Gamma}\Bigl(\int_{-T}^0d\Lambda [a_{\Lambda} c]\Bigr)\Bigr]\\
 &=-\int_{-T}^0d\Lambda[ab_\Lambda c]
 +(-1)^{p(a)\ol{N}}\int_{-T}^0d\Lambda(e^{\nabla\cdot\pdd_\Lambda}a)[b_\Lambda c]
 +(-1)^{(p(a)+\ol{N})p(b)}\int_{-T}^0d\Lambda(e^{\nabla\cdot\pdd_\Lambda}b)[a_\Lambda c].
\end{align*}
Thus 
\begin{align*}
 (ab)c-a(bc)
 &=(-1)^{p(a)\ol{N}}\int_{-T}^0d\Lambda(e^{\nabla\cdot\pdd_\Lambda}a)[b_\Lambda c]
 -a\int_{-T}^0d\Lambda [b_\Lambda c]\\
 &\quad+(-1)^{p(a)p(b)}\Bigl((-1)^{p(b)\ol{N}}\int_{-T}^0d\Lambda(e^{\nabla\cdot\pdd_\Lambda}b)[a_\Lambda c]-b\int_{-T}^0d\Lambda [a_\Lambda c]\Bigr).
\end{align*}
Since a direct calculation shows
\begin{align*}
 (-1)^{p(a)\ol{N}}\int_{-T}^0d\Lambda(e^{\nabla\cdot\pdd_\Lambda}a)[b_\Lambda c]
 -a\int_{-T}^0d\Lambda [b_\Lambda c]
 =\Bigl(\int_0^Td\Lambda a\Bigr)[b_\Lambda c],
\end{align*}
the identity \eqref{eq:lsym} implies the quasi-associativity. 
\end{proof}

Next we introduce the indefinite integral of the $\Lambda$-bracket (\cref{dfn:W:indef}), and translate the definition of the $N_W=N$ SUSY vertex algebra using the integral (\cref{prp:W:skecom,prp:W:Jqas}). As a preliminary, let us recall the residue map \cite[3.1.2]{HK}:
\begin{align*}
 \Res_\Lambda(\lambda^{-1}-)\colon V[\Lambda] \lto V, \quad 
 \Res_{\Lambda}(\lambda^{-1}\Lambda^{m|I}a) \ceq \delta_{m,0} \delta_{I,[N]} a \quad 
 (m \in \bbN, \, I \subset [N], \, a \in V). 
\end{align*}
It satisfies the following property.

\begin{lem}\label{lem:intbra}
For a linear map $[\cdot_\Lambda\cdot]\colon V \otimes V \to V[\Lambda]$ of parity $\ol{N}$ satisfying the sesquilinearity \eqref{eq:LCA:sesq}
and an even linear map $\mu\colon V\otimes V\to V$ satisfying \eqref{eq:VA:der},
there exists a unique linear map $F\colon V\otimes V \to V[\Lambda]$ of parity $\ol{N}$ satisfying the following conditions for every $a, b\in V$: 
\begin{align}\label{eq:2:intbra}
\begin{split}
&\Res_\Lambda(\lambda^{-1}F(S^i a \otimes b))= 
 (-1)^{N+1} \Res_\Lambda(\lambda^{-1}\theta^i F(a \otimes b)) \quad (i \in [N]), \\
&\pdd_\lambda F(a \otimes b) = [a_\Lambda b], \quad 
\Res _\Lambda\left(\lambda^{-1}F(a \otimes b)\right) =ab.
\end{split}
\end{align}
\end{lem}

\begin{proof}
Let us define a linear map $F\colon V \otimes V \to V[\Lambda]$ by 
\begin{align*}
 F(a \otimes b)\ceq 
 \sum_{I \subset [N]}(-1)^{\# I  (N+1)}\sigma(I)\theta^{[N] \bs I}(S^Ia)b
+\int_0^\lambda d\lambda \,[a_\Lambda b] \quad (a,b \in V).
\end{align*}
Then $F$ has the parity $\ol{N}$ and satisfies the conditions \eqref{eq:2:intbra}. 
Conversely, if the linear maps $F, G\colon V\otimes V \to V[\Lambda]$ satisfy \eqref{eq:2:intbra}, then we have the identities
\begin{align*}
 \pdd_\lambda F(a\otimes b)=\pdd_\lambda G(a\otimes b), \quad 
 \Res_\Lambda(\lambda^{-1}\theta^IF(a\otimes b))
 =\Res_\Lambda(\lambda^{-1}\theta^IG(a\otimes b)) \quad (I\subset [N]), 
\end{align*}
which yield $F=G$. 
\end{proof}

\begin{dfn}\label{dfn:W:indef}
Given a linear map $[\cdot_\Lambda\cdot]\colon V \otimes V \to V[\Lambda]$ of parity $\ol{N}$ satisfying the sesquilinearity \eqref{eq:LCA:sesq} and an even linear map $\mu\colon V\otimes V\to V$ satisfying \eqref{eq:VA:der}, we denote by 
\[
 \int^\Lambda d\Gamma[\cdot_\Gamma \cdot]\colon V \otimes V \lto V[\Lambda]
\]
the linear map $F$ in the \cref{lem:intbra}, and call it the integral of $[\cdot_\Lambda \cdot]$ (with respect to $\mu$). 
\end{dfn}
%


\cref{prp:W:skecom,prp:W:Jqas} below are SUSY analogues of \cite[Definition 1.24]{DK}.

\begin{prp}\label{prp:W:skecom}
Let $[\cdot_\Lambda\cdot]\colon V \otimes V \to V[\Lambda]$ be a linear map of parity $\ol{N}$ satisfying the sesquilinearity \eqref{eq:LCA:sesq}, and $\mu\colon V \otimes V\to V$ be an even linear map satisfying \eqref{eq:VA:der}. Then the skew-symmetry \eqref{eq:LCA:ssym} and the quasi-commutativity \eqref{eq:VA:qcom} are equivalent to the following identity: 
\begin{align}\label{eq:W:skecom}
  \int^\Lambda d\Gamma [b_\Gamma a] = 
 (-1)^{p(a)p(b)+\ol{N}}\int^{-\Lambda-\nabla}d\Gamma [a_\Gamma b]. 
\end{align}
\end{prp}

\begin{proof}
By a direct calculation, we get
\begin{align*}
 &\Res_\Lambda\Bigl(\lambda^{-1}\int^{-\Lambda-\nabla}d\Gamma[a_\Gamma S^ib]\Bigr)
 =(-1)^{p(a)}\Res_\Lambda\Bigl(\lambda^{-1}\theta^i\int^{-\Lambda-\nabla}d\Gamma[a_\Gamma b]\Bigr), \\
&\pdd_\lambda\int^{-\Lambda-\nabla}d\Gamma[a_\Gamma b]=-[a_{-\Lambda-\nabla} b], \quad  
\Res_{\Lambda}\Bigl(\lambda^{-1}\int^{-\Lambda-\nabla}d\Gamma[a_\Gamma b]\Bigr)
 =(-1)^N\Bigl(ab-\int_{-T}^0d\Lambda[a_\Lambda b]\Bigr).
\end{align*}
Thus the skew-symmetry and the quasi-commutativity are equivalent to the identity \eqref{eq:W:skecom}. 
\end{proof}

\begin{prp}\label{prp:W:Jqas}
Let $[\cdot_\Lambda\cdot]\colon V \otimes V \to V[\Lambda]$ and $\mu\colon V \otimes V \to V$ be the same in \cref{prp:W:skecom}. In addition, we assume that $[\cdot_\Lambda\cdot]$ and $\mu$ satisfy the skew-symmetry and the quasi-commutativity. Then the Jacobi identity \eqref{eq:LCA:Jac}, the quasi-associativity \eqref{eq:VA:qass} and the Wick formula \eqref{eq:VA:Wick} are equivalent to the following identity: 
\begin{align}\label{eq:W:Jqas}
\begin{split}
&\int^{\Lambda_1}d\Gamma_1\left[a_{\Gamma_1}\left(\int^{\Lambda_2}d\Gamma_2[b_{\Gamma_2}c]\right)\right]
-(-1)^{(p(a)+\ol{N})(p(b)+\ol{N})}\int^{\Lambda_2}d\Gamma_2\left[b_{\Gamma_2}\left(\int^{\Lambda_1}d\Gamma_1[a_{\Gamma_1}c]\right)\right]\\
&=(-1)^{(p(a)+\ol{N})\ol{N}}\int^{\Lambda_1+\Lambda_2}d\Gamma\left[\left(\int^{\Lambda_1}d\Gamma_1[a_{\Gamma_1}b]-\int^{-\Lambda_2-\nabla}d\Gamma_2[a_{\Gamma_2}b]\right)_{\Gamma}c\right]. 
\end{split}
\end{align}
\end{prp}

\begin{proof}
Let $F(a\otimes b\otimes c)$ (resp.\ $G(a\otimes b\otimes c)$) denote the left hand side (resp.\ the right hand side) of \cref{eq:W:Jqas}. 

The identity 
$\pdd_{\lambda_1}\pdd_{\lambda_2}F(a\otimes b\otimes c)
=\pdd_{\lambda_1}\pdd_{\lambda_2}G(a\otimes b\otimes c)$ 
is equivalent to the Jacobi identity: In fact, a direct calculation yields
\begin{align*}
&\pdd_{\lambda_1}\pdd_{\lambda_2}F(a\otimes b\otimes c)
=[a_{\Lambda_1}[b_{\Lambda_2}c]]-(-1)^{(p(a)+\ol{N})(p(b)+\ol{N})}[b_{\Lambda_2}[a_{\Lambda_1}b]],
\\
&\pdd_{\lambda_1}\pdd_{\lambda_2}G(a\otimes b\otimes c)
=(-1)^{(p(a)+\ol{N})\ol{N}}[[a_{\Lambda_1}b]_{\Lambda_1+\Lambda_2}c].
\end{align*}

The identity 
$\pdd_{\lambda_1}\Res_{\Lambda_2}(\lambda_2^{-1}F(a\otimes b\otimes c))
=\pdd_{\lambda_1}\Res_{\Lambda_2}(\lambda_2^{-1}G(a\otimes b\otimes c))$ 
is equivalent to the Wick formula: In fact, one can get by direct calculation that 
\begin{align*}
&\pdd_{\lambda_1}\Res_{\Lambda_2}(\lambda_2^{-1}F(a\otimes b\otimes c))
=(-1)^{(p(a)+\ol{N})\ol{N}}[a_{\Lambda_1}bc]
 -(-1)^{(p(a)+\ol{N})(p(b)+\ol{N})}b[a_{\Lambda_1}c],\\
&\pdd_{\lambda_1}\Res_{\Lambda_2}(\lambda_2^{-1}G(a\otimes b\otimes c))
=(-1)^{(p(a)+\ol{N})\ol{N}}
  \left([a_{\Lambda_1}b]c+\int_0^{\lambda_1}d\Gamma [[a_{\Lambda_1}b]_\Gamma c]\right). 
\end{align*}

The identity 
$\Res_{\Lambda_1}\Res_{\Lambda_2}(\lambda_1^{-1}\lambda_2^{-1}F(a\otimes b\otimes c))
=\Res_{\Lambda_1}\Res_{\Lambda_2}(\lambda_1^{-1}\lambda_2^{-1}G(a\otimes b\otimes c))$  
is equivalent to the identity \eqref{eq:lsym}: In fact, we have
\begin{align*}
 &\Res_{\Lambda_1}\Res_{\Lambda_2}(\lambda_1^{-1}\lambda_2^{-1}F(a\otimes b\otimes c))
 =(-1)^{(p(a)+\ol{N})\ol{N}}(a(bc)-(-1)^{p(a)p(b)}b(ac)),\\
 &\Res_{\Lambda_1}\Res_{\Lambda_2}(\lambda_1^{-1}\lambda_2^{-1}G(a\otimes b\otimes c))
 =(-1)^{(p(a)+\ol{N})\ol{N}}(ab-(-1)^{p(a)p(b)}ba)c.
\end{align*}

Since we have
\begin{align*}
\Res_{\Lambda_1}(\lambda_1^{-1}\theta_1^IF(a\otimes b\otimes c))
&=(-1)^{\#I(N+1)}\Res_{\Lambda_1}(\lambda_1^{-1}F(S^Ia\otimes b\otimes c)), \\
\Res_{\Lambda_1}(\lambda_1^{-1}\theta_1^IG(a\otimes b\otimes c))
&=(-1)^{\#I(N+1)}\Res_{\Lambda_1}(\lambda_1^{-1}G(S^Ia\otimes b\otimes c)),
\end{align*}
and
\begin{align*}
\Res_{\Lambda_2}(\lambda_2^{-1}\theta_2^IF(a\otimes b\otimes c))
&=(-1)^{\ol{\#I}(p(a)+\od)}\Res_{\Lambda_2}(\lambda_2^{-1}F(a\otimes S^Ib\otimes c)), \\
\Res_{\Lambda_2}(\lambda_2^{-1}\theta_2^IG(a\otimes b\otimes c))
&=(-1)^{\ol{\#I}(p(a)+\od)}\Res_{\Lambda_2}(\lambda_2^{-1}G(a\otimes S^Ib\otimes c)),
\end{align*}
for any $I\subset [N]$, the identity $F(a\otimes b\otimes c)=G(a\otimes b\otimes c)$ holds if and only if the following identities hold:
\begin{gather*}
 \pdd_{\lambda_1}\pdd_{\lambda_2}F(a\otimes b\otimes c)
=\pdd_{\lambda_1}\pdd_{\lambda_2}G(a\otimes b\otimes c), \\
 \pdd_{\lambda_1}\Res_{\Lambda_2}(\lambda_2^{-1}F(a\otimes b\otimes c))
=\pdd_{\lambda_1}\Res_{\Lambda_2}(\lambda_2^{-1}G(a\otimes b\otimes c)), \\
 \Res_{\Lambda_1}\Res_{\Lambda_2}(\lambda_1^{-1}\lambda_2^{-1}F(a\otimes b\otimes c))
=\Res_{\Lambda_1}\Res_{\Lambda_2}(\lambda_1^{-1}\lambda_2^{-1}G(a\otimes b\otimes c)).
\end{gather*}
Hence, by \cref{lem:lsym}, the Jacobi identity, the quasi-associativity and the Wick formula are equivalent to the identity \eqref{eq:W:Jqas}. 
\end{proof}

Now we introduce several superalgebras 
which are $N_W=N$ SUSY analogue of those in \cite[\S6.3]{BDHK}. 

\begin{dfn}\label{dfn:W:O}
Let $n \in \bbZ_{>0}$, and $Z_k=(z_k, \zeta_k^1, \ldots, \zeta_k^N)$ be a $(1|N)_{W}$-supervariable for each $k \in [n]$. We set 
\begin{align}\label{eq:W:zkl}
 z_{k,l}\ceq z_k-z_l, \quad \zeta_{k,l}^{i} \ceq \zeta_k^i-\zeta_l^i
\end{align}
for $i \in [N]$ and $k,l \in [n]$. Also, we set $Z_{k,l} \ceq (z_{k,l},\zeta_{k,l}^1,\dotsc,\zeta_{k,l}^N)$ for simplicity.
\begin{enumerate}
\item
We denote by 
\[
 \clO_n^{\star} \ceq \bbK[Z_k]_{k=1}^n[z_{k, l}^{-1}]_{1 \le k < l \le n}
\]
the localization of $\bbK[Z_k]_{k=1}^n$ by the multiplicatively closed set generated by $\{z_{k,l} \mid 1 \le k<l \le n\}$.

\item
Let $\clO_n^{\star \rmT}$ denote the sub-superalgebra of $\clO_n^{\star}$ generated by the subset $\{z_{k, l}^{\pm 1}\mid 1\le k<l\le n\} \cup \{\zeta_{k,l}^i \mid i \in [N], \, 1 \le k<l \le n\}$. Thus, we have 
\[
 \clO_n^{\star \rmT} = \bbK[z_{k,l}^{\pm1},\zeta_{k,l}^i \mid i \in [N], \ 1 \le k<l \le n].
\]

\item 
Let $\clD_n$ denote the superalgebra of regular differential operators of $Z_k$,  
i.e, the sub-superalgebra of $\End_{\bbK}(\clO_n^{\star})$ generated by $Z_k=(z_k,\zeta_k^1,\dotsc,\zeta_k^N)$ and $\pdd_{Z_k}=(\pdd_{z_k}, \pdd_{\zeta_k^1}, \ldots, \pdd_{\zeta_k^N})$ for $k \in [n]$. As a superspace, we have
\begin{align*}
 \clD_n = \bbK[Z_k]_{k=1}^n [\pdd_{Z_k}]_{k=1}^n = 
 \bbK[z_k,\zeta_k^1,\dotsc,\zeta_k^N]_{k=1}^n
     [\pdd_{z_k},\pdd_{\zeta_k^1},\dotsc,\pdd_{\zeta_k^N}]_{k=1}^n.
\end{align*}
Here $\bbK[Z_k]_{k=1}^n[\pdd_{Z_k}]_{k=1}^n$ is the free commutative superalgebra over $\bbK[Z_k]_{k=1}^n$ generated by even $\pdd_{z_k}$ and odd $\pdd_{\zeta_k^i}$ for $i \in [N]$ and $k \in [n]$. 

\item 
Let $\clD_n^{\rmT}$ denote the sub-superalgebra of $\clD_n$ generated by the subset $\{z_{k,l}, \zeta_{k, l}^i\mid i\in[N], 1\le k<l\le n\}\cup \{\pdd_{z_k}, \pdd_{\zeta_k^i}\mid i\in[N], k\in[n]\}$. Thus we have
\begin{align*}
 \clD_n^{\rmT}=\bbK[Z_{k, l}]_{1\le k<l\le n}[\pdd_{Z_k}]_{k=1}^n. 
\end{align*}
\end{enumerate}

By convention, we denote $\clO_0^{\star}=\clO_0^{\star\rmT}=\clD_0=\clD_0^{\rmT}\ceq \bbK$. 
\end{dfn}

As in the non-SUSY case \cite[\S6.2]{BDHK}, the superalgebra $\clO_n^{\star\rmT}$ is the translation invariant subalgebra of $\clO_n^{\star}$. In other words, we have
\begin{align}\label{eq:W:Ker-delta}
 \clO_n^{\star\rmT} = \Ker(\delta^0) \cap \Ker(\delta^1) \cap \dotsb \cap \Ker(\delta^N) 
 \subset \clO_n^{\star}
\end{align}
with $\delta^0 \ceq \sum_{k=1}^n \pdd_{z_k}$ and $\delta^i \ceq \sum_{k=1}^n \pdd_{\zeta_k^i}$ for $i \in [N]$.

In this subsection, we fix an $\clH_W$-supermodule $(V, \nabla)$. 

Let $n\in\bbZ_{>0}$. 
The space $V^{\otimes n}\otimes \clO_n^{\star\rmT}$ carries the structure of a right $\clD_n^{\rmT}$-supermodule by letting $Z_{k, l}=(z_{k, l}, \zeta_{k, l}^1, \ldots, \zeta_{k, l}^N)$ act as
\begin{align*}
 (v \otimes f) \cdot z_{k,l} \ceq v \otimes f z_{k,l}, \quad 
 (v \otimes f) \cdot \zeta_{k,l}^i \ceq v \otimes f \zeta_{k,l}^i
\end{align*}
and $\pdd_{Z_k}=(\pdd_{z_k}, \pdd_{\zeta_k^1},\ldots,\pdd_{\zeta_k^N})$ act as
\begin{align*}
 (v \otimes f) \cdot \pdd_{z_k}
&\ceq T^{(k)}v \otimes f - v \otimes \pdd_{z_k}f, \\ 
 (v \otimes f)\cdot \pdd_{\zeta_k^i}
&\ceq (-1)^{p(v)+p(f)}S^{(k)}v \otimes f - (-1)^{p(f)}v \otimes \pdd_{\zeta_k^i}f. 
\end{align*}
for each $v\in V^{\otimes n}$ and $f \in \clO_n^{\star\rmT}$. Here, for a linear transformation $\varphi$ on $V$, the symbol $\varphi^{(k)}$ denotes the linear transformation on $V^{\otimes n}$ defined by $\varphi^{(k)}\ceq \id_V\otimes \cdots \otimes \overset{k}{\varphi} \otimes \cdots \otimes \id_V$. 

Now let us recall the linear superspace $V_\nabla[\Lambda_k]_{k=1}^n$ in \eqref{eq:LCA:HWn-mod}.
It also has the structure of a right $\clD_n^{\rmT}$-supermodule by letting
\begin{align*}
&(a(\Lambda_1,\ldots,\Lambda_n)\otimes v) \cdot z_{k,l}
 \ceq (\pdd_{\lambda_l}-\pdd_{\lambda_k})a(\Lambda_1,\ldots,\Lambda_n) \otimes v, \\
&(a(\Lambda_1,\ldots,\Lambda_n) \otimes v) \cdot \zeta_{k,l}^i
 \ceq (-1)^{p(a)+p(v)} (\pdd_{\zeta_l^i}-\pdd_{\zeta_k^i}) a(\Lambda_1,\ldots,\Lambda_n) \otimes v, 
\end{align*}
and 
\begin{align*}
&(a(\Lambda_1,\ldots,\Lambda_n) \otimes v) \cdot \pdd_{z_k}
 \ceq -\lambda_k a(\Lambda_1,\ldots,\Lambda_n) \otimes v, \\
&(a(\Lambda_1,\ldots,\Lambda_n) \otimes v) \cdot \pdd_{\zeta_k^i}
 \ceq -(-1)^{p(a)+p(v)} \theta_k^ia(\Lambda_1,\ldots,\Lambda_n) \otimes v
\end{align*}
for each $a(\Lambda_1,\ldots,\Lambda_n) \in \bbK[\Lambda_k]_{k=1}^n$ and $v \in V$. 

Now we introduce the main object of this \cref{ss:W:VA}, which is an $N_W=N$ SUSY analogue of the operad $\oPch{}$ in \cite[\S6.3]{BDHK}. 

\begin{dfn}
For an $\clH_W$-supermodule $V=(V,\nabla)$ and $n \in \bbN$, we denote
\begin{align*}
 \oPchW{V}(n) \ceq \Hom_{\clD_n^\rmT}
 \bigl(V^{\otimes n} \otimes \clO_n^{\star \rmT},V_{\nabla}[\Lambda_k]_{k=1}^n\bigr). 
\end{align*}
In other words, $\oPchW{V}(n)$ is a linear sub-superspace of $\Hom_\bbK\bigl(V^{\otimes n} \otimes \clO_n^{\star \rmT},V_{\nabla}[\Lambda_k]_{k=1}^n\bigr)$ spanned by elements $X$ such that 
\begin{align*}
 X((v \otimes f) \cdot \varphi) = X(v \otimes f) \cdot \varphi \quad 
 (v \in V^{\otimes n}, \, f \in \clO_n^{\star\rmT}, \, \varphi \in \clD_n^{\rmT}). 
\end{align*}
To stress the variables $\Lambda_1,\ldots,\Lambda_n$, we express an element $X \in \oPchW{V}(n)$ as $X_{\Lambda_1,\ldots,\Lambda_n}$. 
Note that we have 
\begin{align*}
&\oPchW{V}(0) = \Hom_{\bbK}(\bbK \otimes \bbK, V/\nabla V)\cong V/\nabla V, \\
&\oPchW{V}(1) = \Hom_{\bbK[\pdd_Z]}(V \otimes \bbK, V_\nabla[\Lambda]) \cong \End_{\clH_W}V. 
\end{align*}
\end{dfn}

For $\sigma \in \frS_n$ and $X \in \oPchW{V}(n)$, we define a linear map $X^\sigma\colon V^{\otimes n} \otimes \clO_n^{\star \rmT} \to V_{\nabla}[\Lambda_k]_{k=1}^n$ by
\begin{align*}
 X^\sigma(v_1 \otimes \cdots \otimes v_n \otimes f) \ceq 
 X_{\sigma(\Lambda_1,\ldots,\Lambda_n)}(\sigma(v_1 \otimes \cdots \otimes v_n) \otimes \sigma f)
\end{align*}
for $v_1, \ldots, v_n\in V$ and $f \in \clO_n^{\star \rmT}$, where 
\begin{align*}
 \sigma(\Lambda_1,\ldots,\Lambda_n) &=
 (\Lambda_{\sigma^{-1}(1)},\ldots,\Lambda_{\sigma^{-1}(n)}), \\
 \sigma(v_1\otimes \cdots \otimes v_n) &= 
 \prod_{\substack{1 \le k < l \le n \\ \sigma(k) > \sigma(l)}} (-1)^{p(v_k)p(v_l)}
  (v_{\sigma^{-1}(1)}\otimes \cdots \otimes v_{\sigma^{-1}(n)})
\end{align*}
as before and 
\begin{align*}
 (\sigma f)(Z_1,\ldots,Z_n) \ceq f(Z_{\sigma(1)},\ldots,Z_{\sigma(n)}). 
\end{align*}
Then, one can verify $X^\sigma \in \oPchW{V}(n)$ and that the linear superspace $\oPchW{V}(n)$ is a right $\frS_n$-supermodule by this action. Thus we have an $\frS$-supermodule 
\[
 \oPchW{V} \ceq \bigl(\oPchW{V}(n)\bigr)_{n\in\bbN}. 
\]

\begin{dfn}
Let $m,n \in \bbN$, $(n_1,\ldots,n_m) \in \bbN^m_n$, and set $N_j \ceq n_1+\cdots+n_j$ for $j \in [m]$ and $N_0\ceq 0$. 
\begin{enumerate}
\item
For $Y_j\in \oPchW{V}(n_j)$ with $j\in[m]$, we define a linear map
\begin{align*}
 Y_1 \odot \cdots \odot Y_m\colon V^{\otimes n} \otimes \clO^{\star\rmT}_n \lto 
 \bigotimes_{j=1}^m V[\Lambda_k]_{k=N_{j-1}+1}^{N_j-1} \otimes \clO^{\star\rmT}_m
\end{align*}
by 
\begin{align*}
&(Y_1 \odot \cdots \odot Y_m)(v_1 \otimes \cdots \otimes v_n \otimes f) \\
&\ceq \pm (Y_1)_{\Gamma_1,\ldots,\Gamma_{N_1-1}}(w_1 \otimes f_1) \otimes \cdots 
 \otimes (Y_m)_{\Gamma_{N_{m-1}+1},\ldots,\Gamma_{N_m-1}}(w_m \otimes f_m)
 \otimes f_0|_{Z_k=Z_{N_j} (N_{j-1}< k\le N_j)}. 
\end{align*}
for $v_1,\ldots,v_n \in V$ and $f \in \clO^{\star\rmT}_n$. 
Here we denote by
\begin{align*}
 (Y_j)_{\Lambda_1,\ldots,\Lambda_{n_j-1}}(w)
 \ceq (Y_j)_{\Lambda_1,\ldots,\Lambda_{n_j-1},-\Lambda_1,\ldots,-\Lambda_{n_j-1}-\nabla}(w) 
 \quad (w \in V^{\otimes n_j}),
\end{align*} 
the element of $V[\Lambda_k]_{k=1}^{n_j-1}$ corresponding to $Y_j(w)$ by the isomorphism \eqref{eq:LCA:HWn-mod}. 
Also, we set
\begin{align*}
&\pm\ceq \prod_{1\le i<j\le m}(-1)^{(p(w_i)+p(f_i))p(Y_j)}, \quad
 w_j\ceq v_{N_{j-1}+1} \otimes \cdots \otimes v_{N_j} \quad (j \in [m]), \\
& \Gamma_k \ceq \Lambda_k-\pdd_{Z_k} \quad (k \in [n] \bs \{N_1,\ldots,N_m\}), 
\end{align*}
 and
\begin{align}\label{eq:decompf}
 f = f_0 f_1 \cdots f_m, \quad f_0 \in \clO_n^{\star\rmT}, \quad
 f_j = \prod_{N_{j-1}<k<l\le N_j} z_{k,l}^{-m_{k,l}^j} \quad (m_{k,l}^j \in \bbN)
\end{align}
is a decomposition such that $f_0$ has no poles at $z_k=z_l$ ($N_{j-1}<k<l\le N_j$, $j\in[m]$). 
Note that the $Y_1\odot \cdots \odot Y_m$ is independent of the choice of decomposition \eqref{eq:decompf} as in the proof of \cite[Lemma 6.6]{BDHK}. 

\item 
For $X \in \oPchW{V}(m)$ and $Y_j \in \oPchW{V}(n_j)$ with $j \in [m]$, we denote by
\begin{align*}
 X \circ (Y_1 \odot \cdots \odot Y_m)\colon 
 V^{\otimes n} \otimes \clO^{\star\rmT}_n \lto V_\nabla[\Lambda_k]_{k=1}^n
\end{align*}
the composition of linear maps 
\[
 V^{\otimes n} \otimes \clO^{\star\rmT}_n \xrr{Y_1 \odot \cdots \odot Y_m} 
 \bigotimes_{j=1}^m V[\Lambda_k]_{k=N_{j-1}+1}^{N_j-1} \otimes \clO^{\star\rmT}_m
 \xrr{X_{\Lambda'_1,\ldots,\Lambda'_m}}
 V_\nabla[\Lambda_k]_{k=1}^n. 
\]
Here we set $\Lambda'_j \ceq \Lambda_{N_{j-1}+1}+\cdots+\Lambda_{N_j}$. 
Also, the symbol $X_{\Lambda'_1,\ldots,\Lambda'_m}$ is the linear map defined by
\begin{align*}
 a_1 v_1 \otimes \cdots \otimes a_m v_m \otimes f \lmto 
 \pm (a_1 \cdots a_m)X_{\Lambda'_1,\ldots,\Lambda'_m}(v_1 \otimes \cdots \otimes v_m \otimes f)
\end{align*}
for each $a_j \in \bbK[\Lambda_k]_{k=N_{j-1}+1}^{N_j-1}$, $v_j \in V$ ($j \in [m]$) 
and $f \in \clO^{\star\rmT}_m$ with the sign
\begin{align*}
 \pm \ceq \prod_{1 \le i < j \le m} (-1)^{p(v_i)p(a_j)} \cdot \prod_{j=1}^m (-1)^{p(a_j)p(X)}. 
\end{align*}
\end{enumerate}
\end{dfn}

Let $m,n \in \bbN$ and $\nu \in \bbN^m_n$. For $X \in \oPchW{V}(m)$ and $Y_1 \otimes \cdots \otimes Y_m \in \oPchW{V}(\nu)$, a direct calculation shows that
\begin{align*}
 X \circ (Y_1 \odot \cdots \odot Y_m) \in \oPchW{V}(n). 
\end{align*}

\begin{lem}
The $\frS$-supermodule
\[
 \oPchW{V} = \bigl(\oPchW{V}(n)\bigr)_{n \in \bbN}
\]
is a superoperad by letting $X \otimes Y_1 \otimes \cdots \otimes Y_m \mto X \circ (Y_1 \odot \cdots \odot Y_m)$ be the composition map and $\id_V$ be the identity. 
\end{lem}

\begin{proof}
The proof of non-SUSY case in \cite[Proposition 6.7]{BDHK} works with minor modifications. 
We omit the detail. 
\end{proof}

Summarizing the discussion so far, we name the obtained superoperad as follows. Recall the commutative superalgebra $\clH_W$ in \cref{dfn:W:clHW}.

\begin{dfn}\label{dfn:W:Pch}
For an $\clH_W$-supermodule $V=(V,\nabla)$, we call the superoperad $\oPchW{V}$ the \emph{$N_W=N$ SUSY chiral operad}.
\end{dfn}

Recall the symbol $\Pi^{k}V = (\Pi^{k}V,\nabla)$, which is explained immediately after \cref{dfn:W:Chom}. 
Also recall the set $\MC\bigl(L(\oQ)\bigr)$ of Maurer-Cartan solutions $X \square X = 0$, $X \in L(\oQ)^1$ in the graded Lie superalgebra $L(\oQ)$ associated to a superoperad $\oQ$, explained in \cref{ss:1:LQ}. 
The following is the main statement of this \cref{ss:W:VA}, which is a natural $N_W=N$ SUSY analogue of \cite[Theorem 6.12]{BDHK}.

\begin{thm}\label{thm:W:VA}
Let $(V,\nabla)=(V,T,S^1,\dotsc,S^N)$ be an $\clH_W$-supermodule.
\begin{enumerate}
\item
For an odd Maurer-Cartan solution $X \in \MC\bigl(L(\oPchW{\Pi^{N+1}V})\bigr)_{\od}$, define linear maps $[\cdot_\Lambda\cdot]_X\colon V \otimes V \to V[\Lambda]$ and $\mu_X\colon V \otimes V \to V$ by
\begin{align}
 \label{eq:VAlbra}
 [a_\Lambda b]_X &\ceq 
 (-1)^{p(a)(\ol{N}+\od)}X_{\Lambda, -\Lambda-\nabla}(a \otimes b \otimes 1_\bbK), \\
 \label{eq:VAprod}
 \mu_X(a \otimes b) &\ceq (-1)^{p(a)(\ol{N}+\od)+\od}
 \Res_\Lambda \bigl(\lambda^{-1}X_{\Lambda, -\Lambda-\nabla}(a \otimes b \otimes z_{1, 2}^{-1})\bigr)
\end{align}
for each $a,b \in V$. 
Then $(V,\nabla,[\cdot_\Lambda \cdot]_X,\mu_X)$ is a non-unital $N_W=N$ SUSY vertex algebra. 
\item
The map $X \mto ([\cdot_\Lambda \cdot]_X,\mu_X)$ gives a bijection 
\[
 \MC\bigl(L(\oPchW{\Pi^{N+1}V})\bigr)_{\od} \lsto 
 \{\text{non-unital $N_W=N$ SUSY vertex algebra structures on $(V,\nabla)$}\}.
\] 
\end{enumerate}
\end{thm}

In the rest of this section, we prove \cref{thm:W:VA}. We denote $\wt{V} \ceq \Pi^{N+1}V$ and $\oP \ceq \oPchW{\Pi^{N+1}V}$ for simplicity, and let $p$, $\wt{p}$ be the parity of $V$, $\wt{V}$ respectively. Consider the linear maps $[\cdot_\Lambda\cdot]_X\colon V \otimes V \to V[\Lambda]$ and $\mu_X\colon V \otimes V \to V$ defined by \cref{eq:VAlbra} and \cref{eq:VAprod} for an odd element $X \in \oP(2)_{\ol{1}}$.

\begin{itemize}
\item 
The linear map $[\cdot_\Lambda\cdot]_X\colon V \otimes V \to V[\Lambda]$ has the parity $\ol{N}$: 
In fact, since $X$ is an odd linear map $X\colon \wt{V}^{\otimes 2} \otimes \clO^{\star \rmT}_2 \to \wt{V}_\nabla[\Lambda_k]_{k=1,2}$, 
\begin{align*}
 p(X(a \otimes b \otimes 1_\bbK)) = \wt{p}(X(a \otimes b\otimes 1_\bbK))+\ol{N}+\od
 = \wt{p}(a)+\wt{p}(b)+\ol{N} = p(a)+p(b)+\ol{N}
\end{align*}
holds for $a,b \in V$.
  
\item 
The linear map $\mu_X\colon V \otimes V \to V$ is even: 
This is clear since the residue map $\Res_\Lambda\colon \wt{V}[\Lambda] \to \wt{V}$ has the parity $\ol{N}$. 

\item 
The linear map $[\cdot_\Lambda\cdot]_X$ satisfies (i) in \cref{dfn:W:LCA}: 
This can be checked by a direct calculation similar to the proof of \cref{thm:W:LCA}. 

\item 
The linear map $\mu_X$ satisfies (i) in \cref{dfn:NWVA}: Since 
\begin{align*}
 &X_{\Lambda, -\Lambda-\nabla}(Ta\otimes b\otimes z_{1, 2}^{-1})
 =-\lambda X_{\Lambda, -\Lambda-\nabla}(a\otimes b\otimes z_{1, 2}^{-1})
 -X_{\Lambda, -\Lambda-\nabla}(a\otimes b\otimes z_{1, 2}^{-2}), \\
 &X_{\Lambda, -\Lambda-\nabla}(a\otimes Tb\otimes z_{1, 2}^{-1})
 =(\lambda+T)X_{\Lambda, -\Lambda-\nabla}(a\otimes b\otimes z_{1, 2}^{-1})
 +X_{\Lambda, -\Lambda-\nabla}(a\otimes b\otimes z_{1, 2}^{-2})
\end{align*}
hold for $a,b \in V$, we have $T(ab)=(Ta)b+a(Tb)$. 
Similarly, we can prove $S^i(ab)=(S^ia)b+(-1)^{p(a)}a(S^ib)$. 
\end{itemize}

Note now that there exists the integral of $[\cdot_\Lambda\cdot]_X$. By \cref{lem:intbra}, we have 
\begin{align*}
 \int^\Lambda d\Gamma[a_\Gamma b]_X = (-1)^{p(a)(\ol{N}+\od)+\od} 
 X_{\Lambda, -\Lambda-\nabla}(a \otimes b \otimes z_{1, 2}^{-1}) \quad (a,b \in V).
\end{align*}
Thus we find that the map $X \mto ([\cdot_\Lambda\cdot]_X, \mu_X)$ gives a bijective correspondence between the set $\oP(2)_{\od}$ and the set of all pairs $([\cdot_\Lambda\cdot], \mu)$ of a linear map $[\cdot_\Lambda\cdot]\colon V \otimes V \to V[\Lambda]$ of parity $\ol{N}$ satisfying (i) in \cref{dfn:W:LCA} and an even linear map $\mu\colon V \otimes V \to V$ satisfying (i) in \cref{dfn:NWVA}. 

\begin{itemize}
\item 
The linear maps $[\cdot_\Lambda\cdot]_X$ and $\mu_X$ satisfy the skew-symmetry \eqref{eq:LCA:ssym} and the quasi-commutativity $\eqref{eq:VA:qcom}$ if and only if $X^\sigma=X$ for any $\sigma \in \frS_2$: 
The element $X \in \oP(2)_{\od}$ satisfies $X^\sigma=X$ ($\sigma \in \frS_2$) if and only if 
\begin{align*}
 (-1)^{\wt{p}(a)\wt{p}(b)}X_{\Lambda_2,\Lambda_1}(b \otimes a \otimes z_{2, 1}^{-1})
 =X_{\Lambda_1,\Lambda_2}(a \otimes b \otimes z_{1, 2}^{-1}).
\end{align*}
Since we have, for $a,b \in V$, 
\begin{align*}
 \int^\Lambda d\Gamma[b_\Gamma a]
&= (-1)^{p(b)(\ol{N}+\od)+\od}X_{\Lambda,-\Lambda-\nabla}(b \otimes a \otimes z_{1,2}^{-1})\\
&= (-1)^{p(b)(\ol{N}+\od)}    X_{\Lambda,-\Lambda-\nabla}(b \otimes a \otimes z_{2,1}^{-1}), \\
 \int^{-\Lambda-\nabla} d\Gamma[a_\Gamma b]
&= (-1)^{p(a)(\ol{N}+\od)+\od}X_{-\Lambda-\nabla, \Lambda}(a \otimes b \otimes z_{1,2}^{-1}), 
\end{align*}
by \cref{prp:W:skecom}, the skew-symmetry and the quasi-commutativity are equivalent to $X^\sigma=X$ ($\sigma \in \frS_2$). 
\end{itemize}

In what follows, let $X \in L^1(\oP)_{\od}$. Thus the linear maps $[\cdot_\Lambda\cdot]_X\colon V \otimes V \to V[\Lambda]$ and $\mu_X\colon V \otimes V \to V$ satisfy (i), (ii) in \cref{dfn:W:LCA} and (i), (ii) in \cref{dfn:NWVA}. We need to prove:

\begin{itemize}
\item
The linear maps $[\cdot_\Lambda\cdot]_X\colon V \otimes V \to V[\Lambda]$ and $\mu_X\colon V \otimes V \to V$ satisfy the Jacobi identity \eqref{eq:LCA:Jac}, the quasi-associativity \eqref{eq:VA:qass} and the Wick formula \eqref{eq:VA:Wick} if and only if $X \square X=0$. 
\end{itemize}

First, we have the following lemmas. 

\begin{lem}\label{lem:boxJqW}
Let $X \in L^1\bigl(\oPchW{\Pi^{N+1}V}\bigr)_{\od}$. For any $a,b,c \in V$, we have
\begin{align*}
&\pm(X \square X)(a \otimes b \otimes c \otimes z_{1,3}^{-1} z_{2,3}^{-1}) \\
&=\int^{\Lambda_1}d\Gamma_1\left[a_{\Gamma_1}\left(\int^{\Lambda_2}d\Gamma_2[b_{\Gamma_2}c]\right)\right]
 -(-1)^{(p(a)+\ol{N})(p(b)+\ol{N})}
  \int^{\Lambda_2}d\Gamma_2\left[b_{\Gamma_2}\left(\int^{\Lambda_1}d\Gamma_1[a_{\Gamma_1}c]\right)\right] \\
&\qquad 
 -(-1)^{(p(a)+\ol{N})\ol{N}}
 \int^{\Lambda_1+\Lambda_2}d\Gamma\left[\left(\int^{\Lambda_1}d\Gamma_1[a_{\Gamma_1}b]
-\int^{-\Lambda_2-\nabla}d\Gamma_2[a_{\Gamma_2}b]\right)_{\Gamma}c\right],
\end{align*}
where $\pm \ceq (-1)^{\wt{p}(a)}(-1)^{p(a)(\ol{N}+\od)}(-1)^{p(b)(\ol{N}+\od)}$. 
\end{lem}

\begin{proof}
A direct calculation shows that 
\begin{align*}
 (X \circ_2 X)(a \otimes b \otimes c\otimes z_{1,3}^{-1} z_{2,3}^{-1})
&=\pm \int^{\Lambda_1} d\Gamma_1
  \Bigl[a_{\Gamma_1}\Bigl(\int^{\Lambda_2} d\Gamma_2[b_{\Gamma_2}c]\Bigr)\Bigr], \\
 (X \circ_2 X)^{(1, 2)}(a \otimes b\otimes c\otimes z_{1, 3}^{-1}z_{2,3}^{-1})
&=\pm(-1)^{(p(a)+\ol{N})(p(b)+\ol{N})+\od} \int^{\Lambda_2} d\Gamma_2
 \Bigl[b_{\Gamma_2}\Bigl(\int^{\Lambda_1} d\Gamma_1 [a_{\Gamma_1}c]\Bigr)\Bigr]. 
\end{align*}
Also, we have
\begin{align*}
& (X \circ_1 X)(a \otimes b \otimes c \otimes z_{1, 3}^{-1}z_{2, 3}^{-1})\\
&=(\pdd_{\lambda_2}-\pdd_{\lambda_1})(X \circ_1X)
  (a \otimes b \otimes c \otimes z_{1, 2}^{-1}z_{1, 3}^{-1}z_{2, 3}^{-1})\\
&=(\pdd_{\lambda_3}-\pdd_{\lambda_1})(X \circ_1X)
  (a \otimes b \otimes c \otimes z_{1, 2}^{-1}z_{1, 3}^{-1}z_{2, 3}^{-1})
  -(\pdd_{\lambda_3}-\pdd_{\lambda_2})(X \circ_1X)
  (a \otimes b \otimes c \otimes z_{1, 2}^{-1}z_{1, 3}^{-1}z_{2, 3}^{-1})\\
&=(X \circ_1 X)(a \otimes b \otimes c \otimes z_{1, 2}^{-1}z_{2, 3}^{-1})
 -(X \circ_1 X)(a \otimes b \otimes c \otimes z_{1, 2}^{-1}z_{1, 3}^{-1})
\end{align*}
and 
\begin{align*}
(X \circ_1X)(a \otimes b \otimes c\otimes z_{1, 2}^{-1}z_{2, 3}^{-1})
&=\pm(-1)^{(p(a)+\ol{N})\ol{N}+\od} \int^{\Lambda_1+\Lambda_2} d\Gamma
  \Bigl[\Bigl(\int^{\Lambda_1}d\Gamma_1[a_{\Gamma_1}b]\Bigr)_\Gamma c\Bigr], \\
(X \circ_1X)(a \otimes b \otimes c\otimes z_{1, 2}^{-1}z_{1, 3}^{-1})
&=\pm(-1)^{(p(a)+\ol{N})\ol{N}+\od}\int^{\Lambda_1+\Lambda_2} d\Gamma
  \Bigl[\Bigl(\int^{-\Lambda_2-\nabla}d\Gamma_2[a_{\Gamma_2}b]\Bigr)c\Bigr]. 
\end{align*}
Thus
\begin{align*}
&(X \circ_1 X)(a \otimes b \otimes c \otimes z_{1, 3}^{-1}z_{2, 3}^{-1})\\
&=\pm(-1)^{(p(a)+\ol{N})\ol{N}}\int^{\Lambda_1+\Lambda_2} d\Gamma
 \left[\left(\int^{\Lambda_1} d\Gamma_1 [a_{\Gamma_1}b]
            -\int^{-\Lambda_2-\nabla} d\Gamma_2 [a_{\Gamma_2}b]\right)_{\Gamma}c\right]. 
\end{align*}
The claim is now clear. 
\end{proof}

By \cref{prp:W:Jqas} and \cref{lem:boxJqW}, it is clear that $X\square X=0$ implies the Jacobi identity, the quasi-associativity and the Wick formula. 

Conversely, assume the Jacobi identity, the quasi-associativity and the Wick formula. 
To prove $X \square X=0$, it is sufficient to show that 
\begin{align}\label{eq:W:ResFn}
 \Res_{\Lambda_1}\Res_{\Lambda_2}\bigl(\lambda_1^{-1} \lambda_2^{-1}\pdd_{\lambda_1}^n(X \square X)
  (a \otimes b \otimes c \otimes z_{1,2}^{-1}z_{1,3}^{-1}z_{2,3}^{-1})\bigr) = 0
\end{align}
for all $n \in \bbN$ and $a,b,c \in V$, because we have
\begin{align*}
(\pdd_{\lambda_2}-\pdd_{\lambda_1})(X \square X)
 (a \otimes b \otimes c \otimes z_{1,2}^{-1} z_{1,3}^{-1} z_{2,3}^{-1})
=(X \square X)(a \otimes b \otimes c \otimes z_{1,3}^{-1} z_{2,3}^{-1})=0
\end{align*}
and 
\begin{align*}
& \Res_{\Lambda_1}\Res_{\Lambda_2}\bigl(\lambda_1^{-1}\lambda_2^{-1}
   \pdd_{\lambda_1}^n \theta_1^I \theta_2^J(X \square X)
   (a \otimes b \otimes c \otimes z_{1,2}^{-1} z_{1,3}^{-1} z_{2,3}^{-1})\bigr) \\
&=(-1)^{\#I}(-1)^{\wt{p}(a)\#J} \Res_{\Lambda_1}\Res_{\Lambda_2}\bigl(\lambda_1^{-1}\lambda_2^{-1}
   \pdd_{\lambda_1}^n(X \square X)
   (S^I a \otimes S^J b \otimes c \otimes z_{1, 2}^{-1}z_{1, 3}^{-1}z_{2, 3}^{-1})\bigr)
\end{align*}
for each $n \in \bbN$, $I,J \subset [N]$. 
Let us calculate 
$\Res_{\Lambda_2}\bigl(\lambda_2^{-1}(X \square X)(a \otimes b \otimes c \otimes z_{1,2}^{-1}z_{1,3}^{-1}z_{2,3}^{-1})\bigr)$. 
 
\begin{lem}\label{lem:W:circ}
Let $X \in \oPchK{\Pi^{N+1}V}(2)_{\od}$. For any $a, b, c\in V$, we have
\begin{align*}
 \pm\Res_{\Lambda_2}\bigl(\lambda_2^{-1}(X \circ_1 X)
 (a \otimes b \otimes c \otimes z_{1,2}^{-1}z_{1,3}^{-1}z_{2,3}^{-1})\bigr)
 &\begin{aligned}[t]
  &=\Bigl(\int_{-T}^{\lambda_1}d\lambda_1\int^{\Lambda_1}d\Gamma_1[a_{\Gamma_1}b]\Bigr)c\\
  & \qquad +\int_0^{\lambda_1}d\Lambda\Bigl[
  \Bigl(\int_{-T}^{\lambda_1}d\lambda_1\int^{\Lambda_1}d\Gamma_1[a_{\Gamma_1}b]\Bigr)_\Lambda c\Bigr], 
  \end{aligned}
 \\
 \pm \Res_{\Lambda_2}\bigl(\lambda_2^{-1}(X \circ_2 X)
  (a \otimes b \otimes c \otimes z_{1, 2}^{-1}z_{1, 3}^{-1}z_{2, 3}^{-1})\bigr)
 &\begin{aligned}[t]
  &=-\int^{\Lambda_1}d\Gamma_1[(\lambda_1+T)a_{\Gamma_1}bc]\\
  & \qquad -\int^{\Lambda_1}d\Gamma_1\Bigl[\Bigl(\int_0^{\lambda_1+T}d\lambda a\Bigr)_{\Gamma_1}
    \Bigl(\int_0^{\lambda}d\Lambda[b_\Lambda c]\Bigr)\Bigr], 
  \end{aligned}
 \\
 \pm\Res_{\Lambda_2}\bigl(\lambda_2^{-1}(X \circ_2 X)^{(1,2)}
 (a \otimes b \otimes c \otimes z_{1, 2}^{-1}z_{1, 3}^{-1}z_{2, 3}^{-1})\bigr)
 &=(-1)^{(p(a)+\ol{N})p(b)+\od}
 \Bigl(\int_{\lambda_1}^{\lambda_1+T}d\lambda_1b\Bigr) \int^{\Lambda_1}d\Gamma_1[a_{\Gamma_1}c]. 
\end{align*}
where $\pm \ceq (-1)^{p(b)(\ol{N}+\od)}$. 
\end{lem}
 
\begin{proof}
First, we get 
\begin{align*}
 &(X\circ_1X)(a\otimes b\otimes c\otimes  z_{1, 2}^{-1}z_{1, 3}^{-1}z_{2, 3}^{-1})\\
 &=X_{\Lambda_1+\Lambda_2, \Lambda_3}(X_{\Lambda_1\pdd_{Z_1}, -(\Lambda_1-\pdd_{Z_1})-\nabla}(a\otimes b\otimes z_{1, 2}^{-1})\otimes c\otimes z_{1, 3}^{-1}z_{2, 3}^{-1}|_{z_2=z_1})\\
 &=(-1)^{p(a)(\ol{N}+\od)+1}X_{\Lambda_1+\Lambda_2, \Lambda_3}
   \Bigl(\int^{\Lambda_1-\pdd_{Z_1}}d\Gamma_1[a_{\Gamma_1}b] \otimes c \otimes 
         z_{1, 3}^{-1}z_{2, 3}^{-1}|_{z_2=z_1}\Bigr). 
\end{align*}
For $F(\Lambda) \in V[\Lambda]$, we have
\begin{align*}
 F(\Lambda_1-\pdd_{Z_1})\otimes c\otimes z_{1, 3}^{-1}z_{2, 3}^{-1}|_{z_2=z_1}
 &=e^{-\pdd_{Z_1}\pdd_{\Lambda_1}}F(\Lambda_1)\otimes c\otimes z_{1, 3}^{-1}z_{2, 3}^{-1}|_{z_2=z_1}\\
 &=\sum_{n\ge 0}\frac{(-1)^n}{n!}\pdd_{\lambda_1}^nF(\Lambda_1)\otimes c\otimes (\pdd_{z_1}^nz_{1, 3}^{-1}z_{2, 3}^{-1})|_{z_2=z_1}\\
 &=\sum_{n\ge 0}\frac{(-1)^n}{n!}\pdd_{\lambda_1}^nF(\Lambda_1)\otimes c\otimes \frac{-1}{n+1}\pdd_{z_1}^{n+1}z_{1, 3}^{-1}\\
 &=-\int_0^{\pdd_{z_1}}d\Lambda\,\theta^{[N]}e^{-\Lambda\cdot\pdd_{\Lambda_1}}F(\Lambda_1)\otimes c\otimes z_{1, 3}^{-1}\\
 &=-\int_0^{\pdd_{z_1}}d\Lambda\, \theta^{[N]}F(\Lambda_1-\Lambda)\otimes c\otimes z_{1, 3}^{-1}. 
\end{align*}
 Thus
\begin{align*}
 &(X\circ_1X)(a\otimes b\otimes c\otimes  z_{1, 2}^{-1}z_{1, 3}^{-1}z_{2, 3}^{-1})\\
 &=(-1)^{p(a)(\ol{N}+\od)}X_{\Lambda_1+\Lambda_2, \Lambda_3}\Bigl(\int_0^{\pdd_{z_1}}d\Lambda\,\theta^{[N]}\int^{\Lambda_1-\Lambda}d\Gamma_1[a_{\Gamma_1}b]\otimes c\otimes z_{1, 3}^{-1}\Bigr)\\
 &=(-1)^{p(a)(\ol{N}+\od)}X_{\Lambda_1+\Lambda_2, \Lambda_3}\Bigl(\Bigl(\int_0^{\lambda_1+\lambda_2+T}d\Lambda\,\theta^{[N]}\int^{\Lambda_1-\Lambda}d\Gamma_1[a_{\Gamma_1}b]\Bigr)\otimes c\otimes z_{1, 3}^{-1}\Bigr)\\
 &=(-1)^{p(a)(\ol{N}+\od)+\od}X_{\Lambda_1+\Lambda_2, \Lambda_3}\Bigl(\Bigl(\int_{\lambda_1}^{-\lambda_2-T}d\lambda_1\int^{\Lambda_1}d\Gamma_1[a_{\Gamma_1}b]\Bigr)\otimes c\otimes z_{1, 3}^{-1}\Bigr)\\
 &=\pm\int^{\Lambda_1+\Lambda_2}d\Gamma\Bigl[\Bigl(\int^{\lambda_1}_{-\lambda_2-T}d\lambda_1\int^{\Lambda_1}d\Gamma_1[a_{\Gamma_1}b]\Bigr)_\Gamma c\Bigr], 
\end{align*}
which yields the first identity. 
Similarly, one can prove the third identity. Also, we have
\begin{align*}
&\Res_{\Lambda_2}\bigl(\lambda_2^{-1}(X \circ_2 X)
 (a \otimes b \otimes c \otimes z_{1, 2}^{-1}z_{1, 3}^{-1}z_{2, 3}^{-1})\bigr)\\
&=\pm(-1)^{p(a)(\ol{N}+\od)}
  X_{\Lambda_1, -\Lambda_1-\nabla}\Bigl(a \otimes 
   \Res_{\Lambda_2}\Bigl(\lambda_2^{-1}\int^{\Lambda_2-\pdd_{Z_2}}d\Gamma_2 [b_{\Gamma_2}c]\Bigr)
   \otimes z_{1, 2}^{-1}z_{1, 3}^{-1}|_{z_2=z_3}\Bigr). 
\end{align*}
Since a direct calculation shows that
\begin{align*}
\Res_{\Lambda_2}\Bigl(\lambda_2^{-1}\int^{\Lambda-\pdd_{Z_2}}d\Gamma_2\, [b_{\Gamma_2}c]\Bigr)
=bc+\int_0^{-\pdd_{z_2}}d\Lambda\, [b_\Lambda c], 
\end{align*}
we obtain
\begin{align*}
&\pm\Res_{\Lambda_2}\bigl(\lambda_2^{-1}(X \circ_2 X)
 (a \otimes b \otimes c \otimes z_{1, 2}^{-1}z_{1, 3}^{-1}z_{2, 3}^{-1})\bigr)\\
&=-\int^{\Lambda_1} d\Gamma_1 [(\lambda_1+T)a_{\Gamma_1}bc]
  -\int^{\Lambda_1}d\Gamma_1 
   \Bigl[\Bigl(\int_0^{\lambda_1+T}d\Lambda a\Bigr)_{\Gamma_1}
         \Bigl(\int_0^{\lambda}d\lambda[b_\Lambda c]\Bigr)\Bigr]. 
\end{align*}
\end{proof}

For $n \in \bbN$, let us denote
\begin{align}\label{eq:W:Fn}
 F_n(a\otimes b\otimes c) \ceq (-1)^{p(b)(\ol{N}+\od)} \pdd_{\lambda_1}^n \Res_{\Lambda_2}
 \bigl(\lambda_2^{-1}(X \square X)(a \otimes b \otimes c \otimes z_{1,2}^{-1}z_{1,3}^{-1}z_{2,3}^{-1})\bigr).
\end{align}
By \cref{lem:W:circ}, we have
\begin{align}\label{eq:W:F0}
\begin{split}
  F_0(a\otimes b\otimes c)
&=\Bigl(\int_{-T}^{\lambda_1}d\lambda_1\int^{\Lambda_1}d\Gamma_1[a_{\Gamma_1}b]\Bigr)c
 +\int_0^{\lambda_1}d\Lambda\Bigl[
   \Bigl(\int_{-T}^{\lambda_1}\int^{\Lambda_1}d\Gamma_1[a_{\Gamma_1}b]\Bigr)_\Lambda c\Bigr] \\
 &\quad 
 -\int^{\Lambda_1}d\Gamma_1[(\lambda_1+T)a_{\Gamma_1}bc]
 -\int^{\Lambda_1}d\Gamma_1\Bigl[\Bigl(\int_0^{\lambda_1+T}d\Lambda a\Bigr)_{\Gamma_1}
   \Bigl(\int_0^{\lambda_2}d\lambda[b_\Lambda c]\Bigr)\Bigr] \\
 &\quad 
 -(-1)^{(p(a)+\ol{N})p(b)} \Bigl(\int_{\lambda_1}^{\lambda_1+T}d\lambda_1b\Bigr)
   \int^{\Lambda_1}d\Gamma_1[a_{\Gamma_1}c]. 
\end{split}
\end{align} 

\begin{lem}\label{lem:ResF01}
For $n=0,1$ and $a,b,c \in V$, we have
\begin{align*}
\Res_{\Lambda_1}\Res_{\Lambda_2}\bigl(\lambda_1^{-1}
 \lambda_2^{-1}\pdd_{\lambda_1}^n(X\square X)(a\otimes b\otimes c\otimes 
 z_{1, 2}^{-1}z_{1, 3}^{-1}z_{2, 3}^{-1})\bigr)=0
\end{align*}
\end{lem}

\begin{proof}
Since we get
\begin{align*}
 &\Res_{\Lambda_1}(\lambda_1^{-1}F_0(a\otimes b\otimes c))\\
 &=\Bigl(((Ta)b)c-(Ta)(bc)-\Bigl(\int_0^Td\Lambda Ta\Bigr)[b_\Lambda c]
   -(-1)^{p(a)p(b)}\Bigl(\int_0^Td\Lambda b\Bigr)[Ta_\Lambda c]\Bigr) \\
 &\quad 
  +(-1)^{p(a)p(b)}\Bigl(((Tb)a)c-(Tb)(ac)-\Bigl(\int_0^Td\Lambda Tb\Bigr)
   [a_\Lambda c]-(-1)^{p(a)p(b)}\Bigl(\int_0^Td\Lambda a\Bigr)[Tb_\Lambda c]\Bigr), 
\end{align*}
the equation \eqref{eq:W:ResFn} for $n=0$ follows from the quasi-associativity. Also, by \eqref{eq:W:F0}
\begin{align*}
 F_1(a\otimes b\otimes c)
 &=\Bigl(\int^{\Lambda_1}d\Gamma_1[a_{\Gamma_1}b]\Bigr)c
 +\int_0^{\lambda_1}d\Lambda\Bigl[\Bigl(\int^{\Lambda_1}d\Gamma_1[a_{\Gamma_1}b]\Bigr)_\Lambda c\Bigr]\\
 &\quad -\int^{\Lambda_1}d\Gamma_1[a_{\Gamma_1}bc]
 -\int^{\Lambda_1}d\Gamma_1\Bigl[\Big(\int_0^{\lambda_1+T}d\Lambda a\Bigr)_{\Gamma_1}[b_{\Lambda}c]\Bigr]\\
 &\quad -(-1)^{(p(a)+\ol{N})p(b)}\Bigl(\int_{\lambda_1}^{\lambda_1+T}d\lambda_1 b\Bigr)[a_{\Lambda_1}c],
\end{align*}
thus we get
\begin{align*}
 \Res_{\Lambda_1}(\lambda_1^{-1}F_1(a\otimes b\otimes c))
 =(ab)c-a(bc)-\Bigl(\int_0^Td\Lambda a\Bigr)[b_\Lambda c]
  -(-1)^{p(a)p(b)}\Bigl(\int_0^Td\Lambda b\Bigr)[a_\Lambda c]. 
\end{align*}
Hence the equation \eqref{eq:W:ResFn} for $n=1$ follows from the quasi-associativity.
\end{proof}

To show \eqref{eq:W:ResFn} for $n \ge 2$, we expand $[\cdot_\Lambda \cdot]$ in the form
\begin{align*}
 [a_\Lambda b]=\sum_{m \in \bbN, \, I \subset [N]} \Lambda^{m|I} a_{m|I}b \quad (a, b\in V). 
\end{align*}
We also denote $a_{m|N}b \ceq a_{m|[N]}b$. Using induction on $n$, we get for any $n\in\bbN$ that 
\begin{align*}
\begin{split}
  F_{n+2}(a \otimes b \otimes c)
&=\sum_{m\ge n}n!\binom{m}{n}\lambda_1^{m-n}
  \Bigl(\int^{\Lambda_1}d\Gamma_1[a_{\Gamma_1}b]\Bigr)_{m|N}c
 +\sum_{k=0}^{n-1}k!\binom{m}{k}\binom{n}{k}
  \lambda_1^{m-k}\pdd_{\lambda_1}^{n-k-1}[a_{\Lambda_1}b]_{m|N}c\\
 &\quad 
 -\sum_{m\ge n}n!\binom{m}{n}\int^{\Lambda_1}d\Gamma_1[(\lambda_1+T)^{m-n}a_{\Gamma_1}(b_{m|N}c)]
 -\sum_{k=0}^{n-1}k!\pdd_{\lambda_1}^{n-k-1}[a_{\Lambda_1}(b_{k|N}c)] \\
 &\quad 
 -(-1)^{(p(a)+\ol{N})p(b)}\sum_{m \ge n, \, I\subset [N]}
  (-1)^{p(b)\ol{\# I}}n!\binom{m}{n}((\lambda_1+T)^{m-n}\theta_1^Ib)(a_{m|I}c).
\end{split}
\end{align*}

\begin{lem}\label{lem:W:ResFn}
For $n \ge 2$ and $a,b,c \in V$, we have
\begin{align*}
 \Res_{\Lambda_1}\Res_{\Lambda_2}\bigl(\lambda_1^{-1}\lambda_2^{-1} \pdd_{\lambda_1}^n(X \square X)
  (a \otimes b \otimes c \otimes z_{1,2}^{-1}z_{1,3}^{-1}z_{2,3}^{-1})\bigr) = 0.
\end{align*}
\end{lem}

\begin{proof}
Let us use the notation $F_n$ in \eqref{eq:W:Fn}. First, we have 
\begin{align*}
&\Res_{\Lambda_1}(\lambda_1^{-1}F_{n+2}(a\otimes b\otimes c))\\
&=n!\Res_{\Lambda}\bigl(\lambda^{-n-1}
 ([ab_{\Lambda}c]
 -(-1)^{p(a)\ol{N}}(e^{\nabla\cdot\pdd_{\Lambda}}a)[b_\Lambda c]
 -(-1)^{(p(a)+\ol{N})p(b)}(e^{\nabla\cdot\pdd_{\Lambda}}b)[a_\Lambda c])\bigr) \\
&\quad 
-\sum_{k=0}^{n-1}k!(n-k-1)!\Bigl(a_{n-k-1|N}(b_{k|N}c)-\binom{n}{k}(a_{n-k-1|N}b)_{k|N}c\Bigr).
\end{align*}
Since the Jacobi identity yields
\begin{align*}
 [a_{\Lambda_1-\Lambda_2}[b_{\Lambda_2}c]] = 
 (-1)^{(p(a)+\ol{N})\ol{N}} [[a_{\Lambda_1-\Lambda_2}b]_{\Lambda_1}c] + 
 (-1)^{(p(a)+\ol{N})(p(b)+\ol{N})}[b_{\Lambda_2}[a_{\Lambda_1-\Lambda_2}c]], 
\end{align*}
we obtain
\begin{align*}
& (-1)^{(p(a)+\ol{N})p(b)} \Res_{\Lambda_1} \Bigl(\lambda_1^{-n-1}
   \int_0^{\lambda_1}d\Lambda_2[b_{\Lambda_2}[a_{\Lambda_1-\Lambda_2}c]\Bigr) \\
&=(-1)^{(p(a)+\ol{N})\ol{N}} \Res_{\Lambda_1} \Bigl(\lambda_1^{-n-1}
   \int_0^{\lambda_1}d\Lambda_2[a_{\Lambda_1-\Lambda_2}[b_{\Lambda_2}c]]\Bigr)
  -\Res_{\Lambda_1} \Bigl(\lambda_1^{-n-1}
   \int_0^{\lambda_1}d\Lambda_2[[a_{\Lambda_1-\Lambda_2}b]_{\Lambda_1}c]\Bigr) \\
&=\sum_{k=0}^{n-1} \frac{k!(n-k-1)!}{n!} 
  \Bigl(a_{n-k-1|N}(b_{k|N}c)-\binom{n}{k}(a_{n-k-1|N}b)_{k|N}c\Bigr).
\end{align*}
Thus 
\begin{align*}
   \Res_{\Lambda_1}\bigl(\lambda_1^{-1}F_{n+2}(a\otimes b\otimes c)\bigr) 
=n!\Res_{\Lambda}\Bigl(\lambda^{-n-1}
 \bigl([ab_{\Lambda}c]
&-(-1)^{p(a)\ol{N}}(e^{\nabla\cdot\pdd_{\Lambda}}a)[b_\Lambda c] \\
 -(-1)^{(p(a)+\ol{N})p(b)}(e^{\nabla\cdot\pdd_{\Lambda}}b)[a_\Lambda c] 
&-(-1)^{(p(a)+\ol{N})p(b)}\int_0^{\lambda}d\Gamma[b_{\Gamma}[a_{\Lambda-\Gamma}c]\bigr)\Bigr).
\end{align*}
Hence the equation \eqref{eq:W:ResFn} for $n\ge 2$ follows from the right Wick formula. 
\end{proof}

As was mentioned around \eqref{eq:W:ResFn}, 
$X \square X=0$ follows from \cref{lem:ResF01} and \cref{lem:W:ResFn}. 

The proof of \cref{thm:W:VA} is now complete. 

\subsection{Relation to $N_W=N$ SUSY chiral algebras}\label{ss:W:CA}

In \cite[Appendix A]{BDHK}, the authors discussed the relation between their operadic theory of vertex algebras and the theory of chiral algebras \cite{BD}, which is an algebro-geometric reformulation of vertex algebras. Here we give an $N_W=N$ SUSY analogue of the argument in \cite[Appendix A]{BDHK}. A SUSY analogue of a chiral algebra was introduced by Heluani in \cite[\S4]{H}, which will be used in our discussion.

We continue to work over a field $\bbK$ of characteristic $0$, and take a positive integer $N$. We will also extensively use the theory of superschemes. See \cite{DM}, \cite[\S2.2]{H} and \cite[\S1]{KV} for the detail. An ordinary scheme is called an even scheme for distinction.

\subsubsection{$N_W=N$ SUSY chiral algebras}\label{sss:W:CA}

Let $X$ be an irreducible smooth projective $(1|N)$-dimensional supercurve over $\bbK$ in the sense of \cite[p.1085]{KV}. More explicitly, it is a pair $(X,\shO_X)$ of a topological space $X$ and a sheaf $\shO_X$ of commutative $\bbK$-superalgebras over $X$ such that
\begin{clist}
\item $(X,\shO_X^{\tred})$ is an irreducible smooth projective curve, where $\shO_X^{\tred}$ is the reduced part of $\shO_X$. The sheaf $\shO_X$ is called the structure sheaf of $X$.
\item
For every $x \in X(\bbK)$, there is an open subsets $U \subset X$ and some linearly independent odd elements $\theta^i \in \shO_X(U)$ ($i \in [N]$) such that $\shO_X(U)=\shO_X^{\tred}(U) \otimes \bbK[\theta^1,\dotsc,\theta^N]$. Such $U$ is called a coordinate neighborhood of $x$. 
\end{clist}
Hereafter we call such $X$ a smooth $(1|N)$-supercurve for simplicity.
On each coordinate neighborhood $U \subset X$ of $x$, we can take a local coordinate $z$ of the even curve  $(X,\shO_X^{\tred})$ and a tuple $Z=(z,\zeta^1,\dotsc,\zeta^N)$ satisfying the relation \eqref{eq:W:poly}. We call $Z$ a local coordinate of $U$. 

On a smooth $(1|N)$-supercurve $X$, we have the sheaf $\shD_X$ of differential operators, whose sections $\shD_X(U)$ on a coordinate neighborhood $U \subset X$ is the superalgebra 
\[
 \clD \ceq \bbK[Z][\pdd_Z] = 
 \bbK[z,\zeta^1,\dotsc,\zeta^n][\pdd_z,\pdd_{\zeta^1},\dotsc,\pdd_{\zeta^N}]
\]
of regular differential operators generated by $\pdd_z \ceq \frac{\pdd}{\pdd z}$ and $\pdd_{\zeta^k} \ceq \frac{\pdd}{\pdd \zeta^k}$ ($i \in [N]$). We denote by $\crMod \shD_{X}$ the category of quasi-coherent sheaves of right $\shD_{X}$-modules on $X$.

For each $n \in \bbZ_{>0}$, we denote by $\Delta\colon X \inj X^n$ the embedding of the big diagonal in $X^n$, i.e., the union of the hypersurfaces 
\begin{align}\label{eq:W:bigdiag}
 Z_i = Z_j 
\end{align}
for $i \ne j \in [n]$, using the local coordinates $Z_i=(z_i,\zeta_i^1,\dotsc,\zeta_i^N)$. Hereafter we denote $\Delta \ceq \Delta(X) \subset X^n$ for simplicity. As commented in the beginning of \cite[\S4]{H}, the big diagonal has relative codimension $1|N$, so that it is not a divisor (of codimension $1|0$) in genuine sense, and there is an ambiguity of the pushforward functor $\Delta_*$. We follow \cite[\S4.1.4]{H} to resolve this problem: Let us consider the embedding 
\[
 j\colon X^n \bs \Delta \linj X^n,
\]
where $X^n \bs \Delta$ is the superscheme with underlying topological space $U \ceq \abs{X}^n \bs \abs{\Delta}$ and structure sheaf $\rst{\shO_{X^n}}{U}$. Then we define the pushforward $\Delta_* \shA$ of a right $\shD_X$-module $\shA$ to be $\Delta_* \shA \ceq j_* j^*(\shO_{X^{n-1}} \boxtimes \shA)/(\shO_{X^{n-1}} \boxtimes \shA)$.

Now, for a right $\shD_X$-module $\shA$, we set 
\[
 \oPchW{\shA}(n) \ceq 
 \Hom_{\crMod \shD_{X^n}}\bigl(j_* j^* \shA^{\boxtimes n},\Delta_*\shA\bigr).
\]
By the argument in the even case \cite{BD}, $\oPchW{\shA}=\bigl(\oPchW{\shA}(n)\bigr)_{n \in \bbN}$ has a structure of a superoperad. The image of the generator of $\oLie(2)$ under $\varphi$ gives a binary operation $\mu \in \oPchW{\shA}(2)$, which is a map $\mu_{\shA}\colon j_* j^* \shA \boxtimes \shA \to \Delta_* \shA$ satisfying skew-symmetry and Jacobi identity. We call $\mu_{\shA}$ the chiral operation of $\shA$, following the terminology in \cite[\S3.1]{BD}.

A \emph{non-unital $N_W=N$ SUSY chiral algebra on $X$} is a pair $(\shA,\varphi)$ of a right $\shD_X$-module $\shA$ and a superoperad morphism $\varphi\colon \oLie \to \oPchW{\shA}$. In the language of \cref{dfn:1:Lie}, such a morphism $\varphi$ is nothing but a Lie algebra structure on the superoperad $\oPchW{\shA}$.

Let us denote by $\omega_X$ the Berezinian bundle \cite[2.2.7]{H} of the smooth $1|X$-dimensional curve $X$, which can be regarded as a SUSY analogue of the canonical bundle of an even curve. It is a right $\shD_X$-module in the standard way, and moreover carries a non-unital $N_W=N$ SUSY chiral algebra structure, established by Heluani in \cite[4.1.7]{H}. We denote by $\mu_{\omega} \in \oPchW{\omega_X}(2)$ the corresponding chiral operation. 

Now we define an \emph{$N_W=N$ SUSY chiral algebra $\shA$ on $X$} to be a non-unital one equipped with a morphism $\iota\colon \omega_X \to \shA$ of right $\shD_X$-modules such that the composition $\mu_{\iota \otimes \id_{\shA}}$ (the restriction of $\mu_{\shA}$ to $j_* j^* \omega_X \boxtimes \shA \to j_* j^* \shA \boxtimes \shA$) coincides with $j_* j^* \omega_X \boxtimes \shA \to (j_* j^* \omega_X \boxtimes \shA)/\omega_X \boxtimes \shA \sto \Delta_*\shA$ (the latter in the even case is called the unit operation \cite[3.1.12, 3.3.3]{BD}).

\subsubsection{SUSY chiral algebras on the superline}

We write down $N_W=N$ SUSY chiral algebras over the affine superline $\bbA^{1|N}$, following the argument for the even case \cite[\S A.2]{BDHK}.

Let $Z=(z,\zeta^1,\dotsc,\zeta^N)$ be a $(1|N)_W$-supervariable and $\bbK[Z]$ be the corresponding polynomial superring in the sense of \cref{ss:W:poly}. In the setting of the previous \cref{sss:W:CA}, consider the case where the smooth $(1|N)$-supercurve $X$ is taken to be
\[
 X = \bbA^{1|N} \ceq \Spec(\bbK[Z]),
\] 
i.e., the affine superspace of dimension $1|N$ in the sense of \cite[Example 1.1.6]{KV}. A right $\shD_X$-module $\shA$ is determined by the right module $A \ceq \Gamma(X,\shA)$ over the superalgebra $\clD=\Gamma(X,\shD_{X})$ of regular differential operators on the supervariable $Z$. 

Let $n \in \bbZ_{>0}$, and $Z_k=(z_k,\zeta_k^1,\dotsc,\zeta_k^N)$ be a $(1|N)_W$-supervariable for $k \in [n]$, regarded as the local coordinate of the $k$-th component of the $n$-th product supervariety $X^n = (\bbA^{1|N})^n$. Then we denote the superalgebra $\Gamma(X^n,\shO_{X^n})$ of regular functions on $X^n$ by
\[
 \clO_n \ceq \bbK[Z_1,\dotsc,Z_n]. 
\]
Also recall the superalgebra $\clO_n^{\star}=\bbK[Z_k]_{k=1}^n[z_{k,l}^{-1}]_{1 \le k<l \le n}$ in \cref{dfn:W:O}. Then, by the argument in \cite[\S A.2]{BDHK}, we have
\begin{align}\label{eq:W:A2}
 \Gamma(X,j_* j^* \shA^{\boxtimes n}) \cong \clO_n^{\star} \otimes_{\clO_n} A^{\otimes n},
\end{align}
which is a right module over the superalgebra 
\[
 \clD_n \ceq \bbK[Z_1,\dotsc,Z_n][\pdd_{Z_1},\dotsc,\pdd_{Z_n}]
\]
of regular differential operators on $X^n$. The module structure is given by the action on the right factor in \eqref{eq:W:A2}.

Let $I \subset \clD_n$ be the left ideal generated by $z_1-z_k$ for $k \in \{2,\dotsc,n\}$. Then the quotient $I \bs \clD_n$ is a $\clD_1$-$\clD_n$-bimodule, where the $\clD_n$-action is the right multiplication, and the left action of $\clD_1 = \bbK[Z][\pdd_Z]$-action is given by letting $z,\zeta^i$ act from left by $z_1,\zeta_1^i$ and $\pdd_z,\pdd_{\zeta^i}$ act as multiplication on the left by $\sum_{k=1}^n \pdd_{z_k}, \sum_{k=1}^n \pdd_{\zeta^i_k}$. Then, by the argument in \cite[\S A.2]{BDHK}, we have
\begin{align}\label{eq:W:A3}
 \Gamma(X^n,\Delta_*\shA) \cong A \otimes_{\clD_1} (I \bs \clD_n),
\end{align}
where the right $\clD_n$-module structure is given by right multiplication on the right factor. Combining \eqref{eq:W:A2} and \eqref{eq:W:A3}, we obtain an $\frS_n$-supermodule isomorphism 
\begin{align}\label{eq:W:A4}
 \oPchW{\shA}(n) \cong \Hom_{\crMod \clD_n}\bigl(
  \clO^{\star}_n \otimes_{\clO_n} A^{\otimes n}, A \otimes_{\clD_1}(I \bs \clD_n)\bigr),
\end{align}
where $\crMod \clD_n$ denotes the category of right $\clD_n$-supermodules.

\subsubsection{Equivariant SUSY chiral algebras on the superline}

Next, we restrict the above description to the translation-equivariant part. We use the terminology on equivariant $\shD$-modules in \cite[\S A.3]{BDHK}, which we extend to superschemes in an obvious way. In particular, we will use the notion of \emph{weakly equivariant $\clD$-modules}.

The affine superline $X=\bbA^{1|N}$ has a natural action of the group superscheme $\Aff^{1|N}$ of affine transformations. Consider the even subgroup scheme $\rmT \ceq \bbG_a \times \bbG_a^N \subset \Aff^{1|N}$ consisting of translations. We have an equivalence of categories between $\rmT$-equivariant quasi-coherent $\shO_X$-modules and linear superspaces by taking the stalk at the origin $0 \in X$. The inverse functor sends a linear superspace $V$ to the sheaf associated to the $\bbK[Z]$-supermodule $V[Z]=\bbK[Z] \otimes_\bbK V$ equipped with the action of $(u,\upsilon^1,\dotsc,\upsilon^N) \in \rmT$ by $v(z,\zeta^1,\dotsc,\zeta^N) \mto v(z+u,\zeta^1+\upsilon^1,\dotsc,\zeta^N+\upsilon^N)$.  

As for $\rmT$-equivariant $\clD$-modules on $X=\bbA^{1|N}$, following the even case \cite[\S A.4]{BDHK}, we have an equivalence of categories 
\begin{align}\label{eq:W:TDmod=Hmod}
\begin{split}
 \bigl(\text{weakly $\rmT$-equivariant quasi-coherent $\shD_X$-modules $\shA$ on $X$}) \\
 \lsto \bigl(\text{$\clH_W$-supermodules $(V,\nabla) = (V,T,S^1,\dotsc,S^N)$}\bigr).
\end{split}
\end{align}
For the latter, see \cref{dfn:W:clHW}. Indeed, given a $\shD_X$-module $\shA$, taking the stalk of $\shA$ at the origin $0 \in X$, we obtain $V$ as in the last paragraph, together with the endomorphisms $\nabla=(T,S^1,\dotsc,S^N)$ on $V$ induced by the differential operators $\pdd_Z=(\pdd_z,\pdd_{\zeta^1},\dotsc,\pdd_{\zeta^N})$ which satisfy the relations \eqref{eq:W:TSi}. Conversely, given a pair $(V,\nabla)$, we have a $\clD_1$-module structure on $V[Z]=\bbK[z,\zeta^1,\dotsc,\zeta^i] \otimes_{\bbK} V$ by letting $Z$ act by multiplication by $Z$, and $\pdd_z$ (resp.\ $\pdd_{\zeta^i}$) act by $T-\frac{\pdd}{\pdd z}$ (resp.\ $S^i-\frac{\pdd}{\pdd \zeta^i}$).

Now, following the argument in \cite[\S A.6]{BDHK}, let us consider a weakly $\rmT$-equivariant $\clD$-module $\shA$ on $X=\bbA^{1|N}$ corresponding to the $\clH_W$-supermodule $(V,\nabla)$. Then the associated $\clD_n$-module \eqref{eq:W:A3} is given by 
\[
 V \otimes_{\bbK[\nabla]}\bbK[Z][\nabla_1,\dotsc,\nabla_n]
\]
with $Z=(z,\zeta^1,\dotsc,\zeta^N)$ and $\nabla_k=(T_k,S_k^1,\dotsc,S_k^N)$ satisfying the relations \eqref{eq:W:TSi} for each $k \in [n]$. The commutative superalgebra $\bbK[Z][\nabla_1,\dotsc,\nabla_n]$ is regarded as a $\bbK[\pdd]$-$\clD_n$-bimodule in the following way:
\begin{itemize}
\item 
The left action of $\nabla=(T,S^1,\dotsc,S^N)$ is given by 
$(\sum_{k=1}^n T_k,\sum_{k=1}^n S_k^1,\dotsc,\sum_{k=1}^n S_k^N)$.
\item
The right action of $\pdd_{Z_k}=(\pdd_{z_k},\pdd_{\zeta_k^1},\dotsc,\pdd_{\zeta_k^N})$ is the right multiplication by $\nabla_k$, and the right action of $Z_k=(z_k,\zeta_k^1,\dotsc,\zeta_k^N)$ on $f=f(Z,\nabla_1,\dotsc,\nabla_n) \in \bbK[Z][\nabla_1,\dotsc,\nabla_n]$ is given by 
\begin{align*}
 f \cdot z_k \ceq z f+\frac{\pdd}{\pdd T_k}f, \quad 
 f \cdot \zeta_k^i \ceq (-1)^{p(f)}\bigl(\zeta^i f+\frac{\pdd}{\pdd S_k^i}f\bigr),
\end{align*}
where $p(f)$ denotes the parity of $f$. 
\end{itemize}
Hence, $\oPchW{\shA}(n)$ in \eqref{eq:W:A4} is equal to 
\begin{align}\label{eq:W:A8}
 \oPchW{\shA}(n) = \Hom_{\crMod \clD_n}\bigl(\clO^{\star}_n \otimes V^{\otimes n}, 
       V \otimes_{\bbK[\nabla]}\bbK[Z][\nabla_1,\dotsc,\nabla_n]\bigr).
\end{align}

The translation group $\rmT = \bbG_a \times \bbG_a^N$ acts on $\oPchW{\shA}(n)$ by letting $U=(u,\upsilon^1,\dotsc,\upsilon^N) \in \rmT$ act on $\varphi \in \oPchW{\shA}(n)$ to give $\varphi^U \in \oPchW{\shA}(n)$ with 
\begin{align}\label{eq:W:A10}
 \varphi^U \bigl(f(Z_0,\dotsc,Z_n) \otimes v_1 \otimes \dotsb \otimes v_n) \ceq 
 \rst{\varphi(f(Z_1-U,\dotsc,Z_n-U) \otimes  v_1 \otimes \dotsb \otimes v_n)}{Z=Z+U}
\end{align}
for $f=f(Z_1,\dotsc,Z_n) \in \clO^{\star}_n$ and $v_k \in V$, $k \in [n]$. Then the superoperad $\oPchW{\shA}$ is $\rmT$-equivariant in the sense of \cite[\S A.5]{BDHK}, and the $\rmT$-invariants $\oPchWT{\shA}(n) \subset \oPchW{\shA}(n)$ form a sub-operad 
\[
 \oPchWT{\shA} \subset \oPchW{\shA}.
\]

Now, we have the main statement of this \cref{ss:W:CA}. 

\begin{prp}
Let $V=(V,\nabla)$ be an $\clH_W$-supermodule (\cref{dfn:W:clHW}), and $\shA$ be the corresponding weakly $\rmT$-equivariant $\clD$-module on $\bbA^{1|N}$ under the equivalence \eqref{eq:W:TDmod=Hmod}. Let $\oPchW{V}$ be the $N_W=N$ SUSY chiral operad in \cref{dfn:W:Pch}. Then there is an isomorphism of operads
\[
 \oPchWT{\shA} \cong \oPchW{V}.
\]
\end{prp}

\begin{proof}
The even case argument \cite[Lemma A.1]{BDHK} works, so we will just explain the outline.

By \eqref{eq:W:A8}, we want to determine the $\rmT$-invariant parts of $\clO^{\star}_n \otimes V^{\otimes n}$ and $V \otimes_{\bbK[\nabla]}\bbK[Z][\nabla_1,\dotsc,\nabla_n]$. As for the former, by \eqref{eq:W:Ker-delta}, we have  an isomorphism
\[
 (\clO^{\star}_n \otimes V^{\otimes n})^{\rmT} = 
 (\clO^{\star}_n)^{\rmT} \otimes V^{\otimes n} = \clO^{\star \rmT}_n \otimes V^{\otimes n}
\]
as $\clD_n^{\rmT}$-supermodules, where 
\[
 \clD_n^{\rmT} \ceq \Ker(\ad \delta^0) \cap \Ker(\ad \delta^1) \cap \dotsb \cap 
                    \Ker(\ad \delta^N) \bigr) \subset \clD_n
\]
with $\delta^i$ given after \eqref{eq:W:Ker-delta}. 
As for the latter $V \otimes_{\bbK[\nabla]}\bbK[Z][\nabla_1,\dotsc,\nabla_n]$, by the argument of \cref{lem:W:VnLn=VLn-1}, we have a $\clD_n^{\rmT}$-supermodule isomorphism
\begin{align*}
 (V \otimes_{\bbK[\nabla]}\bbK[Z][\nabla_1,\dotsc,\nabla_n])^{\rmT}  = 
 V \otimes_{\bbK[\nabla]} \bbK[\nabla_1,\dotsc,\nabla_n] &\lsto 
 V[\Lambda_k]_{k=1}^n/\Img(\nabla+\Lambda_1+\dotsb+\Lambda_n), \\
 v \otimes f(\nabla_1,\dotsc,\nabla_n) &\lmto f(-\Lambda_1,\dotsc,-\Lambda_n)v.
\end{align*}

Now, any $n$-operation $\varphi \in \oPchWT{\shA}(n)$ on $f \otimes v_1 \otimes \dotsb \otimes v_n \in \clO_n^{\star} \otimes V^{\otimes n}$ is independent of $Z$ by \eqref{eq:W:A10}, so that it defines an element of $V[\Lambda_k]_{k=1}^n/\Img(\nabla+\Lambda_1+\dotsb+\Lambda_n)$. Thus we obtain an element of $\oPchW{V}(n)$ from $\varphi$. An inverse construction is given by the Taylor-expansion technique in the proof of \cite[Lemma A.1]{BDHK}.
\end{proof}

\section{\texorpdfstring{$N_K=N$}{NK=N} SUSY chiral operad}\label{s:K}

In this section, we introduce $N_K=N$ SUSY analogues of the operads $\oCh{}$ and $\oPch{}$ in \cite{BDHK}. Most  of the results and proofs in \cref{s:W} for the $N_W=N$ case apply to the $N_K=N$ case with minor modifications, so we omit some details. We fix a positive integer $N$ throughout this section.

\subsection{Polynomial superalgebra}\label{ss:K:poly}

Here we give a summary of the polynomial ring of supervariables in the $N_K=N$ setting. 

\begin{dfn}
Let $A$ be a set and $\Lambda_\alpha=(\lambda_\alpha,\theta_\alpha^1,\ldots,\theta_\alpha^N)$ be a sequence of letters for each $\alpha\in A$. We denote by $\bbK[\Lambda_\alpha]_{\alpha\in A}$ the $\bbK$-superalgebra generated by even generators $\lambda_\alpha$ ($\alpha \in A$) and odd generators $\theta_\alpha^i$ ($\alpha \in A$, $i \in [N]$) with relations
\begin{align*}
 \lambda_\alpha  \lambda_\beta -\lambda_\beta  \lambda_\alpha  = 0, \quad 
 \lambda_\alpha  \theta_\beta^i-\theta_\beta^i \lambda_\alpha  = 0, \quad 
 \theta_\alpha^i \theta_\beta^j+\theta_\beta^j \theta_\alpha^i =
 -2\delta_{\alpha,\beta}\delta_{i,j} \lambda_\alpha \quad (\alpha,\beta \in A, \, i,j \in [N]).
\end{align*}
Each $\Lambda_\alpha$ for $\alpha \in A$ is called a \emph{$(1|N)_K$-supervariable}, and the $\bbK$-superalgebra $\bbK[\Lambda_\alpha]_{\alpha \in A}$ is called the \emph{$N_K=N$ polynomial superalgebra of the supervariables $(\Lambda_\alpha)_{\alpha \in A}$}. 
\end{dfn}

As in the $N_W=N$ case (see the lines after \cref{dfn:W:poly}), we will also use the simplified symbols $K[\Lambda_k]_{k=1}^n$ and $K[\Lambda]$. Note that the $N_K=N$ polynomial superalgebra is not commutative, unlike the $N_W=N$ polynomial superalgebra.

For a linear superspace $V$ and a $(1|N)_K$-supervariable $\Lambda_\alpha$ ($\alpha\in A$) , we denote
\begin{align*}
 V[\Lambda_\alpha]_{\alpha \in A} \ceq \bbK[\Lambda_\alpha]_{\alpha \in A} \otimes_\bbK V, 
\end{align*}
which is naturally a left $\bbK[\Lambda_\alpha]_{\alpha\in A}$-supermodule. 

\begin{dfn}
Let $\clH_K$ be the $\bbK$-superalgebra generated by an even generator $T$ and odd generators $S^i$ $(i\in[N])$ with relations
\begin{align*}
 TS^i-S^iT = 0, \quad S^iS^j+S^jS^i = 2\delta_{i,j}T \quad (i,j \in [N]). 
\end{align*}
For simplicity, we denote $\nabla\ceq(T, S^1, \ldots, S^N)$. 
A linear superspace $V$ equipped with a left $\clH_K$-supermodule structure is denoted as $(V,\nabla)=(V,T,S^1,\ldots,S^N)$. 
\end{dfn}

Note that $\clH_K$ is isomorphic to $\bbK[\Lambda]$ as a superalgebra by the homomorphism defined by 
\begin{align}\label{eq:HKKLam}
 T \lmto -\lambda, \quad S^i \lmto -\theta^i \quad (i\in[N]). 
\end{align}

\subsection{$N_K=N$ SUSY Lie conformal operad}\label{ss:K:LCA}

In this subsection, we introduce an $N_K=N$ SUSY analogue of the operad $\oCh{}$ in \cite{BDHK}. All the results and proofs for the $N_W=N$ case in \cref{ss:W:LCA} are also valid for the $N_K=N$ case with the modifications: 
\begin{itemize}
 \item Replace the superalgebra $\clH_W$ with $\clH_K$. 
 \item Replace all $(1|N)_W$-supervariables with $(1|N)_K$-supervariables. 
\end{itemize}
So we only refer to the results without proofs. 

\medskip

Let us fix a $(1|N)_K$-supervariable $\Lambda=(\lambda, \theta^1, \ldots, \theta^N)$. 

\begin{dfn}[{\cite[Definition 4.10]{HK}}]\label{dfn:K:LCA}
Let $V$ be a left $\clH_K$-supermodule and $[\cdot_\Lambda\cdot]\colon V\otimes V\to V[\Lambda]$ be a linear map of parity $\ol{N}$. 
A triple $(V, \nabla, [\cdot_\Lambda\cdot])$ is called an \emph{$N_K=N$ SUSY Lie conformal algebra} if it satisfies the following conditions: 
\begin{clist}
 \item (\emph{sesquilinearity}) 
 For any $a, b\in V$, 
 	\begin{align}\label{eq:Ksesq}
	[S^ia_\Lambda b]=-(-1)^N\theta^i[a_\Lambda b], \quad 
	[a_\Lambda S^ib]=(-1)^{p(a)}(\theta^i+S^i)[a_\Lambda b]\quad (i\in[N]). 
 	\end{align}
 
 \item (\emph{skew-symmetry})
  For any $a, b\in V$, 
 	\begin{align}\label{eq:Kssym}
 	[b_\Lambda a]=-(-1)^{p(a)p(b)+\ol{N}}[a_{-\Lambda-\nabla}b]. 
 	\end{align}

 \item (\emph{Jacobi identity}) 
 For any $a, b, c\in V$, 
	 \begin{align}\label{eq:KJac}
	 [a_{\Lambda_1}[b_{\Lambda_2}c]]
	 =(-1)^{p(a)\ol{N}+\ol{N}}[[a_{\Lambda_1}b]_{\Lambda_1+\Lambda_2}c]
	 +(-1)^{(p(a)+\ol{N})(p(b)+\ol{N})}[b_{\Lambda_2}[a_{\Lambda_1}c]],
 	\end{align}
 where $\Lambda_1, \Lambda_2$ are $(1|N)_K$-supervariables. 
\end{clist}

The linear map $[\cdot_\Lambda\cdot]$ is called the \emph{$\Lambda$-bracket} of the left $\clH_K$-supermodule $V$. For simplicity, we say $V$ is an $N_K=N$ SUSY Lie conformal algebra. 
\end{dfn}

\begin{rmk}
In \eqref{eq:Ksesq} and \eqref{eq:Kssym},  the symbols $T$ and $S^i$ on the left of the $(1|N)_K$-supervariable $\Lambda$ are identified with the element of $\bbK[\Lambda]$ by the isomorphism \eqref{eq:HKKLam}. Also, \eqref{eq:KJac} is the identity in $V[\Lambda_k]_{k=1, 2}$. 
\end{rmk}

In the remaining of this section, let $(V, \nabla)$ be a left $\clH_K$-supermodule and $\Lambda_k=(\lambda_k, \theta_k^1, \ldots, \theta_k^N)$ be a $(1|N)_K$-supervariable for each $k\in\bbZ_{>0}$. 

The polynomial superalgebra $\bbK[\Lambda_k]_{k=1}^n$ carries a structure of an $(\clH_K^{\otimes n}, \clH_K)$-superbimodule as follows: 
First, since any element of $\clH_K$ is expressed uniquely as $\varphi(\nabla)$ for some $\varphi(\Lambda) \in \bbK[\Lambda]$, the space $\bbK[\Lambda_k]_{k=1}^n$ has a left $\clH_K^{\otimes n}$-supermodule structure by letting $\varphi_1(\nabla)\otimes \cdots \otimes \varphi_n(\nabla)\in\clH_W^{\otimes n}$ act on $a(\Lambda_1, \ldots, \Lambda_n) \in \bbK[\Lambda_k]_{k=1}^n$ as 
\[
 \bigl(\varphi_1(\nabla)\otimes \cdots \otimes \varphi_n(\nabla)\bigr)a(\Lambda_1 \ldots, \Lambda_n)
 \ceq \varphi_1(-\Lambda_1)\cdots \varphi_n(-\Lambda_n)a(\Lambda_1, \ldots, \Lambda_n). 
\]
Second, the space $\bbK[\Lambda_k]_{k=1}^n$ is a right $\clH_K$-supermodule by letting 
\begin{align*}
 a(\Lambda_1, \ldots, \Lambda_n)\cdot T \ceq 
 a(\Lambda_1, \ldots, \Lambda_n)\Bigl(-\sum_{k=1}^n\lambda_k\Bigr), \quad
 a(\Lambda_1, \ldots, \Lambda_n)\cdot S^i\ceq 
 a(\Lambda_1, \ldots, \Lambda_n)\Bigl(-\sum_{k=1}^n\theta_k^i\Bigr)
\end{align*}
for $a(\Lambda_1, \ldots, \Lambda_n) \in \bbK[\Lambda_k]_{k=1}^n$. 
These left and right supermodule structures are consistent, and we have the desired superbimodule structure.

Thus, for an $\clH_K$-supermodule $V=(V,\nabla)$ and $n\in\bbZ_{>0}$, we obtain a left $\clH_K^{\otimes n}$-supermodule
\begin{align}\label{eq:LCA:HKn-mod}
 V_{\nabla}[\Lambda_k]_{k=1}^n \ceq \bbK[\Lambda_k]_{k=1}^n \otimes _{\clH_K} V. 
\end{align}
For $n=0$, we denote
\begin{align*}
V_\nabla[\Lambda_k]_{k=1}^n \ceq V/\nabla V, 
\end{align*}
where $\nabla V\ceq TV+\sum_{i=1}^NS^iV\subset V$. 

As in the $N_W=N$ case (\cref{lem:W:VnLn=VLn-1}), we have:

\begin{lem}\label{lem:K:VnLn=VLn-1}
The linear superspace $V_\nabla[\Lambda_k]_{k=1}^n$ is isomorphic to $V[\Lambda_k]_{k=1}^{n-1}$ by the linear map 
\begin{align*}
 a(\Lambda_1, \cdots, \Lambda_n) \otimes v \lmto 
 a(\Lambda_1, \cdots, \Lambda_{n-1}, -\Lambda_1-\cdots-\Lambda_{n-1}-\nabla)v
 \quad (a \in \bbK[\Lambda]_{k=1}^n, \, v \in V).
\end{align*}
\end{lem}

Since the tensor product $V^{\otimes n}$ is naturally a left $\clH_K^{\otimes n}$-supermodule, the following definition makes sense:

\begin{dfn}
For a left $\clH_K$-supermodule $V=(V,\nabla)$ and $n \in \bbN$, we denote 
\begin{align*}
 \oChK{V}(n) \ceq \Hom_{\clH_K^{\otimes n}}(V^{\otimes n}, V_\nabla[\Lambda_k]_{k=1}^n).
\end{align*}
In other words, $\oChK{V}(n)$ is the linear sub-superspace of $\Hom_{\bbK}(V^{\otimes n}, V_\nabla[\Lambda_k]_{k=1}^n)$ spanned by elements $X$ such that 
\begin{align*}
 X(\varphi v) = (-1)^{p(\varphi)p(X)}\varphi X(v) \quad 
 (\varphi \in \clH_K^{\otimes n}, \ v \in V^{\otimes n}).
\end{align*}
To stress the variable $\Lambda_1,\ldots,\Lambda_n$, we express an element $X \in \oChK{V}(n)$ as $X_{\Lambda_1,\ldots,\Lambda_n}$. 
\end{dfn}

As in \cref{ss:W:LCA} for the $N_W=N$ case, define the action of $\frS_n$ on $\oChK{V}$ and the composition map, then we have a superoperad $\oChK{V} \ceq \bigl(\oChK{V}(n)\bigr)_{n\in\bbN}$. 

\begin{dfn}\label{dfn:K:Chom}
The superoperad $\oChK{V}$ is called the \emph{$N_K=N$ SUSY Lie conformal operad} of the left $\clH_K$-supermodule $V=(V,\nabla)$. 
\end{dfn}

The following is the main claim of this \cref{ss:K:LCA}, which is a natural $N_K=N$ SUSY analogue of \cite[Proposition 5.1]{BDHK}. This can be proved in the same way as in the $N_W=N$ case (\cref{thm:W:LCA}) with minor modifications. 

\begin{thm}\label{thm:K:LCA}
Let $V=(V,\nabla)$ be a left $\clH_K$-supermodule. 
\begin{enumerate}
\item 
For an odd Maurer-Cartan solution $X \in \MC\bigl(L\bigl(\oChK{\Pi^{N+1}V}\bigr)\bigr)_{\od}$, define a linear map $[\cdot_\Lambda\cdot]_X\colon V \otimes V \to V[\Lambda]$ by 
 \begin{align*}
 [a_\Lambda b] \ceq (-1)^{p(a)(\ol{N}+\od)} X_{\Lambda,-\Lambda-\nabla}(a \otimes b) \quad (a,b \in V).
 \end{align*}
 Then $(V, \nabla, [\cdot_\Lambda\cdot]_X)$ is an $N_K=N$ SUSY Lie conformal algebra. 
 
 \item 
 The map $X \mto [\cdot_\Lambda\cdot]_X$ gives a bijection
 \begin{align*}
 \MC\bigl(L(\oChK{\Pi^{N+1}V})\bigr)_{\od} \lsto 
 \{\text{$N_K=N$ Lie conformal algebra structures on $(V,\nabla)$}\}.
 \end{align*}
\end{enumerate}
\end{thm}

\subsection{$N_K=N$ SUSY chiral operad}\label{ss:K:VA}

In this subsection, we introduce an $N_K=N$ SUSY analogue of the operad $P^{\ch}$ in \cite{BDHK}. 
All the results and proofs in \cref{ss:W:VA} for the $N_W=N$ case are also valid for the $N_K=N$ case with the modifications: 
\begin{itemize}
\item 
Replace the superalgebra $\clH_W$ with $\clH_K$. 

\item 
Replace $(1|N)_W$-supervariables $\Lambda$, $\Lambda_k$ $(k\in\bbZ_{>0})$ and $\Gamma$ with $(1|N)_K$-supervariables (while do not replace $(1|N)_W$-supervariables $Z_k$).

\item 
For $(1|N)_W$-supervariables $Z_k=(z_k, \zeta^1_k, \ldots, \zeta^N_k)$ with $k\in\bbZ_{>0}$, replace $z_{k,l}=z_k-z_l$ and $\pdd_{\zeta_k^i}$ by  
\begin{align}\label{eq:K:Z-W}
 z_{k,l} \ceq z_k-z_l-\sum_{i=1}^N \zeta^i_k \zeta^i_l, \quad 
 D_{\zeta^i_k} \ceq \pdd_{\zeta^i_k}+\zeta^i_k\pdd_{z_k}, 
\end{align}
respectively. 
\end{itemize}

\medskip

Let us fix a $(1|N)_K$-supervariable $\Lambda=(\lambda,\theta^1,\ldots,\theta^N)$. 
For even linear transformations $F,G$ on linear superspace $V$, we can define a linear map $\int_F^G d\Lambda\colon V[\Lambda]\to V$ of parity $\ol{N}$ as in \cref{ss:W:VA}. 
Also, if $V$ is a superalgebra (not necessarily unital nor associative), we can define a linear map $\int_F^G d\Lambda\, a$ ($a \in V$) of parity $p(a)$. Using this integral, we introduce: 

\begin{dfn}[{\cite[Definition 4.19]{HK}}]\label{dfn:NKVA}
Let $(V,\nabla,[\cdot_\Lambda\cdot])$ be an $N_K=N$ SUSY Lie conformal algebra and $\mu\colon V \otimes V \to V$ be an even linear map. We denote $ab \ceq \mu(a \otimes b)$ for $a,b \in V$. 
A tuple $(V,\nabla,[\cdot_\Lambda\cdot],\mu)$ is called a \emph{non-unital $N_K=N$ SUSY vertex algebra} if it satisfies the following conditions: 
\begin{clist}
\item For any $a,b \in V$, 
\begin{align}\label{eq:KVA:der}
  S^i(ab) = (S^ia)b+(-1)^{p(a)}a(S^ib) \quad (i \in [N]). 
\end{align}

\item (\emph{quasi-commutativity}) For any $a,b \in V$, 
\begin{align}\label{eq:KVA:qcom}
 ab-(-1)^{p(a)p(b)}ba = \int_{-T}^0 d\Lambda [a_\Lambda b]. 
\end{align}

\item (\emph{quasi-associativity}) For any $a,b,c \in V$, 
\begin{align}\label{eq:KVA:qass}
  (ab)c-a(bc) = \Bigl(\int_0^Td\Lambda a\Bigr)[b_\Lambda c]
+(-1)^{p(a)p(b)}\Bigl(\int_0^Td\Lambda b\Bigr)[a_\Lambda c]. 
\end{align}

\item (\emph{Wick formula}) For any $a,b,c \in V$, 
\begin{align}\label{eq:KVA:Wick}
 [a_\Lambda bc] = [a_\Lambda b]c + (-1)^{(p(a)+\ol{N})p(b)} b[a_\Lambda c] 
                  + \int_0^\lambda d\Gamma [[a_\Lambda b]_\Gamma c], 
\end{align}
where $\Gamma$ is an additional $(1|N)_K$-supervariable. 
\end{clist}

For simplicity, we say $(V,\nabla)$, or more simply $V$, is a non-unital $N_K=N$ SUSY vertex algebra. The linear map $\mu$ is called the \emph{multiplication} of $V$. 
\end{dfn}

\begin{dfn}
A non-unital $N_K=N$ SUSY vertex algebra $V$ is called an \emph{$N_K=N$ SUSY vertex algebra} if there exists an even element $\vac\in V$ such that $a\vac=\vac a=a$ for all $a\in V$. 
\end{dfn}

\begin{eg}[{\cite[Example 5.8]{HK}}]
Let $B_N$ be the $N_K=N$ SUSY vertex algebra generated by $\Psi$ of parity $\ol{N}$ whose OPE is 
\begin{align*}
 [\Psi_\Lambda\Psi]=
 \begin{cases}
  \Lambda^{1|N}\vac & (\ol{N}=\ol{0}), \\
  \Lambda^{0|N}\vac & (\ol{N}=\ol{1}). 
 \end{cases}
\end{align*}
When $N=1$, expand the corresponding superfield
\begin{align*}
 \Psi(Z) = \varphi(z)+\theta\alpha(z), 
\end{align*}
then we have
\begin{align*}
 [\varphi_\lambda \varphi]=\vac, \quad [\alpha_\lambda \alpha]=\lambda. 
\end{align*}
These are the OPEs of well-known boson-fermion system. 
\end{eg}

In the remaining of this \cref{ss:K:VA}, let $(V, \nabla)$ be a left $\clH_K$-supermodule and $\Lambda_k=(\lambda_k, \theta_k^1, \ldots, \theta_k^N)$ be a $(1|N)_K$-supervariable for each $k\in\bbZ_{>0}$. 

\cref{lem:rWick}, \cref{lem:lsym} and their proofs are valid for the $N_K=N$ case by replacing all $(1|N)_W$-supervariables with $(1|N)_K$-supervariables. Also, \cref{lem:intbra} and its proof are valid by replacing supervariables, and hence we have the integral of $\Lambda$-bracket $\int^\Lambda d\Gamma\colon V\otimes V\to V[\Lambda]$ for the $(1|N)_K$-supervariable $\Lambda$. 
Using this integral, one can similarly prove \cref{prp:W:skecom} and \cref{prp:W:Jqas} with $(1|N)_W$-supervariables replaced by $(1|N)_K$-supervariables. 

\begin{dfn}
Let $n\in\bbZ_{>0}$ and $Z_k=(z_k, \zeta_k^1, \ldots, \zeta_k^N)$ be a $(1|N)_W$-supervariable for each $k\in[n]$. We set
\begin{align*}
z_{k, l}\ceq z_k-z_l-\sum_{i=1}^N\zeta_k^i\zeta_l^i, \quad
\zeta_{k, l}^i\ceq \zeta_k^i-\zeta_l^i
\end{align*}
for $k, l\in[n]$ and $i\in[N]$. Also, we set $Z_{k, l}\ceq (z_{k, l}, \zeta_{k, l}^1, \ldots, \zeta_{k, l}^N)$ for simplicity. 
\begin{enumerate}
\item 
We denote by 
\begin{align*}
 \clO_n^{\star}\ceq \bbK[Z_k]_{k=1}^n[z_{k, l}^{-1}]_{1\le k<l\le n}
\end{align*}
the localization of $\bbK[Z_k]_{k=1}^n$ by the multiplicatively closed set 
generated by $\{z_{k, l}\mid 1\le k<l\le n\}$. 

\item 
Let $\clO_n^{\star\rmT}$ denote the sub-superalgebra of $\clO_n^{\star}$ generated by the subset $\{z_{k, l}^{\pm 1}\mid 1\le k<l\le n\}\cup \{\zeta^i_{k, l}\mid i\in[N], 1\le k< l\le n\}$. Thus, we have
\begin{align*}
 \clO_n^{\star\rmT}=\bbK[z_{k, l}^{\pm1}, \zeta_{k, l}^i\mid i\in[N], 1\le k<l\le n]. 
\end{align*}


\item 
Let $\clD_n$ denote the superalgebra of regular differential operators of $Z_k$, i.e, the sub-superalgebra of $\End_\bbK \clO^{\star}_n$ generated by $Z_k=(z_k, \zeta_k^1, \ldots, \zeta_k^N)$ and $D_{Z_k}=(\pdd_{z_k}, D_{\zeta_k^1}, \ldots, D_{\zeta_k^N})$ for $k\in[n]$. As a superspace, we have
\begin{align*}
 \clD_n = \bbK[Z_k]_{k=1}^n[D_{Z_k}]_{k=1}^n = 
 \bbK[z_k,\zeta_k^1,\ldots,\zeta_k^N]_{k=1}^n[\pdd_{z_k}, D_{\zeta_k^1},\ldots,D_{\zeta_k^N}]_{k=1}^n.
\end{align*}

\item 
Let $\clD_n^\rmT$ denote the sub-superalgebra of $\clD_n$ generated by the subset 
$\{z_{k,l}, \zeta_{k,l}^i \mid i \in [N], 1\le k<l\le n\} \cup 
 \{\pdd_{z_k}, D_{\zeta_k^i} \mid i \in [N], k\in[n]\}$. Thus we have
\begin{align*}
 \clD_n^\rmT=\bbK[Z_{k, l}]_{1\le k<l\le n}[D_{Z_k}]_{k=1}^n. 
\end{align*}
\end{enumerate}
By convention, we denote $\clO_0^{\star}=\clO_0^{\star\rmT}=\clD_0=\clD_0^{\rmT} \ceq \bbK$. 
\end{dfn}



The superspace $V^{\otimes n}\otimes \clO_n^{\star\rmT}$ carries a structure of a right $\clD_n^{\rmT}$-supermodule by letting $Z_k=(z_k, \zeta_k^1, \ldots, \zeta_k^N)$ act as 
\begin{align*}
(v\otimes f)\cdot z_k\ceq v\otimes fz_k, \quad 
(v\otimes f)\cdot \zeta_k^i \ceq v\otimes f\zeta_k^i,
\end{align*}
and $D_{Z_k}=\bigl(\pdd_{z_k}, D_{\zeta_k^1}, \ldots, D_{\zeta_k^N}\bigr)$ act as
\begin{align*} 
 &(v\otimes f) \cdot \pdd_{z_k}\ceq T^{(k)} v \otimes f - v \otimes \pdd_{z_k}f, \\
 &(v\otimes f) \cdot D_{\zeta_k^i} \ceq 
 (-1)^{p(v)+p(f)} (S^i)^{(k)} v \otimes f - (-1)^{p(f)} v \otimes D_{\zeta_k^i}f
\end{align*}
for $v \in V^{\otimes n}$, $f \in \clO_n^{\star\rmT}$. Here, for a linear transformation $\varphi \in \End V$, the symbol $\varphi^{(k)}$ denotes the linear transformation on $V^{\otimes n}$ defined by $\varphi^{(k)} \ceq \id_V \otimes \dotsb \otimes \overset{k}{\varphi} \otimes \dotsb \otimes \id_V$.

Now let us recall the superspace $V_\nabla[\Lambda_k]_{k=1}^n$ in \eqref{eq:LCA:HKn-mod}. It also has a structure of right $\clD_n^{\rmT}$-supermodule by letting 
\begin{align*}
&(a(\Lambda_1, \ldots, \Lambda_n)\otimes v) \cdot z_{k, l}
 \ceq (\pdd_{\lambda_l}-\pdd_{\lambda_k})a(\Lambda_1, \ldots, \Lambda_n)\otimes v, \\
&(a(\Lambda_1, \ldots, \Lambda_n)\otimes v) \cdot \zeta_{k, l}^i 
 \ceq (-1)^{p(a)+p(v)}(D_{\theta_l^i}-D_{\theta_k^i})a(\Lambda_1, \ldots, \Lambda_n)v, 
\end{align*}
and
\begin{align*}
&\bigr(a(\Lambda_1, \ldots, \Lambda_n)\otimes v\bigr)\cdot\pdd_{z_k}
 \ceq -\lambda_ka(\Lambda_1, \ldots, \Lambda_n)\otimes v, \\
&\bigl(a(\Lambda_1, \ldots, \Lambda_n)\otimes v\bigr)\cdot D_{\zeta_k^i}
 \ceq -(-1)^{p(a)+p(v)}\theta_k^i a(\Lambda_1, \ldots, \Lambda_n)\otimes v. 
\end{align*}
for $a(\Lambda_1, \ldots, \Lambda_n) \in \bbK[\Lambda_k]_{k=1}^n$ and $v \in V$. 

\begin{dfn}
For a left $\clH_K$-supermodule $V=(V,\nabla)$ and $n \in \bbN$, we denote
\begin{align*}
 \oPchK{V}(n) \ceq \Hom_{\clD_n^\rmT}
 \bigl(V^{\otimes n} \otimes \clO_n^{\star \rmT},V_{\nabla}[\Lambda_k]_{k=1}^n\bigr). 
\end{align*}
In other words, $\oPchK{V}(n)$ is a linear sub-superspace of $\Hom_\bbK\bigl(V^{\otimes n} \otimes \clO_n^{\star \rmT},V_{\nabla}[\Lambda_k]_{k=1}^n\bigr)$ spanned by elements $X$ such that 
\begin{align*}
 X((v \otimes f) \cdot \varphi) = X(v \otimes f) \cdot \varphi \quad 
 (v \in V^{\otimes n}, \, f \in \clO_n^{\star\rmT}, \, \varphi \in \clD_n^{\rmT}). 
\end{align*}
To stress the variables $\Lambda_1,\ldots,\Lambda_n$, we express an element $X \in \oPchK{V}(n)$ as $X_{\Lambda_1,\ldots,\Lambda_n}$. 
\end{dfn}

As in \cref{ss:W:VA} for the $N_W=N$ case, define the action of $\frS_n$ and the composition map, then we have a superoperad $\oPchK{V} \ceq \bigl(\oPchK{V}(n)\bigr)_{n\in\bbN}$. 

\begin{dfn}\label{dfn:K:Pch}
For a left $\clH_K$-supermodule $V=(V,\nabla)$, we call the superoperad $\oPchK{V}$ the \emph{$N_K=N$ SUSY chiral operad}. 
\end{dfn}

The following is the main statement of this \cref{ss:K:VA}, which is a natural $N_K=N$ SUSY analogue of \cite[Theorem 6.12]{BDHK}. We can prove similarly to the $N_W=N$ case \cref{thm:W:VA} with minor modification. 

\begin{thm}\label{thm:K:VA}
Let $(V, \nabla)=(V, T, S^1, \ldots, S^N)$ be a left $\clH_K$-supermodule. 
\begin{enumerate}
\item 
For an odd Maurer-Cartan solution $X\in \MC\bigl(L(\oPchK{\Pi^{N+1}V})\bigr)_{\od}$, define linear maps $[\cdot_\Lambda\cdot]_X\colon V \otimes V \to V[\Lambda]$ and $\mu_X\colon V \otimes V \to V$ by 
\begin{align}
 \label{eq:K:VAlbra}
 [a_\Lambda b]_X
&\ceq (-1)^{p(a)(\ol{N}+\od)}X_{\Lambda, -\Lambda-\nabla}(a \otimes b \otimes 1_{\bbK}), \\
 \label{eq:K:VAprod}
 \mu_X(a\otimes b)
&\ceq (-1)^{p(a)(\ol{N}+\od)+\od} \Res_\Lambda \bigl(
  \lambda^{-1}X_{\Lambda,-\Lambda-\nabla}(a \otimes b \otimes z_{1,2}^{-1})\bigr)
\end{align}
for $a, b\in V$. Then $(V, \nabla, [\cdot_\Lambda\cdot]_X, \mu_X)$ is a non-unital $N_K=N$ SUSY vertex algebra. 

\item 
The map $X \mto ([\cdot_\Lambda\cdot]_X, \mu_X)$ gives a bijection
\begin{align*}
 \MC\bigl(L(\oPchK{\Pi^{N+1}V})\bigr)_{\od}
 \lsto \{\text{non-unital $N_K=N$ SUSY vertex algebra structures on $(V,\nabla)$}\}.
\end{align*}
\end{enumerate}
\end{thm}

\subsection{Relation to $N_K=N$ SUSY chiral algebras}\label{ss:K:CA}

We can relate our superoperad $\oPchK{}$ to the $N_K=N$ SUSY chiral algebras, as we did for the $N_W=N$ case in \cref{ss:W:CA}. We just comment the necessary modification, and suppress the full presentation.

Instead of a $(1|N)$-supercurve, we consider an oriented superconformal $(1|N)$-supercurve $X$ \cite[2.2.11]{H}. It is defined to be a $(1|N)$-supercurve equipped with the differential form $\omega$ which is locally give by 
\[
 \omega = d z + \sum_{k=1}^N \zeta^i d \zeta^i
\]
with respect to some local coordinate $Z_i=(z,\zeta^1,\dotsc,\zeta^N)$, and is well-defined up to multiplication by a function. Then, recalling \eqref{eq:K:Z-W}, we replace the big diagonal $\Delta\colon X \subset X^n$ in \eqref{eq:W:bigdiag} by the \emph{super-diagonal} $\Delta^s\colon X \subset X^n$. Using local coordinates $Z_i=(z_i,\zeta_i^1,\dotsc,\zeta_i^N)$, it is the union of the divisors (sub-superscheme of $1|0$-codimension) in $X^n$ defined by 
\[
 0 = \zeta_i-\zeta_j-\sum_{k=1}^N \zeta_i^k \zeta_j^k \quad (i \ne j \in [n]).
\]
Then we have the standard pushforward functor $\Delta^s_*$ of $\clD$-modules. Together with the open embedding $j\colon U \inj X^n$, we define 
\[
 \oPchK{\shA}(n) \ceq \Hom_{\crMod \shD_{X^n}}\bigl(j_* j^* \shA^{\boxtimes n},\Delta^s_*\shA\bigr)
\]
for a right $\shD_X$-module $\shA$. It gives rise to a superoperad $\oPchK{\shA} = \bigl(\oPchK{\shA}(n)\bigr)_{n \in \bbN}$.

The translation group $\rmT = \bbG_a \times \bbG_a^N$ acts on $\oPchK{\shA}(n)$ by the formula \eqref{eq:W:A10} with $Z_i-U$ understood as \eqref{eq:K:Z-W}. Then the superoperad $\oPchK{\shA}$ is $\rmT$-equivariant, and the $\rmT$-invariants $\oPchKT{\shA}(n) \subset \oPchK{\shA}(n)$ form a sub-operad 
\[
 \oPchKT{\shA} \subset \oPchK{\shA}.
\]
The remaining part goes similarly as the $N_W=N$ case in \cref{ss:W:CA}, and we have:

\begin{prp}
Let $V=(V,\nabla)$ be an $\clH_K$-supermodule, and $\shA$ be the corresponding weakly $\rmT$-equivariant $\clD$-module on $\bbA^{1|N}$. 
Let $\oPchK{V}$ be the $N_K=N$ SUSY chiral operad. Then there is an isomorphism of superoperads
\[
 \oPchKT{\shA} \cong \oPchK{V}.
\]
\end{prp} 

\section{Cohomology of SUSY vertex algebras}\label{s:coh}

\subsection{Recollection on the Lie algebra cohomology}\label{ss:coh:CE}

Here we give a brief recollection on the Chevalley-Eilenberg cohomology of a Lie algebra from the viewpoint of the operad theory. For the detail, we refer to \cite[Chap.\ 12, \S 13.2]{LV}. 

We use the language of operads reviewed in \cref{s:op}, although in this subsection we work in the non-super setting. Let $V$ be a linear space. Recall the convolution Lie algebra 
\[
 \frg_{\oLie,V} = \bigl(\Hom_{\frS}(\oLie^{\ash},\oHom_V),[\cdot,\cdot]\bigr)
\]
in \eqref{eq:1:gPV} and the set of solutions of the Maurer-Cartan equation
\[
 \MC(\frg_{\oLie,V}) = \{X \in \frg_{\oLie,V}^1 \mid \tfrac{1}{2}[X,X]=0\},
\]
both being specialized for the Lie operad $\oP = \oLie$. A Maurer-Cartan solution $X \in \MC(\frg_{\oLie,V})$ corresponds bijectively to a Lie algebra structure $[\cdot,\cdot]_X$ on $V$ . 

Given a Maurer-Cartan solution $X \in \MC(\frg_{\oLie,V})$, we have a differential 
\[
 d_X \ceq [X,-]\colon \frg_{\oLie,V}^{\bl} \lto \frg_{\oLie,V}^{\bl+1}, \quad d_X^2=0.
\]
The obtained cochain complex $(\frg_{\oLie,V}^{\bl},d_X)$ coincides with the Chevalley-Eilenberg cochain complex $C^\bl(L,L)$ of the Lie algebra $L=(V,[\cdot,\cdot]_X)$ with coefficients in itself as an adjoint module, up to a degree shift. It is known that the cochain complex $(\frg_{\oLie,V}^{\bl},d_X)$ has a Lie bracket, giving a dg Lie algebra structure, and that the Lie bracket is essentially equal to the Nijenhuis-Richardson bracket on the Chevalley-Eilenberg cochain complex.

We can extend the construction of the cochain complex $(\frg_{\oLie,V}^{\bl},d_X)$ to the version with coefficients in a module over the Lie algebra $(V,[\cdot,\cdot]_X)$ along the line of the general argument for a quadratic operad in \cite[\S12.4]{LV}. Let us give a brief recollection.
\begin{itemize}
\item 
First, for an operad $\oP$ and a linear space $V$, we have an isomorphism
\begin{align}\label{eq:coh:PV}
 \Hom_{\frS}(\oP,\oHom_V) \cong C_{\oP}(V,V) \ceq \Hom(\oP(V),V)
\end{align}
by the hom-tensor adjunction and \eqref{eq:1:HomV}, where 
\[
 \oP(V) \ceq \bigoplus_{n \ge 0}\oP(V)(n) \otimes_{\frS_n} V^{\otimes n}
\]
is the image of the Schur functor of the operad $\oP$ \cite[\S 5.1.2]{LV}. Thus, a $\oP$-algebra structure $\varphi\colon \oP \to \oHom_V$ can be encoded by $C_{\oP}(V,V)$. The linear space $C_{\oP}(V,V)$ has an $\bbN$-grading with $C_{\oP}^n(V,V) \ceq \Hom(\oP(V)(n) \otimes_{\frS_n} V^{\otimes n},V)$.

\item
Next, recall from \cite[\S 12.3.1]{LV} that a module $M$ over a $\oP$-algebra $(V,\varphi)$ is defined to be a linear space equipped with two morphisms $\gamma_M\colon \oP \circ (V;M) \to M$ and $\eta_M\colon M \to \oP \circ (V;M)$ satisfying some axioms. Here we used the composite product \cite[\S6.1.1]{LV}
\[
 \oP \circ (V;M) \ceq \bigoplus_{n \ge 0}\oP(n) \otimes_{\frS_n} \Bigl(
 \bigoplus_{1 \le i \le n} V^{\otimes i} \otimes M \otimes V^{\otimes n-i}\Bigr).
\]
As it implies, the notion of modules over a $\oP$-algebra is a natural operadic analogue of bimodules over the algebra structure considered.

\item
Following \cite[\S 12.4.1]{LV}, let $\oP$ be a homogeneous quadratic operad, $V$ be a linear space equipped with a $\oP$-algebra structure, and $M$ be a module over the $\oP$-algebra $V$. Let us consider the linear space 
\begin{align}\label{eq:coh:CVM}
 C_{\oP}(V,M) \ceq \Hom(\oP^{\ash}(V),M)
\end{align}
which has the natural $\bbN$-grading induced by that on $\oP^{\ash}(V)$. In the case $M=V$, this graded space coincides with $C_{\oP}^{\bl}(V,V)$ in \eqref{eq:coh:PV}. According to \cite[\S 12.4.1]{LV}, the graded space $C_{\oP}^{\bl}(V,M)$ has a natural differential $d$ coming from the $\oP$-algebra structure on $V$ and the $V$-module structure on $M$, and we obtain a cochain complex $(C_{\oP}^{\bl}(V,M),d)$ called the \emph{operadic cochain complex}. 

\item
The cohomology of the operadic cochain complex $(C_{\oP}^{\bl}(V,M),d)$ is denoted by $H^{\bl}_{\oP}(A,M)$. The low degree cohomology have analogous interpretation as the classical cohomology. For example, we have $H^0_{\oP}(A,M) \cong \Der_A(A,M)$, the space of derivations. We refer to \cite[\S 12.4]{LV} for the detail.

\item
In the case $\oP=\oLie$, the operadic cochain complex $(C_{\oP}(V,M),d)$ coincides, up to a degree shift, with the Chevalley-Eilenberg cochain complex of the Lie algebra $V$ with coefficients in the $V$-module $M$. Similarly, we can recover the Harrison cochain complex in the case $\oP=\oCom$, and 
recover the Hochschild cochain complex in the case $\oP=\oAss$.
\end{itemize}

In \cite[\S7]{BDHK}, the authors introduced a natural analogue of the Chevalley-Eilenberg complex for vertex algebras, based on their operad $\oPch{V}$. As explained in \cref{s:op}, in order to have a vertex algebra analogue in the operad theory, we should keep $\oLie$ as it is, and replace the endomorphism operad $\oHom_V$  by the chiral operad $\oPch{V}$. The construction in loc.\ cit.\ can be understood as such a replacement in \eqref{eq:coh:PV} and \eqref{eq:coh:CVM}.
We will not give the detail of the construction in \cite[\S7]{BDHK}. Instead, we explain a straightforward modification for SUSY vertex algebras in the following \cref{ss:coh:W} and \cref{ss:coh:K}.

\subsection{Cohomology of $N_W=N$ SUSY vertex algebras}\label{ss:coh:W}

In this subsection, we introduce the cochain complex of an $N_W=N$ SUSY vertex algebra with coefficients in its module, and investigate the low degree cohomology. First, we introduce modules over an $N_W=N$ SUSY vertex algebra.

\begin{dfn}\label{dfn:W:LCAmod}
Let $(V, \nabla, [\cdot_\Lambda\cdot])$ be an $N_W=N$ SUSY Lie conformal algebra, $(M,\nabla)$ be an $\clH_W$-supermodule and $\rho_\Lambda\colon V \otimes M \to M[\Lambda]$ be a linear map of parity $\ol{N}$. We denote 
\begin{align}\label{eq:coh:aLx}
 a_\Lambda x \ceq \rho_\Lambda(a \otimes x)
\end{align}
for $a \in V$ and $x \in M$. A triple $(M,\nabla,\rho_\Lambda)$ is called a \emph{module over $(V,\nabla,[\cdot_\Lambda\cdot])$} if the following conditions are satisfied. 
\begin{clist}
\item For any $a \in V$ and $x \in M$, 
\begin{align}\label{eq:Lmodsesq}
\begin{aligned}
&  (Ta)_\Lambda     x  = - \lambda a_\Lambda x, & 
&     a_\Lambda   (Tx) =  (\lambda+T)a_\Lambda x,\\
&(S^ia)_\Lambda     x  = -(-1)^N \theta^i a_\Lambda x, & 
&     a_\Lambda (S^ix) =  (-1)^{p(a)+\ol{N}}(\theta^i+S^i)a_\Lambda x \quad (\forall i \in [N]).
\end{aligned} 
\end{align}

\item For any $a, b\in V$ and $x\in M$, 
\begin{align}\label{eq:LmodJ}
a_{\Lambda_1}(b_{\Lambda_2}x)
=(-1)^{(p(a)+\ol{N})\ol{N}}[a_{\Lambda_1}b]_{\Lambda_1+\Lambda_2}x
+(-1)^{(p(a)+\ol{N})(p(b)+\ol{N})}b_{\Lambda_2}(a_{\Lambda_1}x),
\end{align}
where $\Lambda_1, \Lambda_2$ are $(1|N)_W$-supervariables. 
\end{clist}

For simplicity, we say $(M, \nabla)$, or more simply $M$, is a module over $V$. The map $\rho_\Lambda$ is called \emph{the} (\emph{left}) \emph{$\Lambda$-action on $M$}. 
\end{dfn}

\begin{dfn}\label{dfn:W:VAmod}
Let $(V, \nabla, [\cdot_\Lambda\cdot],\mu)$ be a non-unital $N_W=N$ SUSY vertex algebra, $(M,\nabla,\rho_\Lambda)$ be a module over $N_W=N$ Lie conformal algebra $(V,\nabla,[\cdot_\Lambda \cdot])$ and $\rho\colon V \otimes M \to M$ be an even linear map. We denote 
\begin{align}\label{eq:coh:adx}
 a \cdot x \ceq \rho(a \otimes x)
\end{align}
for $a \in V$ and $x \in M$. A tuple $(M,\nabla,\rho_\Lambda,\rho)$ is called a \emph{module over $(V,\nabla,[\cdot_\Lambda\cdot],\mu)$} if it satisfies the following conditions: 
\begin{clist}

\item 
For any $a \in V$ and $x \in M$,
\begin{align}\label{eq:VAmodder}
   T(a \cdot x) =   (Ta) \cdot x + a \cdot (Tx), \quad
 S^i(a \cdot x) = (S^ia) \cdot x + (-1)^{p(a)}a \cdot (S^ix) \quad (i \in [N]). 
\end{align}

\item 
For any $a, b\in V$ and $x\in M$, 
\begin{align*}
 (ab) \cdot x
=a\cdot(b\cdot x)+ \Bigl(\int_0^T d\Lambda a\Bigr)\cdot b_\Lambda x
+(-1)^{p(a)p(b)}\Bigl(\int_0^T d\Lambda b\Bigr)\cdot a_\Lambda x. 
\end{align*}

\item For any $a, b\in V$ and $x\in M$, 
\begin{align*}
 a_\Lambda(b \cdot x)
=[a_\Lambda b] \cdot x + (-1)^{(p(a)+\ol{N})p(b)}b \cdot (a_\Lambda x) + 
 \int_0^\lambda d\Gamma[a_\Lambda b]_\Gamma x. 
\end{align*}
\end{clist}

For simplicity, we say $(M, \nabla)$, or more simply $M$, is a module over $V$. 
\end{dfn}

\begin{rmk}
For a module $M$ over a non-unital $N_W=N$ SUSY vertex algebra $V$, we define a right action of $V$ by
\begin{align}\label{eq:ract}
 x_\Lambda a \ceq -(-1)^{p(a)p(x)+\ol{N}} a_{-\Lambda-\nabla}x, \quad 
 x \cdot   a \ceq  (-1)^{p(a)p(x)}        a \cdot x \quad (a \in V, \, x \in M).
\end{align}
Then $M$ is a ``right'' module over $V$. 
\end{rmk}

In the remaining of this \cref{ss:coh:W}, we fix a non-unital $N_W=N$ SUSY vertex algebra $(V, \nabla, [\cdot_\Lambda\cdot], \mu)$, and the word `$V$-module' always means a module over non-unital SUSY vertex algebra $(V, \nabla, [\cdot_\Lambda\cdot], \mu)$, not a module over $N_W=N$ SUSY Lie conformal algebra $(V, \nabla, [\cdot_\Lambda\cdot])$. 

Let $M$ and $N$ be $V$-modules. An even homomorphism $\varphi\colon M\to N$ of $\clH_W$-supermodule is called a $V$-module morphism if it satisfies
\begin{align*}
\varphi(a_\Lambda x)=a_\Lambda \varphi(x), \quad
\varphi(a\cdot x)=a\cdot \varphi(x)
\end{align*}
for any $a\in V$ and $x\in M$. 

Now, we introduce the indefinite integral of the $\Lambda$-action.

\begin{lem}\label{lem:intact}
Let $M$ be an $\clH_W$-supermodule. For a linear map $\rho_\Lambda\colon V\otimes M\to M[\Lambda]$ of parity $\ol{N}$ satisfying \eqref{eq:Lmodsesq} and an even linear map $\rho\colon V\otimes M\to M$, there exists a unique linear map $F\colon V\otimes M\to M[\Lambda]$ of parity $\ol{N}$ such that 
\begin{align*}
&\Res(\lambda^{-1}F(S^ia \otimes x)) = 
 (-1)^{N+1}\Res_\Lambda(\lambda^{-1}\theta^iF(a \otimes x)) \quad (i \in [N]), \\
&\pdd_\lambda F(a \otimes x) = a_\Lambda x, \quad 
 \Res(\lambda^{-1}F(a \otimes x)) = a \cdot x
\end{align*}
for every $a \in V$ and $x \in M$. Here we used the notations \eqref{eq:coh:aLx} and \eqref{eq:coh:adx}.
\end{lem}

\begin{proof}
Define a linear map $F\colon V \otimes M \to M[\Lambda]$ by
\begin{align*}
 F(a \otimes x) \ceq 
 \sum_{I \subset [N]} (-1)^{\#I(N+1)} \sigma(I) \theta^{[N] \bs I}(S^Ia)\cdot x + 
 \int_0^\lambda d\lambda (a_\Lambda x). 
\end{align*}
Then, similarly as in the proof of \cref{lem:intbra}, we can show that $F$ is the desired map. 
\end{proof}

\begin{dfn}\label{dfn:intact}
Let $M$ be an $\clH_W$-supermodule. Given a linear map $\rho_\Lambda\colon V \otimes M \to M[\Lambda]$ of parity $\ol{N}$ satisfying \eqref{eq:Lmodsesq} and an even linear map $\rho\colon V \otimes M \to M$, 
the linear map $F$ in \cref{lem:intact} is denoted by
\begin{align}\label{eq:coh:intact}
 \int^{\Lambda}d\Gamma\, (\cdot_\Gamma \cdot)\colon V \otimes M \lto M[\Lambda],
\end{align}
and called \emph{the integral of the} (\emph{left}) \emph{$\Lambda$-action $\rho_\Lambda$} (\emph{with respect to $\rho$}). 
\end{dfn}


For $V$-modules $M$ and $N$, it is easy to check that an even $\clH_W$-supermodule morphism $\varphi\colon M\to N$ is a $V$-module morphism if and only if it satisfies
\begin{align*}
 \varphi\Bigl(\int^\Lambda d\Gamma (a_\Lambda x)\Bigr) = \int^\Lambda d\Gamma (a_\Lambda \varphi(x))
\end{align*}
for any $a\in V$ and $x\in M$.

In the rest of this \cref{ss:coh:W}, let $M$ be a $V$-module over $V$ and $\Lambda_k=(\lambda_k, \theta_k^1, \ldots, \theta_k^N)$ be a $(1|N)_W$-supervariable for each $k \in \bbZ_{>0}$. 

 Let us introduce the underlying graded superspace of the cochain complex. 

\begin{dfn}
For $n \in \bbN$ and a module $M$ over a non-unital $N_W=N$ SUSY vertex algebra $V$, we define a linear superspace $C^n(V,M)$ by
\begin{align*}
 C^n(V, M) \ceq \Hom_{\clD_n^\rmT}
  \bigl(V^{\otimes n} \otimes \clO_n^{\star \rmT}, M_\nabla[\Lambda_k]_{k=1}^n\bigr)^{\frS_n}. 
\end{align*}
Note that we can naturally consider $C^n(V,M) \subset L^{n-1}\bigl(\oPchW{V \oplus M}\bigr)$ for $n \in \bbN$ (see \cref{dfn:1:LP} for the symbol $L^{n-1}$). 
\end{dfn}

Next, we want to define the differential. 
For that, let
\begin{align*}
 [\cdot_\Lambda\cdot]\colon (V \oplus M) \otimes (V \oplus M) \lto (V\oplus M)[\Lambda] \quad
 \mu\colon (V \oplus M) \otimes (V \oplus M) \lto V \oplus M
\end{align*}
be linear maps defined by
\begin{align*}
 [(a+x)_\Lambda (b+y)]  \ceq [a_\Lambda b]+a_\Lambda y+x_\Lambda b, \quad
 \mu\bigl((a+x)\otimes(b+y)\bigr) \ceq ab + a \cdot y + x \cdot b
\end{align*}
for $a,b \in V$ and $x,y \in M$, where we used the right action of $V$ defined by \eqref{eq:ract}. These yield a non-unital $N_W=N$ SUSY vertex algebra structure on the $\clH_W$-supermodule $V\oplus M$. Thus, by \cref{thm:W:VA}, we have the corresponding Maurer-Cartan solution $X \in \MC\bigl(L\bigl(\oPchW{\Pi^{N+1}(V\oplus M)}\bigr)\bigr)_{\od}$.
Since $X$ is an odd element of $ L^1\bigl(\oPchW{\Pi^{N+1}(V\oplus M)}\bigr)$, we have
\begin{align*}
 [X, Y] = X \square Y-(-1)^{\wt{p}(Y)}Y\square X
 \in L^n\bigl(\oPchW{\Pi^{N+1}(V \oplus M)}\bigr)
\end{align*}
for $Y \in C^n(\Pi^{N+1}V, \Pi^{N+1}M)$, where $\wt{p}$ denotes the parity in $\Pi^{N+1}(V\oplus M)$. 
A direct calculation shows:

\begin{lem}
For any $Y \in C^n(\Pi^{N+1}V, \Pi^{N+1}M)$, we have 
\begin{align*}
 \rst{[X,Y]}{(\Pi^{N+1} V)^{\otimes (n+1)} \otimes \clO_{n+1}^{\star \rmT}} 
 \in C^{n+1}(\Pi^{N+1}V, \Pi^{N+1}M).
\end{align*}
\end{lem}

Thus, for each $n \in \bbN$, one can define a linear map $\pdd^n\colon C^n(\Pi^{N+1}V, \Pi^{N+1}M) \to C^{n+1}(\Pi^{N+1}V,\Pi^{N+1}M)$ by
\begin{align*}
 \pdd^n Y \ceq \rst{[X,Y]}{(\Pi^{N+1} V)^{\otimes (n+1)} \otimes \clO_{n+1}^{\star \rmT}}
 \quad \bigl(Y \in C^n(\Pi^{N+1}V, \Pi^{N+1}M)\bigr).
\end{align*}
Since $X$ satisfies $X\square X=0$, we have $\pdd^{n+1}\circ \pdd^n=0$. 
Hence we get a cochain complex 
\begin{align}\label{eq:W:comp}
 C^\bullet(\Pi^{N+1}V, \Pi^{N+1}M)\ceq(C^n(\Pi^{N+1}V,\Pi^{N+1}M), \pdd^n)_{n\in\bbZ} 
\end{align}
with convention $C^n(\Pi^{N+1}V,\Pi^{N+1}M) \ceq 0$ and $\pdd^n \ceq 0$ for $n<0$. 

Now we have the main object in this \cref{ss:coh:W}.
It is an $N_W=N$ SUSY analogue of \cite[Definition 7.3]{BDHK}. 

\begin{dfn}
Let $M$ be a module over a non-unital $N_W=N$ SUSY vertex algebra $V$. 
We call $ C^\bullet(\Pi^{N+1}V, \Pi^{N+1}M)$ the \emph{cochain complex of $V$ with coefficients in M}. The $n$-th cohomology is denoted by
\begin{align*}
 H^n_{\ch}(\Pi^{N+1}V,\Pi^{N+1}M) \ceq H^n(C^\bullet(\Pi^{N+1}V,\Pi^{N+1}M)) = \Ker\pdd^{n}/\Img\pdd^{n-1}.
\end{align*}
\end{dfn}

As explained in \cref{ss:coh:CE}, the Chevalley-Eilenberg cochain complex is obtained from the pair $(\oLie, \oHom_V)$ of operads. Since the cochain complex \cref{eq:W:comp} is obtained from $(\oLie, \oPchW{V})$, it is a natural analogue of the Chevalley-Eilenberg cochain complex. Thus, one can understand that the cohomology $H^n_{\ch}(\Pi^NV, \Pi^NM)$ is a SUSY vertex algebra analogue of the Lie algebra cohomology. 

Hereafter we fix a non-unital $N_W=N$ SUSY vertex algebra $V$ and a $V$-module $M$. 
Let us identify the SUSY vertex algebra cohomology $H^n_{\ch}(\Pi^NV, \Pi^NM)$ for $n=0,1,2$. 
First, we describe the zeroth cohomology. For that, let us introduce: 

\begin{dfn}\label{dfn:Cas}
A \emph{Casimir element of $M$} is an element $\int x\in M/\nabla M$ such that $a_{-\nabla}x=0$ for all $a\in V$. Here we set $\nabla M \ceq \Img(T+S^1+\dotsb+S^N)\subset M$, and denote $\int\colon M \to M/\nabla M$ the canonical projection. We denote by $\Cas(V,M)$ the sub-superspace of $M/\nabla M$ consisting of all Casimir elements of $M$. 
\end{dfn}

\begin{prp}\label{prp:Hch0Cas}
The map
\begin{align}\label{eq:Hch0Cas}
 C^0(\Pi^{N+1}V, \Pi^{N+1}M) \lto M/\nabla M, \quad Y \lmto Y(1_\bbK)
\end{align}
gives the following isomorphism of linear superspaces:
\begin{align*}
 H^0_{\ch}(\Pi^{N+1}V, \Pi^{N+1}M) \cong \Cas(V, M). 
\end{align*}
\end{prp}

\begin{proof}
By definition, we have
\begin{align*}
 C^0(\Pi^NV, \Pi^NM) = \Hom_{\bbK}(\Pi^{N+1}\bbK, \Pi^{N+1}M/\nabla \Pi^{N+1}M). 
\end{align*}
It is clear that the map \cref{{eq:Hch0Cas}} is an isomorphism of linear superspaces. 
For $Y \in C^0(\Pi^{N+1}V, \Pi^{N+1}M)$, we can represent $Y(1_\bbK)=\int x$ with $x \in M$. Then we have
\begin{align*}
 (\pdd^0Y)(a \otimes 1_\bbK)
&=(X \circ_1 Y)(a \otimes 1_\bbK)
 =X_{0,-\nabla}(x \otimes a \otimes 1_\bbK)\\
&=(-1)^{\wt{p}(a)\wt{p}(x)}X_{-\Lambda, 0}(a \otimes x \otimes 1_\bbK) \\
&=(-1)^{\wt{p}(a)\wt{p}(x)}(-1)^{p(a)(\ol{N}+\od)}a_{-\nabla}x
\end{align*}
for any $a \in V$. 
Thus, the restriction of \cref{eq:Hch0Cas} to $H_{\ch}^0(V,M)$ gives the desired isomorphism. 
\end{proof}

Next, we turn to describe the first cohomology.

\begin{dfn}\label{dfn:der}
Let $M$ be a $V$-module. 
\begin{enumerate}
\item 
An $\clH_W$-supermodule homomorphism $D\colon V \to M$ is called a \emph{derivation on $V$} (with values in $M$) if it satisfies  
\begin{align*}
D[a_\Lambda b]
&=(-1)^{p(D)\ol{N}}Da_\Lambda b+(-1)^{p(D)(p(a)+\ol{N})}[a_\Lambda Db], \\
D(ab)
&=Da\cdot b+(-1)^{p(D)p(a)}a\cdot Db
\end{align*}
for any $a,b \in V$.
Here we used the integrals of the right $\Lambda$-action \eqref{eq:ract}. 
We denote by $\Der(V, M)$ the sub-superspace of $\Hom_{\clH_W}(V, M)$ consisting of all derivations of $V$ with values in $M$. 

\item
A derivation $D\colon V\to M$ is called an inner derivation if there exists $x\in M$ such that
\begin{align*}
 \rst{x_{\Lambda}a}{\Lambda=0}=-(-1)^{p(a)p(x)+\ol{N}}a_{-\nabla} x \quad (a\in V). 
\end{align*}
We denote by $\Ind(V,M)$ the linear sub-superspace of $\Der(V, M)$ consisting of all inner derivations. 
\end{enumerate}
\end{dfn}

It is clear that an $\clH_W$-supermodule morphism $D\colon V\to M$ is a derivation if and only if it satisfies 
\begin{align*}
 D\Bigl(\int^\Lambda d\Gamma[a_\Gamma b]\Bigr) = 
 (-1)^{p(D)\ol{N}}\int^\Lambda d\Gamma(Da_\Gamma b) + (-1)^{p(D)(p(a)+\ol{N})} \int^\Lambda d\Gamma (a_\Gamma Db)
\end{align*}
for any $a, b\in V$. 

\begin{prp}\label{prp:Hch1Der}
The map
\begin{align}\label{eq:Hch1Der}
C^1(\Pi^{N+1}V, \Pi^{N+1}M)\to \Hom_{\clH_W}(V, M), \quad 
Y\mapsto (-1)^{p(Y)(\ol{N}+\od)}Y_{-\nabla}(\cdot\otimes 1_\bbK)
\end{align}
gives the following isomorphism of linear superspace: 
\begin{align*}
 H^1_{\ch}(\Pi^{N+1}V, \Pi^{N+1}M)\cong \Der(V, M)/\Ind(V, M). 
\end{align*}
\end{prp}

\begin{proof}
By definition, we have
\begin{align*}
 C^1(\Pi^{N+1}V, \Pi^{N+1}M) = \Hom_{\clD_1^{\rmT}}
  \bigl(\Pi^{N+1}V\otimes\clO_1^{\star\rmT}, (\Pi^{N+1}M)_\nabla[\Lambda_1]\bigr). 
\end{align*}
We denote by $F_1$ the map \eqref{eq:Hch1Der}. 
It is clear that $F_1\colon C^1(\Pi^{N+1}V, \Pi^{N+1}M)\to \Hom_{\clH_W}(V, M)$ is an isomorphism of linear superspace. 

For $Y\in C^1(\Pi^{N+1}V, \Pi^{N+1}M)$ and $a, b\in V$, we have
\begin{align*}
&(\pdd^1Y)(a \otimes b \otimes z_{1,2}^{-1}) \\
&=(X\circ_1Y)(a \otimes b \otimes z_{1,2}^{-1})
 +(X\circ_1Y)^{(1,2)}(a \otimes b \otimes z_{1,2}^{-1})
 -(-1)^{\wt{p}(Y)}(Y\circ_1X)(a \otimes b \otimes z_{1,2}^{-1}) \\
&=X_{\Lambda_1, \Lambda_2}(Y_{-\nabla}(a \otimes 1_\bbK)\otimes z_{1,2}^{-1})
 +(-1)^{\wt{p}(a)\wt{p}(b)}X_{\Lambda_2,\Lambda_1}(
   Y_{-\nabla}(b \otimes 1_\bbK)\otimes a \otimes z_{2,1}^{-1}) \\
&\hspace{115pt}
 -(-1)^{\wt{p}(Y)} Y_{\Lambda_1+\Lambda_2}(
   X_{\Lambda_1,-\Lambda_1-\nabla}(a \otimes b \otimes z_{1,2}^{-1})\otimes 1_\bbK) \\
&=\pm \Bigl(
  (-1)^{p(F_1(Y))\ol{N}}\int^\Lambda d\Gamma (F_1(Y)a_\Gamma b)
  +(-1)^{p(F_1(Y))(p(a)+\ol{N})}\int^\Lambda d\Gamma(a_\Gamma F_1(Y)b),
  -F_1(Y)\Bigl(\int^\Lambda d\Gamma [a_\Gamma b]\Bigr)\Bigr). 
\end{align*}
Thus, $F_1(Y)$ belongs to $\Der(V,M)$ if and only if $Y\in\Ker \pdd^1$. 
Also, by the proof of \cref{prp:Hch0Cas}, we have
\begin{align*}
 F_1(\pdd^0 Y)(a) = -(-1)^{p(a)p(x)+\ol{N}}a_{-\nabla}x \quad (a \in V)
\end{align*}
for $Y\in C^0(V,M)$, where we represented $Y(1_\bbK) = \int x \in M/\nabla M$. Thus, for $Y \in C^1(V,M)$, $F_1(Y)$ is an inner derivation if and only in $Y\in\Img\pdd^0$. Hence $F_1$ induces the desired isomorphism. 
\end{proof}


Finally, we give the interpretation of the second cohomology. For that purpose, we define the extension of non-unital $N_W=N$ SUSY vertex algebras (\cref{dfn:NWext}). To define it, we need to introduce morphisms of non-unital $N_W=N$ SUSY vertex algebras. 
For non-unital $N_W=N$ SUSY vertex algebras $V$ and $W$, a \emph{morphism $V \to W$} is an even $\clH_W$-supermodule morphism $\varphi\colon V \to W$ satisfying
\begin{align*}
 \varphi\Bigl(\int^\Lambda d\Gamma[a_\Gamma b]\Bigr) = 
 \int^\Lambda d\Gamma[\varphi(a)_\Gamma \varphi(b)] 
\end{align*}
for any $a,b \in V$.
We see that non-unital $N_W=N$ SUSY vertex algebras and their morphisms form an abelian category. Here, an exact sequence means a diagram of non-unital $N_W=N$ SUSY vertex algebras
\[
 0 \lto U \xrr{\varphi} V \xrr{\psi} W \lto 0
\]
with $\varphi$ injective, $\psi$ surjective, and $\Ker \psi = \Img \varphi$.

\begin{dfn}\label{dfn:NWext}
Let $M$ be a $V$-module.
\begin{enumerate}
\item 
An exact sequence of non-unital $N_W=N$ SUSY vertex algebras
\begin{equation}\label{eq:coh:NWext}
 0 \lto M \xrr{\varphi} E \xrr{\psi} V \lto  0
\end{equation}
is called an extension of $V$ by $M$ if it satisfies 
\begin{align*}
 \varphi\Bigl(\int^\Lambda d\Gamma[\psi(a)_\Gamma x]\Bigr) = \int^\Lambda d\Gamma [a_\Gamma \varphi(x)]. 
\end{align*}
Here we regard $M$ as a non-unital $N_W=N$ SUSY vertex algebra with trivial operations.

\item
An extension of $V$ by module $M$ is called \emph{$\clH_W$-split} if the diagram \eqref{eq:coh:NWext} is a split exact sequence of $\clH_W$-supermodules. 
\end{enumerate}
\end{dfn}



\begin{lem}\label{lem:NWext}
For $V$ and $M$ as above, there exists a surjection
\begin{align}\label{eq:extsrj}
\begin{split}
&\left\{Y \in \MC\bigl(L(\oPchW{\Pi^{N+1}(V \oplus M)})\bigr)_{\od} 
 \;\middle|\;
 \begin{array}{l}
 (X-Y)(a \otimes b \otimes z_{1,2}^{-1})\in M_\nabla[\Lambda_1, \Lambda_2] \\
 (X-Y)(a \otimes x \otimes z_{1,2}^{-1})=0 \\
 Y(x \otimes y \otimes z_{1,2}^{-1})=0 \ (a,b \in V, \, x,y \in M)
 \end{array}
 \right\}\\
&\lto \{\text{isomorphism classes of $\clH_W$-split extexsions of $V$ by $M$}\}. 
\end{split}
\end{align}
Here an isomorphism of $\clH_W$-split extensions of $V$ by $M$ is defined to be an isomorphism of exact sequences of non-unital $N_W=N$ SUSY vertex algebras. 
\end{lem}

\begin{proof}
For $Y$ in the left hand side of \eqref{eq:extsrj}, let $E_Y$ denote the $\clH_W$-supermodule $V \oplus M$ with the non-unital $N_W=N$ SUSY vertex algebra structure corresponding to $Y$. Then we have an $\clH_W$-split extension
\[
 0 \lto M \xrr{\varphi} E_Y \xrr{\psi} V \lto 0, 
\]
where $\varphi\colon M \to V \oplus M$ is the canonical inclusion and $\psi\colon V \oplus M \to V$ is the canonical projection. The map $Y\mapsto E_Y$ gives rise to the surjection \eqref{eq:extsrj}. 
\end{proof}

Similarly as the Lie algebra cohomology $H^n_{\oLie}(L,M)$ recalled in \cref{ss:coh:CE}, and as the non-SUSY case \cite[Theorem 7.6]{BDHK}, the lower degree SUSY vertex algebra cohomology $H^n_{\ch}(V, M)$, $n=0,1,2$, have explicit description.

\begin{thm}\label{thm:W:coh}
Let $M$ be a module over a non-unital $N_W=N$ SUSY vertex algebra $V$. 
\begin{enumerate}
\item 
For the zeroth cohomology, we have an isomorphism 
\[
 H_{\ch}^0(\Pi^{N+1}V, \Pi^{N+1}M) \cong \Cas(V, M)
\]
of linear superspace, 
where $\Cas(V, M)$ is the linear superspace of Casimir elements (\cref{dfn:Cas}). 

\item 
For the first cohomology, we have an isomorphism
\begin{align*}
 H_{\ch}^1(\Pi^{N+1}V, \Pi^{N+1}M) \cong \Der(V, M)/\Ind(V, M)
\end{align*}
where $\Der(V, M)$ is the linear superspace of derivations and $\Ind(V, M)$ is the linear superspace of inner derivations (\cref{dfn:der}). 

\item 
For the odd part of the second cohomology, we have an linear isomorphism
\begin{align*}
 H_{\ch}^2(V,M)_{\od} \cong 
 \{\text{isomorphism classes of $\clH_W$-split extensions of $V$ by modules $M$}\}.
\end{align*}
\end{enumerate}
\end{thm}

\begin{proof}
We have already proved (1), (2) in \cref{prp:Hch0Cas} and \cref{prp:Hch1Der}. Let us show (3). 

Define a map $F_2'\colon C^2(V,M)_{\od} \to L^1\bigl(\oPchW{\Pi^{N+1}(V \oplus M)}\bigr)_{\od}$ by $F_2'(Y) \ceq X+Y$. If $Y \in (\Ker\pdd^2)_{\od}$, then $F_2'(Y)$ satisfies $F_2'(Y) \square F_2'(Y)=0$, so $F_2'$ gives a surjection from $(\Ker \pdd^2)_{\od}$ to 
\begin{align*}
 \left\{Y \in \MC\bigl(L(\oPchW{\Pi^{N+1}(V\oplus M)})\bigr)_{\od} 
 \;\middle|\;
 \begin{array}{l}
 (X-Y)(a\otimes b\otimes z_{1, 2}^{-1})\in M_\nabla[\Lambda_1,\Lambda_2] \\
 (X-Y)(a\otimes x\otimes z_{1, 2}^{-1})=0\\
 Y(x\otimes y\otimes z_{1, 2}^{-1})=0 \ (a,b \in V, \, x,y \in M)
 \end{array}
 \right\}.
\end{align*} 
Thus, the composition with the surjection in \cref{lem:NWext} gives a surjection
\begin{align*}
 F_2\colon (\Ker\pdd^2)_{\od} \lto 
 \{\text{isomorphism classes of $\clH_W$-split extexsions of $V$ by module $M$}\}. 
\end{align*}
Hence it is enough to show that $F_2(Y)=F_2(Y')$ is equivalent to $Y-Y'\in (\Img \pdd^1)_{\od}$ for $Y,Y' \in (\Ker\pdd^2)_{\od}$. 

Note that, for $Y, Y'\in(\Ker\pdd^2)_{\od}$, we have $F_2(Y)=F_2(Y')$ if and only if there exists a morphism of non-unital $N_W=N$ SUSY vertex algebras $f\colon E\to E'$ satisfying
\begin{align}\label{eq:exthom}
 f(a)-a \in M, \quad f(x)=x \quad (a \in V, \, x \in M). 
\end{align}
Here the symbol $E$ (resp.\ $E'$) denotes the $\clH_W$-supermodule $V\oplus M$ with the non-unital $N_W=N$ SUSY vertex algebra structure corresponding to $F_2'(Y)$ (resp.\ $F_2'(Y')$). For an even $\clH_W$-supermodule homomorphism $f\colon E\to E'$ satisfying \eqref{eq:exthom}, define an even linear map $\Phi_f\colon V\to M$ by $\Phi_f(a)\ceq f(a)-a$ for $a\in V$. Then the mapping $f \mapsto \Phi_f$ gives a bijection 
\begin{align*}
\Phi\colon\{f\in\Hom_{\clH_W}(E, E')_{\ev}\ \text{satisfying \eqref{eq:exthom}}\}
\lsto \Hom_{\clH_W}(V, M)_{\ev}. 
\end{align*}
For an even $\clH_W$-supermodule homomorphism $f\colon E\to E'$ satisfying \eqref{eq:exthom} and $a, b\in V$, we have
\begin{align*}
 f\Bigl(\int^\Lambda d\Gamma [a_\Gamma b]_E\Bigr)
 &=\Phi_f\Bigl(\int^\Lambda d\Gamma [a_\Gamma b]_E\Bigr)+\int^\Lambda d\Gamma [a_\Gamma b]_E, \\
 \int^\Lambda d\Gamma [f(a)_\Gamma f(b)]
 &=\int^\Lambda d\Gamma (\Phi_f(a)_\Gamma b)+\int^\Lambda d\Gamma (a_\Gamma \Phi_f(b))
 +\int^\Lambda d\Gamma [a_\Gamma b]_{E'}.
\end{align*}
Thus, $f\colon E \to E'$ is a morphism of non-unital $N_W=N$ SUSY vertex algebras if and only if
\begin{align*}
\int^\Lambda d\Gamma [a_\Gamma b]_E-\int^\Lambda d\Gamma [a_\Gamma b]_{E'}
=\int^\Lambda d\Gamma (\Phi_f(a)_\Gamma b)+\int^\Lambda d\Gamma (a_\Gamma \Phi_f(b))
-\Phi_f\Bigl(\int^\Lambda d\Gamma [a_\Gamma b]_E\Bigr)
\end{align*}
for any $a, b\in V$, which is equivalent to $Y-Y'=\pdd^1(F_1^{-1}(\Phi_f))$. Here $F_1$ is the isomorphism $C^1(V, M) \sto \Hom_{\clH_W}(V, M)$ defined in the proof of \cref{prp:Hch1Der}. Therefore we find that $F_2(Y)=F_2(Y')$ is equivalent to $Y-Y'\in (\Img\pdd^1)_{\od}$ for $Y,Y'\in(\Ker\pdd^2)_{\od}$.
\end{proof}

\subsection{Cohomology of $N_K=N$ SUSY vertex algebras}\label{ss:coh:K}

The theory of $N_W=N$ SUSY vertex algebra cohomology in \cref{ss:coh:W} carries over to the $N_K=N$ case with minor modifications. So we only give the main definitions and statements.

\begin{dfn}
Let $(V,\nabla,[\cdot_\Lambda\cdot])$ be an $N_K=N$ SUSY Lie conformal algebra, $(M,\nabla)$ be a left $\clH_K$-supermodule and $\rho_\Lambda\colon V\otimes M\to M[\Lambda]$ be a linear map of parity $\ol{N}$. We denote $a_\Lambda x \ceq \rho_\Lambda(a \otimes x)$ for $a \in V$, $x \in M$. 
A triple $(M,\nabla,\rho_\Lambda)$ is called a \emph{module over $(V,\nabla,[\cdot_\Lambda \cdot])$} if it satisfies the conditions (i) and (ii) in \cref{dfn:W:LCAmod} replacing the $(1|N)_W$-supervariables $\Lambda,\Lambda_1,\Lambda_2$ by the corresponding $(1|N)_K$-supervariables. 
\end{dfn}

\begin{dfn}
Let $(V,\nabla,[\cdot_\Lambda\cdot],\mu)$ be a non-unital $N_K=N$ SUSY vertex algebra, $(M,\nabla,\rho_\Lambda)$ be a module over the $N_K=N$ Lie conformal algebra $(V,\nabla,[\cdot_\Lambda\cdot])$ and $\rho\colon V \otimes M \to M$ be an even linear map. We denote $a \cdot x \ceq \rho(a \otimes x)$. A tuple $(M,\nabla,\rho_\Lambda, \rho)$ is called a \emph{module over $(V,\nabla,[\cdot_\Lambda\cdot],\mu)$} if it satisfies the conditions (i)--(iii) in \cref{dfn:W:VAmod} replacing the $(1|N)_W$-supervariables $\Lambda,\Gamma$ by the corresponding $(1|N)_K$-supervariables. 
\end{dfn}

For a module $M$ over a non-unital $N_K=N$ SUSY vertex algebra $V$, one can define the cochain complex of $V$ with coefficients in $M$ as in \eqref{eq:W:comp}. Also, the notions of Casimir elements, derivations and extensions are defined in the similar way. Then we find that \cref{thm:W:coh} is valid for a non-unital $N_K=N$ SUSY vertex algebra $V$ and its module $M$.

\subsection{Concluding remarks}\label{ss:cnc}

As mentioned in \cref{s:0}, the paper \cite{BDHK} continues to the development of the cohomology theory of vertex algebras in \cite{BDHK2,BDK20,BDK21,BDKV21}. One of the main topics of these works is the relationship between the cohomology theory of vertex algebras and that of vertex Poisson algebras. As far as we understand, the status of the study of this relationship is still complicated, although some fundamental results have been established. 

It is therefore natural to consider that the next step in the study of our cohomology theory for SUSY vertex algebras is to study the operad theory and the cohomology theory of SUSY vertex Poisson algebras. We leave this to a future work \cite{NY2}.


\end{document}